%% file: main.tex
\documentclass{amsart}

\newcommand\HS{\mathrm{HS}}
\newcommand\triv{\mathrm{triv}}
\newcommand\lSel{\mathfrak S}

\newcommand\ellp{p}
\newcommand\pell{\ell}

\newcommand\loc{\mathrm{loc}}
\newcommand\Col{\mathrm{Col}}
\newcommand\ACol{A_\Col}
\newcommand\ACollog{A_\Col^{\alg}}
\newcommand\glob{\mathrm{glob}}
\newcommand\BD{\mathrm{BD}}
\newcommand\BK{\mathrm{BK}}
\newcommand\jj{j}

\newcommand\Dvsr{\mathcal D}

\newcommand\z{\tau} 

\newcommand\Kbar{\bar K}
\newcommand\Obar{\bar\O}
\newcommand\kbar{\bar k}

\newcommand\Res{\mathrm{Res}}
\newcommand\BCH{\mathrm{BCH}}
\newcommand\E{\mathcal E}
\newcommand\Isoc{\mathbf{Isoc}}
\newcommand\ISO{\underline{\mathrm{Iso}}}
\newcommand\un{\mathrm{un}}
\newcommand\Mod{\mathbf{Mod}}
\newcommand\MIC{\mathbf{MIC}}

\newcommand\DO{\mathfrak D}
\renewcommand\div{\mathrm{div}}
\newcommand\PD{\mathrm{PD}}
\newcommand\supp{\mathrm{supp}}
\newcommand\Mero{\mathcal M}
\newcommand\llpara{(\!(}
\newcommand\rrpara{)\!)}
\newcommand\order{\mathrm{order}}
\newcommand\m{\mathfrak m}

\newcommand\jota{j}
\newcommand\Cpts{\mathcal C}
\newcommand\Fd{F}
\newcommand\sm{{\mathrm{sm}}}

\newcommand\Qbar{{\overline\Q}}

\usepackage{hyperref} 
\hypersetup{
	colorlinks=true,
	urlcolor= blue,  
	linkcolor= blue,
    citecolor=blue,  
	bookmarksopen=true,           
}
\usepackage{url}[lowtilde]

\include{macros}

\usepackage{morefloats}
\usepackage{imakeidx}
\usepackage{enumitem}

\title[Weight filtrations and effective Chabauty--Kim]{Weight filtrations on Selmer schemes and the effective Chabauty--Kim method}
\author{L.\ Alexander Betts}
\address{Harvard University, Department of Mathematics. Science Center Room 325, 1 Oxford Street, Cambridge, MA 02138, USA.}
\email{abetts@math.harvard.edu}
\date\today

\begin{document}

\begin{abstract}
We develop an effective version of the Chabauty--Kim method which gives explicit upper bounds on the number of $S$-integral points on a hyperbolic curve in terms of dimensions of certain Bloch--Kato Selmer groups. Using this, we give a new ``motivic'' proof that the number of solutions to the $S$-unit equation is bounded uniformly in terms of~$\#S$.
\end{abstract}

\maketitle
\tableofcontents

\input{introduction}
\input{weights}
\input{coleman}
\input{effective}
\input{siegel}

\bibliographystyle{alpha}

\bibliography{references}

\end{document}

%% file: macros.tex
\usepackage{amsmath,amsfonts,amssymb,amsthm,enumitem,tikz-cd,graphicx,color,verbatim,xspace,bbm}

		\input cyracc.def

	\newcommand\Sym{\mathrm{Sym}}
	
	\newcommand\Lie{\mathrm{Lie}}

	\newcommand\ab{\mathrm{ab}}
	
	\newcommand\dual{*}

	\newcommand\ev{\mathrm{ev}}

	\def\Spec{\mathrm{Spec}}
	
	\renewcommand\O{\mathcal O}
	\def\Y{\mathcal Y}
	\def\X{\mathcal X}

	\def\A{\mathbb{A}}
	\renewcommand\P{\mathbb{P}}

	\def\Zar{\mathrm{Zar}}
	\def\Sel{\mathrm{Sel}}

\renewcommand\d{\mathrm d}

\newcommand\bigwedgesquare{\mathchoice{\bigwedge\nolimits^{\!\!2}}{\bigwedge\nolimits^{\!2}}{\bigwedge\nolimits^2}{\bigwedge\nolimits^2}}

\newcommand{\et}{\mathrm{\acute{e}t}}

\newcommand{\fin}{\mathrm{fin}}

\newcommand{\red}{\mathrm{red}}

\DeclareMathOperator{\rk}{\mathrm{rk}}

	\newcommand\Disc{\mathbb D}

	\newcommand{\HH}{\mathrm H}

	\newcommand\Cycle{\mathrm Z}
	\newcommand\Chain{\mathrm C}
	\def\loc{\mathrm{loc}}

    \newcommand{\Fil}{\mathrm{F}}
	\def\gr{\mathrm{gr}}
	
	\newcommand\W{W}

          \newcommand{\an}{\mathrm{an}}
          
          \newcommand{\dR}{\mathrm{dR}}
          \newcommand{\cris}{\mathrm{cris}}

\def\N{\mathbb N}
\def\Z{\mathbb Z}

\def\Q{\mathbb Q}

\def\C{\mathbb C}
\def\F{\mathbb F}

\def\B{\mathsf B}
\def\D{\mathsf D}
\def\nr{\mathrm{nr}}

\def\llbrack{\lbrack\!\lbrack}
\def\rrbrack{\rbrack\!\rbrack}
\def\alg{\mathrm{alg}}

\newcommand\SET{\mathbf{Set}}
\newcommand\AFF{\mathbf{Aff}}
\newcommand\WAFF{\W_\bullet\AFF}

\newcommand\AB{\mathbf{Ab}}
\newcommand\ALG{\mathbf{Alg}}
\newcommand\WALG{\W_\bullet\ALG}

\def\Hom{\mathrm{Hom}}
\def\HOM{\underline{\Hom}}
\def\Map{\mathrm{Map}}
\def\cts{\mathrm{cts}}
\def\Aut{\mathrm{Aut}}
\def\AUT{\underline{\Aut}}

\def\Ext{\mathrm{Ext}}

\def\End{\mathrm{End}}

\def\coker{\mathrm{coker}}

\numberwithin{equation}{subsection}
\theoremstyle{plain}
\newcounter{theoremcount}[subsection]
\makeatletter
\@addtoreset{theoremcount}{section}
\@addtoreset{equation}{section}
\@addtoreset{equation}{subsection}
\makeatother

\newtheorem{theorem}[theoremcount]{Theorem}
\newtheorem*{theorem*}{Theorem}
\newtheorem{theorema}{Theorem}

\newtheorem{lemma}[theoremcount]{Lemma}
\newtheorem{proposition}[theoremcount]{Proposition}
\newtheorem*{proposition*}{Proposition}
\newtheorem{corollary}[theoremcount]{Corollary}
\newtheorem*{corollary*}{Corollary}

\theoremstyle{definition}
\newtheorem{definition}[theoremcount]{Definition}
\newtheorem{definition-lemma}[theoremcount]{Definition--Lemma}
\newtheorem{example}[theoremcount]{Example}
\newtheorem{remark}[theoremcount]{Remark}
\newtheorem{notation}[theoremcount]{Notation}

%% file: introduction.tex
\section{Introduction}

One of the crowning achievements of 20th century number theory was the resolution of the Mordell Conjecture, showing that the set $X(K)$ of $K$-rational points on a smooth projective geometrically connected curve $X$ of genus $g\geq2$ over a number field $K$ is finite. Since then, a major open question is whether the Mordell Conjecture can be made effective, by which is meant any of the following.
\begin{enumerate}
	\item\label{pt:algorithmic} Giving an algorithm to compute $X(K)$.
	\item Giving a computable upper bound on the height of $K$-rational points of~$X$.
	\item\label{pt:bound} Giving a computable upper bound on $\# X(K)$.
\end{enumerate}


In this paper, we give a complete resolution of~\eqref{pt:bound} in the case $K=\Q$ subject to the Bloch--Kato Conjecture. We do this using a strengthening of Minhyong Kim's method of non-abelian Chabauty which systematically incorporates weight filtrations on Selmer schemes. The weight-filtered Chabauty--Kim method we develop in this paper is always effective in the sense of~\eqref{pt:bound}, and also yields some algorithmic consequences in the direction of point~\eqref{pt:algorithmic}.
\smallskip

The input in the usual Chabauty--Kim method is a finite-dimensional $G_\Q$-equivariant quotient $U$ of the $\Q_p$-pro-unipotent \'etale fundamental group~$U^\et$ of $X_{\Qbar}$ based at a point $b\in X(\Q)$, where $p$ is a prime of good reduction for $X$. If we write $(V_n)_{n\geq1}$ for the graded pieces of the descending central series, then the Chabauty--Kim method shows that when the inequality
\begin{equation}\label{eq:c-k_ineq}
	\sum_n\dim\HH^1_f(G_\Q,V_n) < \sum_n\dim\HH^1_f(G_p,V_n)
\end{equation}
between the dimensions of the global and local Bloch--Kato Selmer groups of the $V_n$ holds, then there is a non-zero Coleman analytic function $f$ on $X_{\Q_p}$ which vanishes on $X(\Q)$. In particular, \eqref{eq:c-k_ineq} implies finiteness of the set $X(\Q)$.
\smallskip

Much of the interest in the Chabauty--Kim method derives from the fact that it has been made effective for specific quotients~$U$. When the $p^\infty$-Selmer rank of the Jacobian~$J$ of~$X$ is strictly less than the genus, then inequality~\eqref{eq:c-k_ineq} holds for~$U=(U^\et)^\ab$ the abelianisation of the fundamental group, and the Coleman analytic function~$f$ produced by the Chabauty--Kim method is the Coleman integral of a holomorphic $1$-form~$\omega$ on~$X_{\Q_p}$. So the Chabauty--Kim method in this case recovers the usual Chabauty method. Given enough independent points in $J(\Q)$, the $1$-form~$\omega$ and hence the function~$f$ can be computed to any desired $p$-adic precision, and this allows one in some cases to compute $X(\Q)$ exactly \cite{poonen-mccallum:chabauty,flynn-poonen-schaefer:quadratic_cycles,poonen:preperiodic_points}. In general, even without finding~$f$, Coleman showed that it is possible to bound $\#X(\Q)$ in terms of $g$ and the number of $\F_p$-points on~$X$ \cite{coleman:effective}.

When instead the $p^\infty$-Selmer rank of the Jacobian~$J$ is equal to the genus, and the rational N\'eron--Severi rank of~$J$ is at least~$2$, then there is a quotient~$U$ of $U^\et$ which is a central extension of the abelianisation by $\Q_p(1)$, and inequality~\eqref{eq:c-k_ineq} holds for this quotient. The Coleman analytic function~$f$ produced in this case can be expressed as a double Coleman integral of meromorphic $1$-forms on~$X_{\Q_p}$. Again, if~$X$ has sufficiently many rational points, then the function~$f$ can be determined to any desired precision, and this can be used to compute $X(\Q)$ in several cases \cite{jen-netan:quadratic_i,jen-etal:split_modular_curve}. And even without computing~$f$, knowing its general form can be used to give an explicit bound on $\#X(\Q)$ \cite{jen-netan:effective}.
\smallskip

However, for a general quotient $U$, the Chabauty--Kim method does not afford any control over the function $f$, and so tells us nothing about~$X(\Q)$ beyond its finiteness. This is what we address in this paper. In order to state our main theorem, let us say that two rational points $x,y\in X(\Q)$ have the same \emph{reduction type} just when they reduce onto the same component of the mod-$\ell$ special fibre of the minimal regular model of~$X$ for all primes~$\ell$. We write $X(\Q)_\Sigma$ for the set of all rational points of a given reduction type~$\Sigma$. These sets form a finite partition of the set $X(\Q)$.

\begin{theorema}\label{thm:main}
	Let $(c_i^\glob)_{i\geq0}$ and $(c_i^\loc)_{i\geq0}$ be the coefficients of the power series
	\[
	\HS_\glob(t) := \prod_{n\geq1}^\infty(1-t^n)^{-\dim\HH^1_f(G_\Q,V_n)}
	\hspace{0.3cm}\text{and}\hspace{0.3cm}
	\HS_\loc(t) := \prod_{n\geq1}^\infty(1-t^n)^{-\dim\HH^1_f(G_p,V_n)} \,,
	\]
	which are non-negative integers. Suppose that~$m$ is a positive integer such that the inequality
	\begin{equation}\label{eq:c-k_ineq_series}
		\sum_{i=0}^mc_i^\glob < \sum_{i=0}^mc_i^\loc
	\end{equation}
	holds. Then for every reduction type $\Sigma$, the set $X(\Q)_\Sigma$ is contained in the vanishing locus of a non-zero Coleman analytic function~$f$ of weight at most $m$.
\end{theorema}

The advantage of this theorem over the more usual set-up in the Chabauty--Kim method is that it offers some \emph{a priori} control on the Coleman analytic function~$f$, in the form of a bound on its weight. We will explain in \S\ref{sss:coleman_weights} what is meant by the weight of a Coleman analytic function, but in the context of effectivity questions, the significance of this notion of weight is twofold.
\begin{enumerate}
	\item For all $m\geq0$, the space of Coleman analytic functions of weight $\leq m$ is finite-dimensional. (Corollary~\ref{cor:coleman_hilbert_series}.)
	\item For all $m\geq0$, there is an explicit uniform bound on the number of zeroes of a non-zero Coleman analytic function of weight $\leq m$. (Theorem~\ref{thm:weight_bound}.)
\end{enumerate}
In particular, Theorem~\ref{thm:main} gives us effective Chabauty--Kim results for a general quotient~$U$. The first of these is an explicit upper bound on the number of rational points on any curve~$X$ to which the method applies, generalising the bounds of Coleman and Balakrishnan--Dogra.

\begin{theorema}\label{thm:main_bound}
	Suppose that $m$ is a positive integer such that inequality~\eqref{eq:c-k_ineq_series} holds. Then
	\[
	\# X(\Q) \leq \kappa_p\cdot\prod_\ell n_\ell\cdot\#X(\F_p)\cdot(4g-2)^m\cdot\prod_{i=1}^{m-1}(c_i^\loc+1) \,,
	\]
	where $n_\ell$ denotes the number of components of the mod-$\ell$ special fibre of the minimal regular model of~$X$, $\#X(\F_p)$ denotes the number of $\F_p$-points on the mod-$p$ special fibre of the minimal regular model, and the positive constant $\kappa_p$ is defined by
	\[
	\kappa_p := 
	\begin{cases}
		1+\frac{p-1}{(p-2)\log(p)} & \text{if $p\neq2$,} \\
		2+\frac2{\log(2)} & \text{if $p=2$.}
	\end{cases}
	\]
\end{theorema}

It has been a long-standing expectation in the Chabauty--Kim community that if the inequality~\eqref{eq:c-k_ineq} holds for $U=U^\et_n$ the maximal $n$-step unipotent quotient of the $\Q_p$-pro-unipotent \'etale fundamental group, then it should be possible to bound the number of rational points of~$X$ in terms of~$n$. Theorem~\ref{thm:main_bound} realises this expectation by giving the following coarse upper bound on $\#X(\Q)$ in terms of~$n$.

\begin{corollary*}[to Theorem~\ref{thm:main_bound}]\label{cor:depth_bound}
	Suppose that inequality~\eqref{eq:c-k_ineq} holds for~$U=U^\et_n$ the maximal $n$-step unipotent quotient of the $\Q_p$-pro-unipotent \'etale fundamental group where~$n\geq2$. Then
	\[
	\# X(\Q) \leq \kappa_p\cdot\prod_\ell n_\ell\cdot\#X(\F_p)\cdot(4g-2)^{n^{(2g)^n}}\cdot(2g)^{{{n^{(2g)^n}}\choose2}} \,,
	\]
	where $n_\ell$, $\#X(\F_p)$ and $\kappa_p$ are as in Theorem~\ref{thm:main_bound}.
\end{corollary*}

For any~$X/\Q$, inequality~\eqref{eq:c-k_ineq} is known to hold for some computable $n\gg0$ if one assumes the Bloch--Kato Conjecture \cite[\S3]{minhyong:selmer}. So the above Corollary gives in particular an effective upper bound on~$\#X(\Q)$ for any~$X$, subject to the Bloch--Kato Conjecture. Moreover, in several cases, computable upper bounds on~$n$ can be computed without assuming Bloch--Kato e.g.\ \cite{minhyong-coates:cm_jacobians}. Explicit values for these upper bounds on~$\#X(\Q)$ will be given in future work.
\smallskip

The other consequence of Theorem~\ref{thm:main} in the context of effective Chabauty--Kim is that it constrains the Coleman analytic function~$f$ to lie in a finite-dimensional space, and so when $X$ has sufficiently many rational points, it is possible to find~$f$ using the method of undetermined coefficients, i.e.\ using the equations $f(x_i)=0$ for some points $x_i\in X(\Q)$ to solve for~$f$. To state this formally, we adopt the shorthand $C_m^\glob:=\sum_{i=0}^mc_i^\glob$ and $C_m^\loc:=\sum_{i=0}^mc_i^\loc$.

\begin{theorema}\label{thm:main_det}
	Suppose that $m$ is a positive integer such that \eqref{eq:c-k_ineq_series} holds. Let $f_1,\dots,f_{C_m^\loc}$ be a basis of the space of Coleman analytic functions of weight $\leq m$ associated to the quotient $U$ in the sense of \S\ref{ss:proof_general}. Then for any $C_m^\glob+1$ points $x_0,x_1,\dots,x_{C_m^\glob}\in X(\Q)_\Sigma$ of the same reduction type~$\Sigma$, all $(C_m^\glob+1)\times(C_m^\glob+1)$ minors of the matrix~$M$ with entries $M_{ij}=f_i(x_j)$ vanish.
\end{theorema}

In particular, fixing the points $x_1,\dots,x_{C_m^\glob}$, the minors of the matrix~$M$ provide Coleman analytic functions in the variable $x_0$ which vanish on $X(\Q)_\Sigma$. There is no guarantee that the functions obtained in this way are non-zero (e.g.\ if $X$ has fewer than $C_m^\glob$ rational points overall), but in the cases when it is non-zero, it gives an explicit Coleman analytic function on~$X_{\Q_p}$ vanishing on~$X(\Q)$, computable to any desired $p$-adic precision.

\subsubsection*{Remarks}

1) The Chabauty--Kim method applies not just to rational points on smooth projective curves $X$ of genus $\geq2$, but more generally to $S$-integral points on regular models~$\Y/\Z_S$ of smooth hyperbolic curves~$Y/\Q$, for $S$ any finite set of primes. We will in fact prove all of our main results in this greater level of generality: see Theorem~\ref{thm:main_refined}. Theorems~\ref{thm:main}, \ref{thm:main_bound} and~\ref{thm:main_det} are recovered from these more general results by taking $Y=X$, $S=\emptyset$, and $\Y/\Z$ the minimal regular model of~$X$. In dealing with affine curves, an extra technicality arises in that the space of Coleman analytic functions of weight~$\leq m$ is infinite-dimensional, so we cannot hope to bound the number of zeroes of a Coleman analytic function in terms of its weight. To resolve this issue, we introduce in \S\ref{sss:log-coleman} a subspace, which we call the space of Coleman \emph{algebraic} functions, which agrees with the space of Coleman analytic functions in the projective case, but is better behaved in the affine case.

2) Although the Chabauty--Kim method is only usually applied to quotients~$U$ that are finite-dimensional, our Theorems~\ref{thm:main}, \ref{thm:main_bound} and~\ref{thm:main_det} also apply in the case that $U$ is infinite-dimensional. We will see in \S\ref{s:siegel} that this extra flexibility can sometimes be useful in calculations.

3) The factor of $(4g-2)$ appearing in Theorem~\ref{thm:main_bound} is not optimal. It can always be improved to $(4g-3)$, and in some special cases yet smaller values suffice. See \S\ref{ss:better_omega}.

4) When $U$ dominates the abelianisation of the fundamental group, the first exponent $\HH^1_f(G_\Q,V_1)$ appearing in $\HS_\glob(t)$ is the $p^\infty$-Selmer rank of~$J$. The $p^\infty$-Selmer rank is greater than or equal to the Mordell--Weil rank of $J(\Q)$, with equality if and only if the $p$-divisible part of the Tate--Shafarevich group of~$J$ vanishes, as predicted by the Tate--Shafarevich Conjecture. Using a trick of Balakrishnan--Dogra (see \S\ref{ss:jen-netan}), one can circumvent the apparent need to assume the Tate--Shafarevich Conjecture in order to compute $\HS_\glob(t)$: Theorems~\ref{thm:main}, \ref{thm:main_bound} and~\ref{thm:main_det} still hold verbatim when the power series $\HS_\glob(t)$ is replaced by the power series
\[
\HS_\glob^\BD(t):=(1-t)^{-\rk(J(\Q))}\cdot\prod_{n=2}^\infty(1-t^n)^{-\dim\HH^1_f(G_\Q,V_n)} \,.
\]

5) There are also versions of Theorems~\ref{thm:main} and~\ref{thm:main_det} which apply to the whole set $X(\Q)$, rather than the subsets $X(\Q)_\Sigma$ individually: see Theorem~\ref{thm:main_naive}. The price one pays is that the weight of the Coleman analytic function vanishing on the whole of $X(\Q)$ is in general larger than the weights of the Coleman analytic functions vanishing on each $X(\Q)_\Sigma$ separately. This weight discrepancy means that bounding the sizes of the sets $X(\Q)_\Sigma$ individually usually gives better bounds than bounding the size of $X(\Q)$ directly. These versions of Theorems~\ref{thm:main} and~\ref{thm:main_det} for $X(\Q)$ were already obtained in unpublished work of Francis Brown \cite{brown:integral_points}, using related ideas to those we develop in the first half of this paper.

\subsection{Application: $S$-uniformity in Siegel's Theorem}

One feature of the version of Theorem~\ref{thm:main_bound} we prove for $S$-integral points on affine curves is that the bound depends only weakly on the set~$S$ of primes. In particular, it always provides upper bounds which are \emph{$S$-uniform}, meaning that they depend only on the number of primes in~$S$, and not on the set~$S$ itself. For example, our effective Chabauty--Kim method proves the following $S$-uniform upper bound on the number of solutions to the $S$-unit equation.

\begin{theorema}\label{thm:main_siegel}
	Let $\Y:=\P^1\setminus\{0,1,\infty\}$, and let~$s$ be a non-negative integer. Then for every set~$S$ of primes of size~$s$, we have
	\[
	\#\Y(\Z_S)\leq8\cdot6^s\cdot2^{4^s} \,.
	\]
\end{theorema}

$S$-uniform bounds of this kind are already known via analytic techniques. The best available upper bound is $\#\Y(\Z_S)\leq3\cdot 7^{2s+3}$, due to Evertse \cite[Theorem~1]{evertse:equations_in_s-units}. Our bound in Theorem~\ref{thm:main_siegel}, being doubly exponential in~$s$, is much weaker than Evertse's; what is of interest here is not so much the specific bound than the method used to obtain it, demonstrating the applicability of the Chabauty--Kim programme to uniformity questions. Tantalisingly, it suggests that it may be possible to use Theorem~\ref{thm:main_bound} to prove uniformity results of a similar type to  \cite{katz-rabinoff-zureick-brown:uniform,dimitrov-gao-habegger:uniformity}.

\subsection{Examples}\label{ss:examples}

As an illustration of the kinds of explicit results one can obtain from Theorem~\ref{thm:main_bound}, we give a few examples showing that it recovers as special cases both the effective Chabauty results of Coleman and the effective quadratic Chabauty results of Balakrishnan--Dogra, as well as giving a new example from quadratic Chabauty which goes beyond \cite{jen-netan:effective}.

Throughout these examples, we assume the Tate--Shafarevich Conjecture, which ensures that
\[
\dim_{\Q_\ellp}\HH^1_f(G_\Q,\HH^1_\et(X_{\Qbar},\Q_\ellp)^\dual) = \rk(J(\Q))
\]
for every smooth projective curve~$X/\Q$ with Jacobian~$J$. But all the bounds in the below examples are still true without this assumption; see \S\ref{ss:jen-netan}.

\begin{example}
	Suppose that the rank of the Jacobian~$J$ of~$X$ is strictly less than the genus~$g$. We take~$U=U_1:=(U^\et)^\ab$. Then the representations~$V_n$ are zero for~$n>1$, and $V_1$ is the $\Q_\ellp$-linear Tate module of the Jacobian~$J$ of~$X$. The power series $\HS_\glob(t)$ and~$\HS_\loc(t)$ are given by
	\begin{align*}
		\HS_\glob(t) &= (1-t)^{-\rk J(\Q)} = 1+ \rk(J(\Q)) t + \dots \\
		\HS_\loc(t) &= (1-t)^{-g} = 1 + gt + \dots \,.
	\end{align*}
	Thus the assumption that $\rk(J(\Q))<g$ implies that inequality~\eqref{eq:c-k_ineq_series} holds for~$m=1$, so Theorem~\ref{thm:main_bound} provides the bound
	\[
	\# X(\Q) \leq \kappa_\ellp\cdot\prod_{\pell}n_\pell\cdot\# X(\F_\ellp)\cdot(4g-2)
	\]
	whenever~$\ellp$ is a prime of good reduction for~$X$. This bound is worse than the upper bound of $\#X(\F_p)+2g-2$ proved by Coleman in the case $p>2g$ \cite[Corollary~4a \&~\S1]{coleman:effective}, and is indicative of the extra complications which arise when dealing with genuinely non-abelian quotients of the fundamental group.
\end{example}

\begin{example}\label{ex:quadratic_chabauty}
	Suppose that the rank of the Jacobian~$J$ of~$X$ is equal to the genus~$g$, and suppose that the rational N\'eron--Severi rank of~$J$ is strictly bigger than~$1$. Then according to \cite[Proof of Lemma~3.2]{jen-netan:quadratic_i}, the $\Q_\ellp$-pro-unipotent \'etale fundamental group of~$X$ has a quotient~$U$ which is a central extension of the abelianisation of~$U^\et$ by~$\Q_\ellp(1)$. For this quotient~$U$, we have
	\begin{align*}
		\HS_\glob(t) &= (1-t)^{-\rk J(\Q)} = 1+ gt + \frac{g(g+1)}2t^2 + \dots \\
		\HS_\loc(t) &= (1-t)^{-g}(1-t^2)^{-1} = 1 + gt + \left(\frac{g(g+1)}2+1\right)t^2 + \dots \,.
	\end{align*}
	Inequality~\eqref{eq:c-k_ineq_series} holds for~$m=2$, so Theorem~\ref{thm:main_bound} provides the bound
	\[
	\# X(\Q) \leq \kappa_\ellp\cdot\prod_{\pell}n_\pell\cdot\# X(\F_\ellp)\cdot(16g^3-12g+4)
	\]
	whenever~$\ellp$ is a prime of good reduction for~$X$. This is actually slightly better than the upper bound of $\kappa_p\cdot\prod_\ell n_\ell\cdot\#X(\F_p)\cdot(16g^3+15g^2-16g+10)$ obtained by Balakrishnan--Dogra in the case that $p$ is odd \cite[Theorem~1.1(i)]{jen-netan:effective}. We will explain this in Remark~\ref{rmk:why_better}.
\end{example}

\begin{remark}
	Strictly speaking, the constants~$n_\pell$ in our Theorem~\ref{thm:main_bound} are not the same as the constants denoted~$n_\pell$ in \cite{jen-netan:effective}, which are instead defined as the size of the image of $X(\Q_\pell)$ under the non-abelian Kummer map $\jj_{\pell,U}$ at~$\pell$ \cite[\S2]{jen-netan:effective}. We always have $\#\jj_{\pell,U}(X(\Q_\pell))\leq n_\pell$, and in fact Theorem~\ref{thm:main_bound} still holds if we replace the constants~$n_\pell$ with the smaller constants~$\#\jj_{\pell,U}(X(\Q_\pell))$ (see Remarks~\ref{rmk:local_images_outside_s} and~\ref{rmk:slightly_better_bound}).
\end{remark}


We conclude with another example from quadratic Chabauty which illustrates that the bounds from Theorem~\ref{thm:main_bound} can be quite large, even in relatively simple situations.

\begin{example}
	Suppose that the rank of the Jacobian~$J$ of~$X$ is equal to~$g+1$, and suppose that the rational N\'eron--Severi rank of~$J$ is strictly bigger than~$2$. By \cite[Proof of Lemma~3.2]{jen-netan:quadratic_i} again, the $\Q_\ellp$-pro-unipotent \'etale fundamental group of~$X$ has a quotient~$U$ which is a central extension of the abelianisation of~$U^\et$ by~$\Q_\ellp(1)^2$. To find a value of~$m$ for which inequality~\eqref{eq:c-k_ineq_series} holds, we employ a computational trick which we will use at several points in this paper, considering instead the more manageable series
	\begin{align*}
		\frac{(1+t)^2}{1-t}\HS_\glob(t) &\leq 4(1-t)^{-g-2} = 4\cdot\sum_{i=0}^\infty{i+g+1\choose g+1}t^i \,, \\
		\frac{(1+t)^2}{1-t}\HS_\loc(t) &= (1-t)^{-g-3} = \sum_{i=0}^\infty{i+g+2\choose g+2}t^i \,,
	\end{align*}
	where~$\leq$ denotes componentwise inequality of power series. One checks straightforwardly that the coefficient of~$t^{3g+7}$ in $\frac{(1+t)^2}{1-t}\HS_\glob(t)$ is strictly less than the corresponding coefficient in $\frac{(1+t)^2}{1-t}\HS_\loc(t)$. This implies that there must be some~$m\leq3g+7$ such that the $m$th coefficient of $\frac1{1-t}\HS_\glob(t)$ is strictly less than the corresponding coefficient in $\frac1{1-t}\HS_\loc(t)$ (in fact, one of $m=3g+7,3g+6,3g+5$ works). Since the coefficients of $\frac1{1-t}\HS_{?}(t)$ are the partial sums of the coefficients of $\HS_{?}(t)$, we thus know that~\eqref{eq:c-k_ineq_series} holds for some~$m\leq3g+7$.
	
	Combined with the elementary estimate $c_i^\loc+1\leq{{i+g+1}\choose{g+1}}$ for~$i\geq1$, Theorem~\ref{thm:main_bound} gives the bound
	\[
	\# X(\Q) \leq \kappa_\ellp\cdot\prod_\pell n_\pell\cdot\# X(\F_\ellp)\cdot(4g-2)^{3g+7}\cdot\prod_{i=1}^{3g+6}{{i+g+1}\choose{g+1}}
	\]
	whenever~$\ellp$ is a prime of good reduction for~$X$. Note that, despite the superficial similarities with Example~\ref{ex:quadratic_chabauty}, the dependence of this bound on~$g$ is much worse, being super-exponential as opposed to cubic. It seems likely that the above bound is very far from sharp.
\end{example}


\subsection{Outline of the method}

To give an idea of the techniques developed in this paper, let us sketch how the effective Chabauty--Kim method we develop here differs from the original one from \cite{minhyong:siegel,minhyong:selmer,jen-etal:non-abelian_tate-shafarevich,me-netan:ramification}. The usual Chabauty--Kim method studies the rational points of $X/\Q$ inside the $\ellp$-adic points via a certain commuting square
\begin{center}
	\begin{tikzcd}
		X(\Q) \arrow[r,hook]\arrow[d,"\jj_U"] & X(\Q_p) \arrow[d,"\jj_{\ellp,U}"] \\
		\Sel_U(\Q_\ellp) \arrow[r,"\loc_\ellp"] & \HH^1_f(G_\ellp,U)(\Q_\ellp) \,.
	\end{tikzcd}
\end{center}
Both sets on the bottom row are the $\Q_\ellp$-points of affine $\Q_\ellp$-schemes ($\Sel_U$ being the Selmer scheme as defined in \cite[p.~371]{jen-etal:non-abelian_tate-shafarevich}), and the localisation map $\loc_\ellp$ is a morphism of $\Q_\ellp$-schemes. By comparison with crystalline fundamental groups, one shows that the local non-abelian Kummer map $\jj_{\ellp,U}$ is a Coleman analytic map with Zariski-dense image, whose coordinates are given by iterated Coleman integrals.

When inequality~\eqref{eq:c-k_ineq} holds, we have $\dim\Sel_U<\dim\HH^1_f(G_\ellp,U)$, from which one deduces that the localisation map $\loc_\ellp$ is non-dominant. Thus, there is a non-zero algebraic functional $\alpha\colon\HH^1_f(G_\ellp,U)\rightarrow\A^1_{\Q_\ellp}$ vanishing on the scheme-theoretic image of $\loc_\ellp$. This implies that the composite $f:=\alpha\circ\jj_{\ellp,U}$ is a non-zero Coleman analytic function vanishing on the rational points of~$X$.
\smallskip

Of course, this construction affords no control on~$f$ at all. The way we obtain some \emph{a priori} control over~$f$ is to exploit an extra structure on the schemes $\Sel_U$ and $\HH^1_f(G_p,U)$ in the form of a \emph{weight filtration}~$\W_\bullet$ on their affine rings, induced from the weight filtration (=descending central series) on~$U$. Using this weight filtration, we will prove the following additional properties of the Chabauty--Kim square:
\begin{itemize}
	\item the map $\loc_p^*\colon\O(\HH^1_f(G_p,U))\to\O(\Sel_U)$ is compatible with the weight filtration;
	\item for every~$\alpha\in\W_m\O(\HH^1_f(G_p,U))$, the composite $\alpha\circ\jj_{p,U}$ is a Coleman analytic function of weight at most~$m$.
\end{itemize}

Assume now for simplicity that $X$ has everywhere potentially good reduction, which implies that all rational points have the same reduction type. When inequality~\eqref{eq:c-k_ineq_series} holds for~$m$, a Hilbert series computation shows that
\[
\dim\W_m\O(\Sel_U)<\dim\W_m\O(\HH^1_f(G_p,U)) \,,
\]
and hence that the map $\loc_p^*\colon\W_m\O(\HH^1_f(G_p,U))\to\W_m\O(\Sel_U)$ must fail to be injective. Hence there is some non-zero $\alpha\in\W_m\O(\HH^1_f(G_p,U))$ vanishing on the image of $\loc_p$, so $f:=\alpha\circ\jj_{p,U}$ is a non-zero Coleman analytic function of weight at most~$m$ vanishing on~$X(\Q)$. This proves Theorem~\ref{thm:main} in the case of everywhere potentially good reduction.

Theorem~\ref{thm:main_det} is a formal consequence. For Theorem~\ref{thm:main_bound}, we need to combine Theorem~\ref{thm:main} with an upper bound on the number of zeroes of a non-zero Coleman analytic function~$f$ in terms of its weight. This latter we accomplish by means of the ``nice differential operators'' machinery of \cite{jen-netan:effective}, by writing down a suitable differential equation satisfied by~$f$.

\subsection*{Overview of sections}

The proofs of our main theorems require two main ingredients: endowing the local and global Selmer schemes with weight filtrations, and bounding the number of zeroes of a Coleman algebraic function in terms of its weight. The first part we accomplish in \S\ref{s:weights} and \S\ref{s:selmer}, with \S\ref{s:weights} establishing basic results on filtered schemes, and \S\ref{s:selmer} applying these to construct weight filtrations on the affine rings of Selmer schemes. The second part we accomplish in \S\ref{s:coleman} and \S\ref{s:weight_bound}, with \S\ref{s:coleman} giving the definition of Coleman algebraic functions and their weights, and \S\ref{s:weight_bound} using a version of the ``nice differential operators'' machinery of \cite{jen-netan:effective} to bound the number of zeroes of a Coleman algebraic function in terms of its weight.

These two strands are then combined in \S\ref{s:effective} to set up our effective Chabauty--Kim theory and prove all of the main results: Theorems~\ref{thm:main}, \ref{thm:main_bound} and its Corollary, and~\ref{thm:main_det}. \S\ref{s:siegel} is then devoted to an in-depth analysis of the case of the thrice-punctured line, and proves Theorem~\ref{thm:main_siegel}, our $S$-uniform Siegel's Theorem.

\subsection*{Acknowledgements}

I am pleased to thank Netan Dogra and Jennifer Balakrishnan for explaining to me their thoughts on effective Chabauty--Kim, and to thank Pierre Berthelot, Noam Elkies, Majid Hadian, Richard Hain, Minhyong Kim, Martin Olsson, Jonathan Pridham and Umberto Zannier for taking the time to answer my questions during the writing of this paper. I am grateful to Alex Best, Jennifer Balakrishnan and Netan Dogra for their comments and corrections on a preliminary version of this article; in particular, both the use of Lemma~\ref{lem:global_semisimplicity} and the content of \S\ref{ss:better_omega} came about as a result of their suggestions. I was a guest of the Max-Planck Institut f\"ur Mathematik while I wrote this paper, and I gratefully acknowledge their support.


%% file: weights.tex
\section{Weight filtrations and Hilbert series}\label{s:weights}

In this section, we discuss the theory of filtered rings and their Hilbert series. \S\ref{ss:filtrations} is devoted to recalling the basic definitions, and to describing in detail some of the examples which will recur throughout this paper. The second section \S\ref{ss:cohomology} is devoted to a rather less trivial example, showing that the continuous Galois cohomology scheme with coefficients in a filtered $\Q_\ellp$-pro-unipotent group~$U$ admits a natural filtration induced from that on~$U$. This construction is fundamental in the theory we set up, since the local and global Selmer schemes of the Chabauty--Kim method are all closed subschemes of continuous Galois cohomology schemes, and this construction is how we will endow them with filtrations.

\subsection{Filtrations on affine schemes}\label{ss:filtrations}

To begin with, we fix notation and terminology which will be in use throughout this section. We fix a field~$\Fd$ of characteristic~$0$, which for us will almost always be equal to~$\Q_\ellp$. A \emph{filtration} (or \emph{weight filtration} when we want to distinguish it from e.g.\ a Hodge filtration) on an $\Fd$-algebra~$A$ consists of an increasing sequence
\[
\W_0A\subseteq\W_1A\subseteq\W_2A\subseteq\dots\subseteq A
\]
of $\Fd$-subspaces which is compatible with the multiplication on $A$, in the sense that $1\in\W_0A$ and if $x\in\W_iA$ and $y\in\W_jA$ then $xy\in\W_{i+j}A$. In practice, almost all of the filtrations we will consider will be exhaustive ($A=\bigcup_i\W_iA$), have each $\W_iA$ finite-dimensional and satisfy $\W_0A=\Fd$, but we will not assume this in general.

A \emph{filtered affine $\Fd$-scheme} is an affine $\Fd$-scheme~$X$ endowed with a filtration on its affine ring~$\O(X)$; a \emph{filtered morphism} of filtered affine $\Fd$-schemes is one which is compatible with the filtration. We denote the category of filtered affine $\Fd$-schemes by $\WAFF_{\Fd}$, which is dual to the category $\WALG_{\Fd}$ of filtered $\Fd$-algebras. When we want to distinguish between a filtered affine $\Fd$-scheme and its underlying affine $\Fd$-scheme, we denote the former by~$X$ (for example) and the latter by~$X^\circ$.

The category $\WAFF_{\Fd}$ admits many of the same constructions as the category of affine~$\Fd$-schemes, such as finite disjoint unions, products, and more generally all small limits\footnote{In fact, it also has all small colimits, but these are not so well-behaved.}. A common construction we will see throughout this paper is that if~$X$ is a filtered affine $\Fd$-scheme and~$Z^\circ$ is a closed $\Fd$-subscheme of its underlying $\Fd$-scheme~$X^\circ$, then there is an induced filtration on~$\O(Z^\circ)$, namely the image of the filtration on~$\O(X)$ under the map $\O(X^\circ)\twoheadrightarrow\O(Z^\circ)$. We refer to this simply as the \emph{induced filtration} on~$Z^\circ$. The equaliser of a pair of maps $X\rightrightarrows Y$ of filtered affine~$\Fd$-schemes is always a closed $\Fd$-subscheme of~$X$ with the induced filtration.

\subsubsection{Hilbert series}\label{sss:hilbert}

The \emph{Hilbert series} of a filtered affine $\Fd$-scheme~$X$ is the formal power series
\[
\HS_X(t) := \sum_{i\geq0}\dim_{\Fd}\!\left(\gr^\W_i\O(X)\right)t^i\in\N_0^\infty\llbrack t\rrbrack
\]
whose coefficients are non-negative integers or $\infty$. We let $\preceq$ denote the partial ordering on $\N_0^\infty\llbrack t\rrbrack$ where $\sum_{i\geq0}a_it^i\preceq\sum_{i\geq0}a'_it^i$ just when $\sum_{i=0}^na_i\leq\sum_{i=0}^na'_i$ for all~$n$. In other words, $\HS(t)\preceq\HS'(t)$ just when $\frac1{1-t}\HS(t)\leq\frac1{1-t}\HS'(t)$ coefficientwise. We will frequently use the following easily-verified properties of Hilbert series and the relation~$\preceq$.

\begin{lemma}
Let $\HS_1(t)\preceq\HS_1'(t)$ and $\HS_2(t)\preceq\HS_2'(t)$ be elements of $\N_0^\infty\llbrack t\rrbrack$. Then $\HS_1(t)+\HS_2(t)\preceq\HS_1'(t)+\HS_2'(t)$ and $\HS_1(t)\cdot\HS_2(t)\preceq\HS_1'(t)\cdot\HS_2'(t)$. The same holds for infinite sums and products.
\end{lemma}

\begin{lemma}\label{lem:hilbert_series_algebra}
	\leavevmode
	\begin{enumerate}
		\item If~$X_1$ and~$X_2$ are filtered affine $\Fd$-schemes, then
		\[
		\HS_{X_1\times X_2}(t) = \HS_{X_1}(t)\cdot\HS_{X_2}(t) \,.
		\]
		The same holds for infinite products.
		\item If~$X$ is a filtered affine $\Fd$-scheme and $Z\subseteq X$ is a closed subscheme with the induced filtration, then
		\[
		\HS_Z(t)\preceq\HS_X(t) \,.
		\]
		\item If $f\colon X'\twoheadrightarrow X$ is a filtered morphism which is scheme-theoretically dense (e.g.\ $f$ dominant and~$X$ reduced), then we have
		\[
		\HS_X(t)\preceq\HS_{X'}(t) \,.
		\]
	\end{enumerate}
\end{lemma}

\subsubsection{Functors of points}\label{sss:functors_of_points}

We will often define filtered affine $F$-schemes in terms of the functors they represent. If~$X$ is a filtered affine $\Fd$-scheme and~$\Lambda$ a filtered $\Fd$-algebra, then we write
\[
X(\Lambda) := \Hom_{\WALG_{\Fd}}(\O(X),\Lambda)
\]
for the set of $\Lambda$-valued points of~$X$. The assignment $\Lambda\mapsto X(\Lambda)$ is a functor $\WALG_{\Fd}\to\SET$: the \emph{functor of points} of~$X$. The Yoneda Lemma ensures that any natural transformation between the functors of points of two filtered affine $\Fd$-schemes is induced by a unique morphism of filtered affine $\Fd$-schemes. Because of this, we do not distinguish notationally between a filtered affine $\Fd$-scheme and its associated functor of points.

If $Z\subseteq X$ is a closed subscheme with the induced filtration, then its functor of points is the subfunctor of the functor of points of~$X$, consisting of all the morphisms $\Spec(\Lambda)\to X$ of filtered affine $F$-schemes whose image (on the level of underlying topological spaces) is contained in~$Z$.

The functor $(-)^\circ\colon\WAFF_{\Fd}\to\AFF_{\Fd}$ given by forgetting the filtration admits a simple characterisation in terms of functors of points.

\begin{lemma}\label{lem:underlying_functor}
	Let~$X\in\WAFF_{\Fd}$ be a filtered affine $\Fd$-scheme. Then the underlying (non-filtered) affine $\Fd$-scheme~$X^\circ$ of~$X$ represents the composite functor
	\[
	\ALG_{\Fd} \xrightarrow{(-)^\triv} \WALG_{\Fd} \xrightarrow{X(-)} \SET \,,
	\]
	where $(-)^\triv$ denotes the functor which endows an $\Fd$-algebra~$\Lambda$ with the trivial filtration given by $\W_n\Lambda^\triv=\Lambda$ for all~$n$.
	\begin{proof}
		This is just a matter of chasing definitions. We have
		\[
		X(\Lambda^\triv)=\Hom_{\WALG_{\Fd}}(\O(X),\Lambda^\triv)=\Hom_{\ALG_{\Fd}}(\O(X^\circ),\Lambda)= X^\circ(\Lambda)
		\]
		for every~$\Lambda\in\ALG_{\Fd}$, so~$X^\circ$ represents the functor~$\Lambda\mapsto X(\Lambda^\triv)$.
	\end{proof}
\end{lemma}

\subsubsection{Examples: vector spaces and pro-unipotent groups}

The following two examples play a basic role in the theory in this paper. We describe both examples in the most general setting where some objects may be infinite-dimensional, but the reader is free to assume that all objects involved are finite-dimensional on a first reading, which simplifies matters.

\begin{example}\label{ex:filtered_vs}
A \emph{pro-finite-dimensional vector space} is a pro-object of the category of finite-dimensional vector spaces. A \emph{filtration} on a pro-finite-dimensional vector space~$V$ is an increasing sequence of subobjects
\[
0\leq\dots\leq\W_{-3}V\leq\W_{-2}V\leq\W_{-1}V=V \,,
\]
indexed by negative integers.

A \emph{pro-finite-dimensional vector space}~$V$ has an associated affine space, which we also denote by~$V$. If~$V$ is finite-dimensional, its associated affine space is $\Spec(\Sym^\bullet(V^\dual))$. In general, we write $V=\varprojlim V_i$ as an inverse limit of finite-dimensional vector spaces, and the affine space associated to~$V$ is $\Spec(\Sym^\bullet(\varinjlim V_i^\dual))$.

If~$V$ is endowed with a filtration, then there is an induced filtration on the symmetric algebra $\Sym^\bullet(\varinjlim V_i^\dual)$, and this makes the affine space associated to~$V$ into a filtered affine $\Fd$-scheme. When~$V$ is finite-dimensional, the functor of points of~$V$ is
\[
\Lambda\mapsto\W_0(\Lambda\otimes_{\Fd}V) \,,
\]
where~$\Lambda\otimes_{\Fd}V$ is endowed with the tensor product filtration; the general case is an inverse limit of this. The Hilbert series of~$V$ is
\[
\HS_V(t) = \prod_{n>0}\left(1-t^n\right)^{-\dim_{\Fd}(\gr^\W_{-n}V)}\,.
\]
\end{example}

\begin{example}\label{ex:filtered_unipotent}
Let~$U/\Fd$ be a pro-unipotent group. A \emph{filtration} on~$U$ is an increasing sequence
\[
1\subseteq\dots\subseteq\W_{-3}U\subseteq\W_{-2}U\subseteq\W_{-1}U=U
\]
of subgroup-schemes, indexed by negative integers, such that the image of the commutator map $[\cdot,\cdot]\colon\W_{-i}U\times\W_{-j}U\rightarrow U$ is contained in $\W_{-i-j}U$ for all $i,j\geq1$.

The Lie algebra $\Lie(U):=\ker\left(U(\Fd[\epsilon]/(\epsilon^2))\to U(\Fd)\right)$ is naturally a pro-finite-dimensional vector space, being the inverse limit of the Lie algebras of the finite-dimensional quotients of~$U$. A filtration on~$U$ corresponds under the logarithm isomorphism $U\cong\Lie(U)$ to a filtration on~$\Lie(U)$ by $\Fd$-subspaces for which the Lie bracket is a filtered map. Thus~$U$ becomes a filtered affine $\Fd$-scheme by pulling back the filtration on $\O(\Lie(U))$ from Example~\ref{ex:filtered_vs} along the logarithm isomorphism~$U\cong\Lie(U)$. It is then easy to check that the multiplication map $U\times U\to U$ and the identity map $\Spec(\Fd)\to U$ are morphisms of filtered affine $\Fd$-schemes, so that the filtration on~$\O(U)$ is compatible with the Hopf algebra structure maps.

Most often, we will assume that the filtration on~$U$ is \emph{separated}, by which we mean that~$\cap_{n\in\N}\W_{-n}U=\{1\}$.
\end{example}

Since filtered pro-unipotent groups will appear throughout this paper, we fix some notation for particular subquotients which will appear frequently.

\begin{notation}\label{notn:U_n}
	If~$U$ is an $\Fd$-pro-unipotent group with a filtration in the sense of Example~\ref{ex:filtered_unipotent}, then we write\footnote{The quotients appearing in this expression are taken in the most naive possible sense: that is, the $\Lambda$-points of~$U_n$ is the quotient of the $\Lambda$-points of~$U$ by those of~$\W_{-n-1}U$ for all $\Fd$-algebras~$\Lambda$, and similarly for~$V_n$.}
	\[
	U_n := U/\W_{-n-1}U \hspace{0.4cm}\text{and}\hspace{0.4cm} V_n := \gr^\W_{-n}U = \W_{-n}U/\W_{-n-1}U
	\]
	for all~$n\geq1$. Then~$U_n$ is an $\Fd$-pro-unipotent group and~$V_n$ is a vector group, which we often conflate with its underlying pro-finite-dimensional vector space. The quotient maps $U_n\twoheadrightarrow U_{n-1}$ fit into central extensions
	\begin{equation}\label{eq:extension_seq}
	1 \to V_n \to U_n \to U_{n-1} \to 1
	\end{equation}
	for $n\geq2$, Moreover, we have~$U=\varprojlim U_n$ in the category of affine $\Fd$-schemes provided that the filtration~$\W_\bullet$ is separated ($\bigcap_n\W_{-n}U=1$).
	
	Both~$U_n$ and~$V_n$ are naturally endowed with filtrations. In the case of~$U_n$ we take the image of the filtration on~$U$ under the map~$U\twoheadrightarrow U_n$, and in the case of~$V_n$ we give it the filtration supported in filtration-degree~$-n$. With these conventions, \eqref{eq:extension_seq} is a central extension of filtered affine $\Fd$-group schemes, in the strong sense that the maps $U_n\to U_{n-1}$ and $V_n\to U_n$ admit a splitting and retraction, respectively, in the category of filtered affine $\Fd$-schemes. Moreover, the isomorphism $U=\varprojlim U_n$ is an isomorphism in the category of filtered affine $\Fd$-group schemes.
\end{notation}

\begin{remark}
	It is sometimes convenient to adopt the convention that~$\W_0U=U$, so that~$U_0=\W_0U/\W_{-1}U$ is the trivial group, and we have the central extension~\eqref{eq:extension_seq} also for~$n=1$ (reflecting the fact that~$U_1=V_1$). This convention allows us, for example, to begin inductive arguments at~$n=0$, where the result is usually trivial. See the proof of Lemma~\ref{lem:filtered_crystalline_path} for example.
\end{remark}

The pro-unipotent fundamental groups which appear in this paper will all be pro-unipotent surface groups, or quotients thereof. For later use, we calculate here the Hilbert series of these groups.

\begin{lemma}\label{lem:surface_group}
	Let~$g$ and~$r$ be non-negative integers, not both zero, and let~$U_{g,r}$ denote the $\Fd$-pro-unipotent completion of the surface group
	\[
	\Sigma_{g,r} := \left\langle a_1,\dots,a_g,b_1,\dots,b_g,c_1,\dots,c_r \:\mid\: [a_1,b_1]\cdot\ldots\cdot[a_g,b_g]\cdot c_1\cdot\ldots\cdot c_r=1\right\rangle \,.
	\]
	We endow $U_{g,r}$ with the filtration whereby $\W_{-1}U_{g,r}=U_{g,r}$, $\W_{-2}U_{g,r}$ is the subgroup generated by the commutator subgroup of~$U_{g,r}$ and the elements~$c_1,\dots,c_r$, and for~$k\geq3$, $\W_{-k}U_{g,r}$ is the subgroup generated by the commutators of elements in~$\W_{-i}U_{g,r}$ and~$\W_{-j}U_{g,r}$ for~$i+j=k$.
	
	With respect to this filtration, the Hilbert series of~$U_{g,r}$ is
	\begin{equation}\label{eq:surface_series}\tag{$\ast$}
	\HS_{U_{g,r}}(t) = \frac1{1-2gt-(r-1)t^2} \,.
	\end{equation}
	\begin{proof}
		When~$r>0$, \eqref{eq:surface_series} follows from the observation that~$U_{g,r}$ is the free pro-unipotent group generated by the elements~$a_i$,$b_i$ and~$c_j$ for $1\leq i\leq g$ and $1\leq j\leq r-1$, and its weight filtration is the finest filtration for which each $a_i,b_i\in\W_{-1}U_{g,r}$ and each $c_j\in\W_{-2}U_{g,r}$. This implies that~$\O(U_{g,r})$ is the shuffle algebra on symbols $A_i$, $B_i$ and~$C_j$ for $1\leq i\leq g$ and $1\leq j\leq r-1$, endowed with the natural filtration coming from declaring that each $A_i,B_i\in\W_1\O(U_{g,r})$ and each $C_j\in\W_2\O(U_{g,r})$. Thus the $k$th coefficient of $\HS_{U_{g,r}}$ is equal to the number of words in the letters $(A_i)_{i=1}^g,(B_i)_{i=1}^g,(C_j)_{j=1}^{r-1}$ of total weight~$k$. A standard combinatorial argument shows that this is equal to the $k$th coefficient of $(1-2gt-(r-1)t^2)^{-1}=\sum_{l=0}^\infty(2gt+(r-1)t^2)^l$, as claimed.
		
		When $r=0$, the filtration on~$U_{g,r}$ is its descending central series, so the dimensions of the graded pieces of the Lie algebra of~$U_{g,r}$ are given by
		\[
		\dim_{\Fd}\gr^\W_{-k}\Lie(U_{g,r}) = \frac1k\cdot\sum_{d\mid k}\mu(k/d) \left(\sum_{i=0}^{\lfloor d/2\rfloor}(-1)^i\frac d{d-i}{{d-i}\choose i}(2g)^{d-2i}\right)
		\]
		by the main theorem of \cite{labute:dcc}. The proof of \cite[Proposition~4]{labute:dcc} shows that the integers appearing on the right-hand side of the above expression are the unique integers~$d_k$ such that $1-2gt+t^2=\prod_{k=1}^\infty(1-t^k)^{d_k}$. This implies that the Hilbert series of~$U_{g,r}$ is $(1-2gt+t^2)^{-1}$, as desired.
	\end{proof}
\end{lemma}

\subsubsection{Example: homogenous spaces under pro-unipotent groups}

We give here a slightly more involved example, that of homogenous spaces under filtered pro-unipotent groups. Unlike the previous examples, we define the filtration here in an indirect manner, so that the homogenous space in question represents a particular functor. This kind of abstract definition will be typical for the constructions in the rest of this paper.

\begin{lemma}\label{lem:quotient_by_unipotent_subgroup}
	Let~$U$ be a pro-unipotent group over a characteristic~$0$ field~$\Fd$ endowed with a separated filtration in the sense of Example~\ref{ex:filtered_unipotent}, and let~$U^+\leq U$ be a closed subgroup-scheme. Give~$U^+$ the filtration $\W_\bullet U^+:=\W_\bullet U\cap U^+$.
	
	Then the functor $\WALG_{\Fd}\to\SET_*$ given by
	\[
	\Lambda\mapsto U^+(\Lambda)\backslash U(\Lambda) \,,
	\]
	is represented by a filtered affine $\Fd$-scheme $U^+\backslash U\in\WAFF_{\Fd}$, and the Hilbert series of~$U^+\backslash U$ satisfies
	\[
	\HS_{U^+\backslash U}(t)\cdot\HS_{U^+}(t) = \HS_U(t) \,.
	\]
	\begin{proof}
		Let~$L^+\leq L$ denote the Lie algebras of~$U^+\leq U$, respectively, and choose a $\W$-filtered complement~$V$ of~$L^+$ inside~$L$, i.e.\ a pro-finite-dimensional subspace of~$L$ such that~$L^+\oplus V=L$ as filtered pro-finite-dimensional vector spaces. Consider the map
		\begin{equation}\label{eq:multiplication_map}\tag{$\ast$}
			U^+\times V\to U
		\end{equation}
		given by $(u,v)\mapsto u\cdot\exp(v)$. We will show that~\eqref{eq:multiplication_map} is an isomorphism of filtered affine $\Fd$-schemes. This implies the lemma: we see that~$V$ represents the functor~$\Lambda\mapsto U^+(\Lambda)\backslash U(\Lambda)$, and the equality of Hilbert series follows from Lemma~\ref{lem:hilbert_series_algebra}.
		
		The proof that~\eqref{eq:multiplication_map} is an isomorphism is routine. If we identify $U$ with its Lie algebra $L$, then the multiplication map is given by $(u,v)\mapsto\BCH(u,v)$, where $\BCH(x,y)=x+y+\frac12[x,y]+\dots$ is the Baker--Campbell--Hausdorff power series. It suffices to show that this map is bijective on~$\Lambda$-points for every filtered $\Fd$-algebra~$\Lambda$.
		
		To do this, suppose that~$w$ is a $\Lambda$-point of~$L$, and write~$w_n$ for its image in $L(\Lambda)/\W_{-n-1}$ for~$n\geq0$. We will show that there are unique elements $v_n\in V(\Lambda)/\W_{-n-1}$ and $u_n\in L^+(\Lambda)/\W_{-n-1}$ such that $w_n=\BCH(u_n,v_n)$. The elements $(v_n)_{n\geq0}$ and $(u_n)_{n\geq0}$ then define unique elements $v\in V(\Lambda)$ and~$u\in L^+(\Lambda)$ such that $w=\BCH(u,v)$, as desired.
		
		We show the existence of~$v_n$ and~$u_n$ by induction on~$n$, the base case~$n=0$ being trivial. Suppose that we have elements~$v_n$ and~$u_n$ satisfying $w_n=\BCH(u_n,v_n)$, and let~$v_n'\in V(\Lambda)/\W_{-n-2}$ and~$u_n'\in L^+(\Lambda)/\W_{-n-2}$ be arbitrarily chosen lifts. For~$v_n''\in\gr^\W_{-n-1}V(\Lambda)$ and~$u_n''\in\gr^\W_{-n-1}L^+(\Lambda)$, we have
		\[
		\BCH(u_n'+u_n'',v_n'+v_n'')=\BCH(u_n',v_n')+u_n''+v_n''
		\]
		in $L(\Lambda)/\W_{-n-2}$. Since~$\gr^\W_{-n-1}V\oplus\gr^\W_{-n-1}L^+=\gr^\W_{-n-1}L$ and $\BCH(u_n',v_n')-w_{n+1}\in\gr^\W_{-n-1}L(\Lambda)$, we see that there are unique choices of~$v_n''\in\gr^\W_{-n-1}V(\Lambda)$ and~$u_n''\in\gr^\W_{-n-1}L^+(\Lambda)$ such that $\BCH(u_n'+u_n'',v_n'+v_n'')=w_{n+1}$. This completes the induction, setting $v_{n+1}=v_n'+v_n''$ and $u_{n+1}=u_n'+u_n''$.
	\end{proof}
\end{lemma}

\begin{remark}\label{rmk:underlying_homogenous_space}
	It is not immediately obvious from the above definition that the underlying affine scheme of~$U^+\backslash U$ is the quotient of~$U$ by~$U^+$ in the usual (non-filtered) sense. However, this is easy to show using Lemma~\ref{lem:underlying_functor}: if~$U^\circ$ and~$U^{+\circ}$ denote the underlying affine $\Fd$-schemes of~$U$ and~$U^+$, respectively, then we have
	\begin{align*}
	(U^+\backslash U)^\circ(\Lambda) &= (U^+\backslash U)(\Lambda^\triv) = U^+(\Lambda^\triv)\backslash U(\Lambda^\triv) \\
	 &= U^{+\circ}(\Lambda)\backslash U^\circ(\Lambda) = (U^{+\circ}\backslash U^\circ)(\Lambda)
	\end{align*}
	for every $\Fd$-algebra~$\Lambda$, so $(U^+\backslash U)^\circ=U^{+\circ}\backslash U^\circ$ is the quotient of~$U$ by~$U^+$ in the usual non-filtered sense.
\end{remark}

\begin{remark}
	An alternative way to define the filtration on the homogenous space $U^+\backslash U$ is to note that $\O(U^+\backslash U)$ is a subring of~$\O(U)$, so we can endow it with the restriction of the filtration on~$\O(U)$. This in fact gives the same filtration as Lemma~\ref{lem:quotient_by_unipotent_subgroup}: it follows from Lemma~\ref{lem:quotient_by_unipotent_subgroup} that the morphism $U\twoheadrightarrow U^+\backslash U$ is split in $\WAFF_{\Fd}$ (it is a surjection of functors), so the inclusion $\O(U^+\backslash U)\hookrightarrow\O(U)$ is split as a morphism of filtered $\Fd$-algebras. This ensures that the filtration on~$\O(U^+\backslash U)$ constructed in Lemma~\ref{lem:quotient_by_unipotent_subgroup} is the restriction of the filtration on~$\O(U)$.
\end{remark}

\subsection{Filtrations on cohomology schemes}\label{ss:cohomology}

Suppose that $G$ is a profinite group, and that $U/\Q_\ellp$ is a pro-unipotent group on which $G$ acts continuously (in the sense that it acts continuously on $U(\Q_\ellp)$). If $\Lambda$ is any $\Q_\ellp$-algebra, we endow it with the direct limit topology over its finite-dimensional $\Q_\ellp$-subspaces. This induces a topology on $U(\Lambda)$ for which the $G$-action is continuous, and we define $\HH^1(G,U(\Lambda))$ to be the non-abelian continuous cohomology set \cite[\S I.5.1]{serre:galois_cohomology}\footnote{Strictly speaking, \cite{serre:galois_cohomology} deals only with non-abelian continuous cohomology with coefficients in a \emph{discrete} group. However, the definition still makes sense when the coefficients are permitted to have other topologies.}. This is a pointed set, functorial in $\Lambda$, and hence the sets $\HH^1(G,U(\Lambda))$ assemble into a \emph{continuous cohomology functor}
\[
\HH^1(G,U)\colon\ALG_{\Q_\ellp}\rightarrow\SET_* \,.
\]

\smallskip

Suppose now that~$U$ is endowed with a $G$-stable separated filtration
\[
1\subseteq\dots\subseteq\W_{-3}U\subseteq\W_{-2}U\subseteq\W_{-1}U=U \,,
\]
in the sense of Example~\ref{ex:filtered_unipotent}, and assume for simplicity that each subgroup $\W_{-n}U$ is of finite codimension\footnote{This is automatically the case if~$U$ is finitely generated, which will be the case in almost all our examples.} in~$U$. Suppose moreover that the following two conditions are satisfied:
\begin{enumerate}[label=\arabic*),ref=(\arabic*)]
	\setcounter{enumi}{-1}
	\item\label{condn:H^0} $\HH^0(G,V_n)=0$ for all $n$; and
	\item\label{condn:H^1} $\HH^1(G,V_n)$ is finite-dimensional for all~$n$,
\end{enumerate}
where~$V_n=\gr^\W_{-n}U$ (regarded as a vector space) as in Notation~\ref{notn:U_n}. Then the functor $\HH^1(G,U)$ is representable by an affine $\Q_\ellp$-scheme \cite[Proposition~2]{minhyong:siegel}\footnote{Strictly speaking, \cite[Proposition~2]{minhyong:siegel} only proves this result under the additional assumption that $\HH^i(G,V_n)$ is finite-dimensional for all~$i$ and~$n$: a condition which is always satisfied in the cases we care about. However, this extra assumption turns out to be unnecessary, and our arguments in \S\ref{sss:abstract_filtration} will re-prove representability under these weaker assumptions.}, which is of finite type if~$U$ is finite-dimensional. We denote the representing affine $\Q_\ellp$-scheme also by~$\HH^1(G,U)$.
\smallskip

Our aim here is to explain how the filtration~$\W_\bullet$ on~$U$ induces a filtration on the cohomology scheme $\HH^1(G,U)$. We do this abstractly, by describing the functor it represents; for a more concrete description, see \S\ref{sss:explicit}.

\subsubsection{The functor of points of the filtered cohomology scheme}\label{sss:abstract_filtration}

If now~$\Lambda$ is a \emph{filtered} $\Q_\ellp$-algebra, then the set $U(\Lambda)$ of $\Lambda$-points of~$U$ is a subset of the set $U^\circ(\Lambda^\circ)$ of $\Lambda$-points of the underlying (non-filtered) scheme~$U^\circ$. We have already seen how to endow $U^\circ(\Lambda^\circ)$ with a topology; we give  $U(\Lambda)\subseteq U^\circ(\Lambda^\circ)$ the subspace topology. This makes $U(\Lambda)$ into a topological group with a continuous $G$-action, functorial in~$\Lambda\in\WALG_{\Q_\ellp}$. We thus have a filtered continuous cohomology functor
\[
\HH^1(G,U)\colon\WALG_{\Q_\ellp}\to\SET_*
\]
given by $\Lambda\mapsto\HH^1(G,U(\Lambda))$. This will be the functor of points of the desired filtered cohomology scheme.

\begin{theorem}\label{thm:representability}
	Suppose that the filtration on~$U$ satisfies conditions~\ref{condn:H^0} and~\ref{condn:H^1} above. Then the functor
	\[
	\HH^1(G,U)\colon\WALG_{\Q_\ellp}\to\SET_*
	\]
	is representable by a pointed filtered affine $\Q_\ellp$-scheme.
\end{theorem}

\begin{remark}\label{rmk:representability_compatibility}
	The underlying scheme of the filtered affine $\Q_\ellp$-scheme $\HH^1(G,U)$ produced by Theorem~\ref{thm:representability} is the same as the non-abelian cohomology scheme $\HH^1(G,U^\circ)$ associated to the underlying (non-filtered) pro-unipotent group~$U^\circ$ of~$U$, as defined at the beginning of this section (justifying our use of the same notation for both objects). This follows from Lemma~\ref{lem:underlying_functor} since the underlying affine $\Q_\ellp$-scheme of~$\HH^1(G,U)$ represents the functor $\Lambda\mapsto\HH^1(G,U(\Lambda^\triv))=\HH^1(G,U^\circ(\Lambda))$.
\end{remark}

The proof of Theorem~\ref{thm:representability} is ultimately inductive. For any $n\geq0$, we let $U_n=U/\W_{-n-1}U$ and $V_n=\gr^\W_{-n}U=\W_{-n}U/\W_{-n-1}U$ as in Notation~\ref{notn:U_n}, so that we have a central extension
\begin{equation}\label{eq:group_extn}\tag{$\ast$}
	1\to V_n\to U_n\to U_{n-1}\to 1
\end{equation}
for all~$n\geq1$. Since this sequence is split in~$\WAFF_{\Q_\ellp}$, it follows that for every filtered $\Q_\ellp$-algebra~$\Lambda$, the central extension
\[
1 \to V_n(\Lambda) \to U_n(\Lambda) \to U_{n-1}(\Lambda) \to 1
\]
is topologically split: $U_n(\Lambda)\to U_{n-1}(\Lambda)$ admits a continuous splitting and $V_n(\Lambda)\subseteq U_n(\Lambda)$ has the subspace topology. We thus obtain part of a long exact sequence in non-abelian cohomology \cite[Proposition~43]{serre:galois_cohomology}\footnote{Again, \cite{serre:galois_cohomology} only the treats the case where all topological groups in question are discrete. The same applies in general providing one restricts attention to topologically split central extensions.}
\begin{equation}\label{eq:cohomology_es_sections}\tag{$\ast\ast_\Lambda$}
\HH^1(G,V_n(\Lambda)) \to \HH^1(G,U_n(\Lambda)) \to \HH^1(G,U_{n-1}(\Lambda)) \xrightarrow{\delta} \HH^2(G,V_n(\Lambda)) \,.
\end{equation}
This deserves a little more explanation. The sequence above is an exact sequence of pointed sets, but carries further structure which is relevant for us. Specifically, there is an action of the abelian group $\HH^1(G,V_n(\Lambda))$ on the set $\HH^1(G,U_n(\Lambda))$, given by pointwise multiplication of cocycles, whose orbits are the fibres of the map $\HH^1(G,U_n(\Lambda))\to\HH^1(G,U_{n-1}(\Lambda))$.

\begin{proposition}
	The action of $\HH^1(G,V_n(\Lambda))$ on $\HH^1(G,U_n(\Lambda))$ is free.
	\begin{proof}
		If~$\xi\colon G\to U_n(\Lambda)$ is a continuous cocycle, then the stabiliser of the class $[\xi]\in\HH^1(G,U_n(\Lambda))$ is the image of a certain coboundary map $\HH^0(G,{}_\xi U_{n-1}(\Lambda))\to\HH^1(G,V_n(\Lambda))$, where ${}_\xi U_{n-1}(\Lambda)$ denotes the topological group $U_{n-1}(\Lambda)$ with the $\xi$-twisted $G$-action $g\colon u\mapsto \xi(g)\cdot g(u)\cdot\xi(g)^{-1}$. But ${}_\xi U_{n-1}(\Lambda)$ is an iterated central extension of the groups $V_i(\Lambda)=\W_i\Lambda\otimes_{\Q_\ellp}V_i$ for $i<n$, so that condition~\ref{condn:H^0} implies $\HH^0(G,{}_\xi U_{n-1}(\Lambda))=1$. Thus the action on~$[\xi]$ has trivial stabiliser, and hence the action is free.
	\end{proof}
\end{proposition}

Now the construction of the sequence~\eqref{eq:cohomology_es_sections}, including the group action, is functorial in~$\Lambda$, and hence constitutes an exact sequence
\begin{equation}\label{eq:cohomology_es_presheaves}\tag{$\ast\ast$}
	\HH^1(G,V_n) \to \HH^1(G,U_n) \to \HH^1(G,U_{n-1}) \xrightarrow{\delta} \HH^2(G,V_n)
\end{equation}
of functors $\WALG_{\Q_\ellp}\to\SET_*$, where the leftmost term is a functor valued in abelian groups, acting pointwise freely on the second term. The outer terms of the sequence are controlled by the following proposition.

\begin{proposition}\label{prop:abelian_cohomology}
	For a filtered~$\Q_\ellp$-algebra~$\Lambda$, let $\HH^i(G,V_n(\Lambda))$ denote the $i$th continuous cohomology of~$V_n(\Lambda)=\W_n\Lambda\otimes_{\Q_\ellp}V_n$, endowed with topology described at the beginning of \S\ref{sss:abstract_filtration}, i.e.\ the inductive limit topology over its finite-dimensional $\Q_\ellp$-subspaces. Let~$\HH^i(G,V_n):=\HH^i(G,V_n(\Q_\ellp))$ denote the continuous cohomology of~$V_n$ viewed as a $\Q_\ellp$-vector space. Then for any integer~$i$, we have
	\[
	\HH^i(G,V_n(\Lambda)) = \W_n\Lambda\otimes_{\Q_\ellp}\HH^i(G,V_n) \,,
	\]
	natural in~$\Lambda$.
	\begin{proof}
		The continuous cohomology $\HH^\bullet(G,V_n(\Lambda))$ is calculated as the cohomology of a complex $\Chain^\bullet(G,V_n(\Lambda))$, whose $i$th term is the group of continuous maps from~$G^i$ to $V_n(\Lambda)=\W_n\Lambda\otimes_{\Q_\ellp}V_n$. It thus suffices to show that
		\[
		\Map_\cts(G^i,\W_n\Lambda\otimes_{\Q_\ellp}V_n) = \W_n\Lambda\otimes_{\Q_\ellp}\Map_\cts(G^i,V_n) \,.
		\]
		A little care is needed here, since maps out of compact spaces do not commute with filtered colimits in general \cite[Example~2.5.5]{may:alg_top}. We provide instead a direct argument.
		
		It suffices to show that any continuous map from~$G^i$ to a $\Q_\ellp$-vector space factors through a finite-dimensional subspace. Indeed, suppose that $\xi\colon G^i\to\Q_\ellp^{\oplus J}$ is a continuous map for some indexing set~$J$. We may suppose without loss of generality that no coordinate of the map~$\xi$ is zero. For each~$j\in J$ we choose an open subset $\mathcal U_j\subseteq\Q_\ellp$ containing~$0$ but not the image of the $j$th coordinate of~$\xi$. For a finite subset~$J_0\subseteq J$ we define $\mathcal U_{J_0}=\bigoplus_{j\in J\setminus J_0}\mathcal U_j\oplus\bigoplus_{j\in J_0}\Q_\ellp$, which is an open subset of~$\Q_\ellp^{\oplus J}$. The sets $\mathcal U_{J_0}$ cover~$\Q_\ellp^{\oplus J}$ and are closed under finite unions, hence there is some~$J_0$ such that the image of~$\xi$ is contained in~$\mathcal U_{J_0}$. But by construction this is only possible if~$J_0=J$, hence~$J$ is finite as desired.
	\end{proof}
\end{proposition}

\begin{corollary}\label{cor:abelian_representability}
	\leavevmode
	\begin{itemize}
		\item The cohomology functor~$\HH^1(G,V_n)\colon\WALG_{\Q_\ellp}\to\AB$ is representable by a commutative filtered affine $\Q_\ellp$-group scheme. Precisely, the representing object is the vector group associated to the vector space $\HH^1(G,V_n)$, given the filtration supported in degree~$-n$ as per Example~\ref{ex:filtered_vs}.
		\item The cohomology functor~$\HH^2(G,V_n)\colon\WALG_{\Q_\ellp}\to\AB$ is a subfunctor of a representable functor.
	\end{itemize}
\begin{proof}
	For the first point, we simply note that the functor of points of the finite-dimensional vector space $\HH^1(G,V_n)$ is given by $\Lambda\mapsto\W_n\Lambda\otimes_{\Q_\ellp}\HH^1(G,V_n)$, which is the cohomology functor $\HH^1(G,V_n)$ by Proposition~\ref{prop:abelian_cohomology}.
	
	For the second point, we cannot necessarily apply the same argument, since we have not assumed that the vector space $\HH^2(G,V_n)$ is finite-dimensional. Instead, we set $\HH^2(G,V_n)^\wedge:=\Spec(\Sym^\bullet(\HH^2(G,V_n)^\dual))$, which is the affine space associated to the double-dual $\HH^2(G,V_n)^{\dual\dual}$. Endowing $\HH^2(G,V_n)^\wedge$ with its natural filtration, its functor of points is given by
	\[
	\Lambda\mapsto\Hom_{\Q_\ellp}(\HH^2(G,V_n)^\dual,\W_n\Lambda) \,,
	\]
	and this clearly contains the cohomology functor $\Lambda\mapsto\W_n\Lambda\otimes_{\Q_\ellp}\HH^2(G,V_n)$ as a subfunctor.
\end{proof}
\end{corollary}

Using this proposition, we are now in a position to complete the proof of Theorem~\ref{thm:representability}. The functor $\HH^1(G,U_0)$ is representable by the trivial group. From here, we proceed inductively, assuming henceforth that $\HH^1(G,U_{n-1})$ is representable by a filtered affine $\Q_\ellp$-scheme.

Now the coboundary map $\delta\colon\HH^1(G,U_{n-1})\to\HH^2(G,V_n)$ is a natural transformation from a representable functor to a subfunctor of a representable functor, and hence its kernel $\ker(\delta)$ is representable (by a closed subscheme of $\HH^1(G,U_{n-1})$ with the induced filtration).

Now sequence~\eqref{eq:cohomology_es_presheaves} implies that the functor $\HH^1(G,U_n)$ is a $\HH^1(G,V_n)$-torsor over $\ker(\delta)$, in the sense that its $\Lambda$-points are a $\HH^1(G,V_n(\Lambda))$-torsor over $\ker(\delta)(\Lambda)$ for every filtered $\Q_\ellp$-algebra~$\Lambda$. In particular, $\HH^1(G,U_n)\to\ker(\delta)$ is a surjection of functors with representable codomain, hence is split. Thus $\HH^1(G,U_n)\simeq\ker(\delta)\times\HH^1(G,V_n)$ is the trivial $\HH^1(G,V_n)$-torsor over~$\ker(\delta)$, so is representable by the product of~$\ker(\delta)$ and $\HH^1(G,V_n)$.

\smallskip

This completes the inductive step, showing that~$\HH^1(G,U_n)$ is representable for all~$n$, and hence concludes the proof of Theorem~\ref{thm:representability} in the case that~$U$ is finite-dimensional. The general case is then given by the following proposition.

\begin{proposition}
	We have the equality $\HH^1(G,U)=\varprojlim_n\HH^1(G,U_n)$ of functors $\WAFF_{\Q_\ellp}\to\SET_*$.
	\begin{proof}
		The assumption that the filtration on~$U$ is separated implies that the natural map~$U\to\varprojlim_nU_n$ is a $G$-equivariant isomorphism of filtered affine $\Q_\ellp$-schemes. Hence for every filtered $\Q_\ellp$-algebra~$\Lambda$, the natural map $U(\Lambda)\to\varprojlim_nU_n(\Lambda)$ is a $G$-equivariant isomorphism of topological groups. It thus suffices to show that the map
		\[
		\HH^1(G,\varprojlim_nU_n(\Lambda)) \to \varprojlim_n\HH^1(G,U_n(\Lambda))
		\]
		is bijective. We do this by a Mittag--Leffler argument, following \cite[Lemma~4.0.5]{me:thesis}.
		
		For injectivity, suppose that~$\xi$ and~$\xi'$ are $\varprojlim_nU_n(\Lambda)$-valued cocycles whose images~$\xi_n$ and~$\xi'_n$ in $U_n(\Lambda)$ represent the same cohomology class for every~$n$. This says that there is an element~$u_n\in U_n(\Lambda)$ such that $\xi'_n(g)=u_n^{-1}\cdot\xi(g)\cdot g(u_n)$ for all~$g\in G$. In fact, the element~$u_n$ is unique, for, if $u_n'$ were another element, then we would have
		\[
		\xi(g)\cdot g(u_n'u_n^{-1})\cdot \xi(g)^{-1} = u_n'u_n^{-1}
		\]
		for all~$g\in G$. In other words, $u_n'u_n^{-1}$ is fixed under the $\xi$-twisted $G$-action on~$U_n(\Lambda)$. However, we saw above that~$\HH^0(G,{}_\xi U_n(\Lambda))=1$, whence $u_n'=u_n$ as claimed.
		
		It follows from unicity that the image of~$u_n$ in $U_{n-1}(\Lambda)$ is the element~$u_{n-1}$, and hence the elements $(u_n)_{n\in\N}$ define an element $u\in\varprojlim_nU_n(\Lambda)$. It then follows that $\xi'(g) = u^{-1}\cdot\xi(g)\cdot g(u)$ for all~$g\in G$, so that~$\xi$ and~$\xi'$ represent the same cohomology class. This proves injectivity.
		
		\smallskip
		
		For surjectivity, suppose we are given a compatible sequence of elements $[\xi_n]\in\HH^1(G,U_n(\Lambda))$. Choose a cocycle~$\xi_1$ representing $[\xi_1]$. We show how to lift~$\xi_1$ to a cocycle $\xi_2$ representing~$[\xi_2]$. To begin with, choose any cocycle $\xi_2'$ representing~$[\xi_2]$. The image of~$\xi_2$ in~$U_1(\Lambda)$ represents the same cohomology class as~$\xi_1$, hence there is a~$u_1\in U_1(\Lambda)$ such that the image of~$\xi_2$ is the cocycle $g\mapsto u_1^{-1}\cdot\xi_1(g)\cdot g(u_1)$. Let~$u_2\in U_2(\Lambda)$ be any lift of~$u_1$ along the surjection~$U_2(\Lambda)\to U_1(\Lambda)$; then the cocycle $\xi_2\colon g\mapsto u_2^{-1}\cdot\xi_2'(g)\cdot g(u_2)$ also represents~$[\xi_2]$ and lies over~$\xi_1$.
		
		Iterating this construction, we obtain continuous cocycles $\xi_n\in\Cycle^1(G,U_n(\Lambda))$ representing the classes~$[\xi_n]$ which are compatible under the maps $U_n(\Lambda)\to U_{n-1}(\Lambda)$. These maps define a continuous cocycle $\xi\colon G\to\varprojlim_nU_n(\Lambda)$, and it follows from the construction that~$[\xi]$ maps to~$[\xi_n]$ in~$U_n(\Lambda)$ for all~$n$. This implies that the map  $\HH^1(G,\varprojlim_nU_n(\Lambda)) \to \varprojlim_n\HH^1(G,U_n(\Lambda))$ is surjective, as desired.
	\end{proof}
\end{proposition}

We note for later use the following consequence of our construction, which provides the archetypal example of a bound on Hilbert series.

\begin{corollary}\label{cor:cohomology_hilbert_series_bound}
	The filtered cohomology scheme $\HH^1(G,U)$ is (non-canonically) a closed subscheme of $\prod_{n>0}\HH^1(G,V_n)$, with the induced filtration. In particular,
	\[
	\HS_{\HH^1(G,U)}(t) \preceq \prod_{n>0}(1-t^n)^{-\dim_{\Q_\ellp}(\HH^1(G,V_n))} \,.
	\]
\end{corollary}

\subsubsection{An alternative construction}\label{sss:explicit}

Although we will not need this in what follows, we also describe an alternative construction of the filtration on~$\HH^1(G,U)$ which is more explicit than the construction in Theorem~\ref{thm:representability}. Let~$\Cycle^1(G,U)$ denote the affine $\Q_\ellp$-scheme parametrising continuous $U$-valued cocycles. For every~$g\in G$ there is an evaluation map $\ev_g\colon\Cycle^1(G,U)\to U$ which is a morphism of affine $\Q_\ellp$-schemes. We endow the affine ring of~$\Cycle^1(G,U)$ with the finest filtration making all the maps~$\ev_g$ filtered. This filtration then restricts to a filtration on $\O(\HH^1(G,U))\subseteq\O(\Cycle^1(G,U))$, making~$\HH^1(G,U)$ into a filtered affine $\Q_\ellp$-scheme.

This explicitly constructed filtration is the same as the one from Theorem~\ref{thm:representability}. This can be proved in two steps, the verification of which we leave to the interested reader. Firstly, one shows that $\Cycle^1(G,U)$, with the filtration constructed explicitly above, represents the functor $\Lambda\mapsto\Cycle^1(G,U(\Lambda))$ for~$\Lambda\in\WALG_{\Q_\ellp}$. Then one uses the description of the functor of points from Theorem~\ref{thm:representability} to show that the map $\Cycle^1(G,U)\twoheadrightarrow\HH^1(G,U)$ is split as a morphism of filtered affine $\Q_\ellp$-schemes. This ensures that the filtration on~$\O(\HH^1(G,U))$ from Theorem~\ref{thm:representability} is the restriction of the filtration on $\O(\Cycle^1(G,U))$, and hence the two definitions of the filtration agree.

\section{Filtrations on Selmer schemes}\label{s:selmer}

With the general machinery of filtered affine $\Q_\ellp$-schemes now set up, we come to the first of the two main strands of this work: endowing the local and global Selmer schemes appearing in the Chabauty--Kim method with filtrations induced from the weight filtration on the fundamental group. This is purely a matter of representation theory, so we will describe here the construction for a general pro-unipotent group~$U$, not just those arising from the fundamental group of a curve.

We treat two different kinds of Selmer schemes separately: local Bloch--Kato Selmer schemes in \S\ref{ss:local_selmer_p} and global Selmer schemes in \S\ref{ss:global_selmer}. Importantly, we will compute the Hilbert series of both kinds of Selmer schemes, and comparing the coefficients of these Hilbert series is what will ultimately allow us to control the algebraic functionals~$\alpha$ coming out of the Chabauty--Kim method.

For an introduction to the theory of global Selmer schemes, see \cite[\S2]{jen-etal:non-abelian_tate-shafarevich}.

\subsection{Local Bloch--Kato Selmer schemes}\label{ss:local_selmer_p}

Suppose to begin with that~$K_v$ is a finite extension of~$\Q_\ellp$, with absolute Galois group~$G_v$. We write $K_{v,0}$ for the maximal unramified subfield of~$K_v$. Suppose that we are given a $\Q_\ellp$-pro-unipotent group~$U$, endowed with a continuous action of~$G_v$ and a $G_v$-stable separated filtration
\[
1\subseteq\dots\subseteq\W_{-3}U\subseteq\W_{-2}U\subseteq\W_{-1}U=U
\]
in the sense of Example~\ref{ex:filtered_unipotent}, where each subgroup~$\W_{-n}U$ is of finite codimension in~$U$. We make moreover the following two assumptions on~$U$:
\begin{itemize}
	\item $U$ is pro-crystalline, in the sense that~$\Lie(U)$ is an inverse limit of crystalline representations; and
	\item the eigenvalues\footnote{Note that the crystalline Frobenius is not in general $K_{v,0}$-linear. By its eigenvalues, we simply mean its eigenvalues as a $\Q_\ellp$-linear automorphism of~$\D_\cris(\gr^\W_{-n}U)$. This is the same as the $f$th roots of the eigenvalues of the $K_{v,0}$-linear automorphism~$\varphi^f$, where $f=[K_{v,0}:\Q_\ellp]$.} of the crystalline Frobenius acting on~$\D_\cris(V_n)$ are all $\ellp$-Weil numbers of weight~$-n$, for all~$n>0$.
\end{itemize}
These two assumptions ensure that conditions~\ref{condn:H^0} and~\ref{condn:H^1} from \S\ref{ss:cohomology} are met, and hence the cohomology functor~$\HH^1(G_v,U)$ is representable (we ignore its filtration for the time being).

\smallskip

In the Chabauty--Kim method, one is interested not in the whole cohomology scheme $\HH^1(G_v,U)$ but in a certain subscheme $\HH^1_f(G_v,U)$ cut out by Bloch--Kato-style Selmer conditions.

\begin{definition}
	If $\Lambda\in\ALG_{\Q_\ellp}$ is a $\Q_\ellp$-algebra, then we define
	\[
	\HH^1_f(G_v,U(\Lambda)) := \ker\left(\HH^1(G_v,U(\Lambda))\to\HH^1(G_v,U(\B_\cris\otimes_{\Q_\ellp}\Lambda))\right) \,,
	\]
	where~$\B_\cris$ denotes Fontaine's ring of crystalline periods, with its usual Galois action. The assignment $\Lambda\mapsto\HH^1_f(G_v,U(\Lambda))$ is functorial in~$\Lambda$, hence defines a functor $\HH^1_f(G_v,U)\colon\ALG_{\Q_\ellp}\to\SET_*$. This is a subfunctor of the cohomology functor~$\HH^1(G_v,U)$.
\end{definition}

\begin{remark}\label{rmk:H^1_f_no_topology}
	Strictly speaking, in order to make sense of the above definition, we need to also specify a topology on the group $U(\B_\cris\otimes_{\Q_\ellp}\Lambda)$ for which the $G_v$-action and the map $U(\Lambda)\to U(\B_\cris\otimes_{\Q_\ellp}\Lambda)$ are continuous. However, $\HH^1_f(G_v,U(\Lambda))$ does not actually depend on the choice of topology: $\HH^1_f(G_v,U(\Lambda))$ consists of those classes of continuous $U(\Lambda)$-valued cocycles which are coboundaries of elements of~$U(\B_\cris\otimes_{\Q_\ellp}\Lambda)$.
\end{remark}

The following fact is foundational in the Chabauty--Kim method.

\begin{proposition}\label{prop:representability_of_H^1_f}
	$\HH^1_f(G_v,U)$ is representable by a closed $\Q_\ellp$-subscheme of the cohomology scheme~$\HH^1(G_v,U)$.
\end{proposition}

However, the proof of this assertion in \cite[p118]{minhyong:selmer}\footnote{Strictly speaking, the assertion in \cite{minhyong:selmer} is that $\HH^1_f(G_v,U)$ is representable by a subscheme of~$\HH^1(G_v,U)$, with no claim that it is closed. Nonetheless, we will show that it is closed.} seems to contain a gap. There, it is shown $\HH^1_f(G_v,U)$ is the image of a certain morphism of affine $\Q_\ellp$-schemes $\delta\colon\HH^0(G_v,U^{\B_\cris}/U)\to\HH^1(G_v,U)$, in the sense that it is the pointwise image of the corresponding morphism of functors of points. It is deduced from this that $\HH^1_f(G_v,U)$ is representable by a subscheme of $\HH^1(G_v,U)$. However, this step does not follow, since the pointwise image of a morphism is not the same as the scheme-theoretic image in general, and need not even be representable\footnote{For example, the pointwise image of the squaring map $\A^1_{\Q_\ellp}\to\A^1_{\Q_\ellp}$ is the functor sending a $\Q_\ellp$-algebra $\Lambda$ to the set of squares in~$\Lambda$, which is not representable.}.
\smallskip

Nonetheless, a proof of representability of~$\HH^1_f(G_v,U)$ can be extracted from the literature. In \cite[Proposition~1.4]{minhyong:tangential_localization}, Kim shows that there is an isomorphism\footnote{Here, we use a different convention for the quotient $\Fil^0\backslash\D_\dR(U)=\Fil^0\D_\dR(U)\backslash\D_\dR(U)$ from \cite{minhyong:selmer,minhyong:tangential_localization}, where the right quotient $\D_\dR(U)/\Fil^0\D_\dR(U)$ is used instead. One can translate between the two conventions via the isomorphism $ \Fil^0\backslash\D_\dR(U)\cong\D_\dR(U)/\Fil^0$ given by $u\mapsto u^{-1}$.}
\[
\log_\BK\colon\HH^1_f(G_v,U)\xrightarrow\sim\Res^{K_v}_{\Q_\ellp}\!\left(\Fil^0\backslash\D_\dR(U)\right)
\]
of functors $\ALG_{\Q_\ellp}\to\SET_*$, which we call the \emph{Bloch--Kato logarithm} by way of analogy with (the inverse of) \cite[Definition~3.10]{bloch-kato:tamagawa_numbers}. Since the right-hand side is the Weil restriction of the quotient of a $K_v$-pro-unipotent group by a subgroup, it is representable by an affine space (possibly of infinite dimension), and hence so too is~$\HH^1_f(G_v,U)$. Note, however, that this does not prove that the inclusion $\HH^1_f(G_v,U)\hookrightarrow\HH^1(G_v,U)$ is a closed immersion; for the proof of the full statement of Proposition~\ref{prop:representability_of_H^1_f}, see Remark~\ref{rmk:filtered_implies_unfiltered}.

\subsubsection{Weight filtration on~$\HH^1_f(G_v,U)$}\label{sss:H^1_f}

We want to explain how to enrich this construction to put a weight filtration on the affine ring of~$\HH^1_f(G_v,U)$. This will turn out to be the subspace filtration induced from~$\HH^1(G_v,U)$, but it is more convenient to define the filtration via its functor of points.

\begin{definition}
	If $\Lambda\in\WALG_{\Q_\ellp}$ is a filtered $\Q_\ellp$-algebra, then we define
	\[
	\HH^1_f(G_v,U(\Lambda)) := \ker\left(\HH^1(G_v,U(\Lambda))\to\HH^1(G_v,U(\B_\cris\otimes_{\Q_\ellp}\Lambda))\right) \,,
	\]
	where $U(\Lambda)$ and $U(\B_\cris\otimes_{\Q_\ellp}\Lambda)$ denote the $\Lambda$- and $\B_\cris\otimes_{\Q_\ellp}\Lambda$-points of~$U$ in the sense of \S\ref{sss:functors_of_points}. The filtration on $\B_\cris\otimes_{\Q_\ellp}\Lambda$ is the tensor product of the given filtration on~$\Lambda$ and the trivial filtration on~$\B_\cris$.
	
	The assignment $\Lambda\mapsto\HH^1_f(G_v,U(\Lambda))$ is functorial in~$\Lambda$, hence defines a functor $\HH^1_f(G_v,U)\colon\WALG_{\Q_\ellp}\to\SET_*$. This is a subfunctor of~$\HH^1(G_v,U)$.
\end{definition}

\begin{remark}\label{rmk:H^1_f_filtered_no_topology}
	As in Remark~\ref{rmk:H^1_f_no_topology}, the definition of~$\HH^1_f(G_v,U(\Lambda))$ does not actually depend on the topology on~$U(\B_\cris\otimes_{\Q_\ellp}\Lambda)$. However, this topology will be relevant at a few points in the following proofs, in which we adopt the following convention. If~$\B$ is a topological $\Q_\ellp$-algebra, then we topologise the group $U(\B\otimes_{\Q_\ellp}\Lambda)$ for any filtered $\Q_\ellp$-algebra~$\Lambda$ by identifying
	\[
	U(\B\otimes_{\Q_\ellp}\Lambda)\cong\varprojlim\left(\B\otimes_{\Q_\ellp}\W_0(\Lambda\otimes_{\Q_\ellp}\Lie(U_n))\right)
	\]
	via the logarithm isomorphism, and endowing the right-hand side with the inverse limit of the direct limit topologies on the free $\B$-modules $\B\otimes_{\Q_\ellp}\W_0(\Lambda\otimes_{\Q_\ellp}\Lie(U_n))$. In the particular case that~$\B=\Q_\ellp$ with the $\ellp$-adic topology, this recovers the topology on~$U(\Lambda)$ from \S\ref{sss:abstract_filtration}.
\end{remark}

We then have the following filtered analogue of Proposition~\ref{prop:representability_of_H^1_f}.

\begin{proposition}\label{prop:representability_of_H^1_f_filtered}
	The functor $\HH^1_f(G_v,U)\colon\WALG_{\Q_\ellp}\to\SET_*$ is representable by a closed $\Q_\ellp$-subscheme of $\HH^1(G_v,U)$, with the induced filtration.
\end{proposition}

\begin{remark}\label{rmk:filtered_implies_unfiltered}
	Similarly to in Remark~\ref{rmk:representability_compatibility}, once we know that the functor $\HH^1_f(G_v,U)\colon\WALG_{\Q_\ellp}\to\SET_*$ is representable, it immediately follows that its underlying affine $\Q_\ellp$-scheme represents the functor $\HH^1_f(G_v,U^\circ)\colon\ALG_{\Q_\ellp}\to\SET_*$ from Proposition~\ref{prop:representability_of_H^1_f}. Thus, Proposition~\ref{prop:representability_of_H^1_f_filtered} in particular gives us a proof of Proposition~\ref{prop:representability_of_H^1_f}. Of course, there is a more direct proof of Proposition~\ref{prop:representability_of_H^1_f} which avoids considering filtrations -- this is easily adapted from the proof we will give of Proposition~\ref{prop:representability_of_H^1_f_filtered} by ignoring the filtrations everywhere.
\end{remark}

In preparation for our proof of Proposition~\ref{prop:representability_of_H^1_f_filtered}, we note the following preparatory proposition, a filtered version of~\cite[Lemma~1]{minhyong:selmer}.

\begin{lemma}\label{lem:filtered_crystalline_path}
	Let $\Lambda$ be a filtered $\Q_\ellp$-algebra and let $\xi\colon G_v\to U(\Lambda)$ be a continuous cocycle whose class lies in~$\HH^1_f(G_v,U)$. Then there exists a \emph{unique} $\varphi$-invariant element $u_\cris\in U(\B_\cris\otimes_{\Q_\ellp}\Lambda)$ whose coboundary is~$\xi$. Here, $\varphi$ denotes the automorphism induced by the crystalline Frobenius on~$\B_\cris$.
	\begin{proof}
		Unicity is easy to see. Two elements $u,u'\in U(\B_\cris\otimes_{\Q_\ellp}\Lambda)$ represent the same cohomology class if and only if $u'u^{-1}$ lies in the $G_v$-fixed subgroup. If~$u$ and~$u'$ are in addition~$\varphi$-fixed, then $u'u^{-1}$ lies in the $\varphi,G_v$-fixed subgroup of $U(\B_\cris\otimes_{\Q_\ellp}\Lambda)$. But this group is an iterated extension of the groups $\B_\cris\otimes_{\Q_\ellp}\W_n\Lambda\otimes_{\Q_\ellp}V_n$, which have no non-identity $\varphi,G_v$-fixed elements by assumption on the weights of crystalline Frobenius. Hence $U(\B_\cris\otimes_{\Q_\ellp}\Lambda)$ also has no non-identity $\varphi,G_v$-fixed elements, so $u'u^{-1}=1$ and we have unicity.
		
		To prove existence, it suffices to prove that the composite $\xi_n\colon G_v\xrightarrow\xi U(\Lambda)\to U_n(\Lambda)$ is the coboundary of a $\varphi$-invariant element $u_{n,\cris}\in U_n(\B_\cris\otimes_{\Q_\ellp}\Lambda)$ for all~$n\geq0$. For, then unicity implies that the elements $u_{n,\cris}$ are mapped to one another under the maps $U_n\twoheadrightarrow U_{n-1}$, and hence define a $\varphi$-invariant element $u_\cris\in U(\B_\cris\otimes_{\Q_\ellp}\Lambda)=\varprojlim U_n(\B_\cris\otimes_{\Q_\ellp}\Lambda)$ whose coboundary is~$\xi$.
		
		We prove the existence of the elements $u_{n,\cris}$ by induction on~$n$, the base case~$n=0$ being trivial. Suppose then that we have a $\varphi$-invariant element $u_{n-1,\cris}\in U_{n-1}(\B_\cris\otimes_{\Q_\ellp}\Lambda)$ whose coboundary is~$\xi_{n-1}$. Since the class of~$\xi$ lies in~$\HH^1_f(G_v,U)$, there is an element $u_n\in U_n(\B_\cris\otimes_{\Q_\ellp}\Lambda)$ whose coboundary is~$\xi_n$. The image of~$u_n$ in~$U_{n-1}$ then has the same coboundary as~$u_{n-1,\cris}$, so that $w_{n-1}=u_{n-1,\cris}u_n^{-1}$ is a $G_v$-fixed element of $U_{n-1}(\B_\cris\otimes_{\Q_\ellp}\Lambda)$. Then, the map
		\[
		U_n(\B_\cris\otimes_{\Q_\ellp}\Lambda)^{G_v} \to U_{n-1}(\B_\cris\otimes_{\Q_\ellp}\Lambda)^{G_v}
		\]
		can be identified with the map $\D_\cris(\Lie(U_n))\otimes_{\Q_\ellp}\Lambda\to\D_\cris(\Lie(U_{n-1}))\otimes_{\Q_\ellp}\Lambda$, and hence is surjective by our assumption that~$\Lie(U)$ is pro-crystalline. Thus~$w_{n-1}$ is the image of a $G_v$-fixed point~$w_n\in U_n(\B_\cris\otimes_{\Q_\ellp}\Lambda)$. Replacing~$u_n$ by $w_nu_n$ if necessary, we may assume without loss of generality that $u_n$ maps to $u_{n-1,\cris}$ under the map $U_n\twoheadrightarrow U_{n-1}$.
		
		Now the cocycle~$\xi_n$ is invariant under the action of~$\varphi$ (it is valued in $U(\Lambda)$), so~$\xi_n$ is also the coboundary of~$\varphi(u_n)$. It follows that $\varphi(u_n)u_n^{-1}$ is a $G_v$-fixed element of $V_n(\B_\cris\otimes_{\Q_\ellp}\Lambda)=\B_\cris\otimes_{\Q_\ellp}\W_n\Lambda\otimes_{\Q_\ellp}V_n$. But the assumption on weights of Frobenius again implies that the endomorphism of $(\B_\cris\otimes_{\Q_\ellp}\W_n\Lambda\otimes_{\Q_\ellp}V_n)^{G_v}$ given by $\varphi-1$ is bijective; hence there is a $v_n\in (\B_\cris\otimes_{\Q_\ellp}\W_n\Lambda\otimes_{\Q_\ellp}V_n)^{G_v}$ such that $\varphi(v_n)-v_n=\varphi(u_n)u_n^{-1}$. It follows that the coboundary of $u_{n,\cris}:=v_n^{-1}u_n$ is $\xi_n$, and $u_{n,\cris}$ is~$\varphi$-fixed by construction. This completes the inductive proof of existence.
	\end{proof}
\end{lemma}

\begin{remark}
	Lemma~\ref{lem:filtered_crystalline_path} shows in particular that~$\HH^1_f(G_v,U(\Lambda))$ is also equal to the kernel of the map $\HH^1(G_v,U(\Lambda))\to\HH^1(G_v,U(\B_\cris^{\varphi=1}\otimes_{\Q_\ellp}\Lambda))$, which would usually be denoted $\HH^1_e(G_v,U(\Lambda))$.
\end{remark}

We are now ready to prove Proposition~\ref{prop:representability_of_H^1_f_filtered}. The observation underlying the proof is the following: if the functor $\Lambda\mapsto\HH^1(G_v,U(\B_\cris\otimes_{\Q_\ellp}\Lambda))$ were representable by a filtered affine $\Q_\ellp$-scheme, then $\HH^1_f(G_v,U)$ would be a closed subscheme of~$\HH^1(G_v,U)$ with the induced filtration, simply by virtue of being a kernel of a morphism of pointed filtered affine $\Q_\ellp$-schemes. It is not actually the case that this functor is representable, but with sufficient care, this idea yields a proof of Proposition~\ref{prop:representability_of_H^1_f_filtered}.

\begin{proof}[Proof of Proposition~\ref{prop:representability_of_H^1_f_filtered}]
	For~$n\geq0$, write $U_n:=U/\W_{-n-1}U$ and $V_n=\gr^\W_{-n}U$ as usual, and define
	\[
	\HH^1_{f,n}(G_v,U(\Lambda)) := \ker\left(\HH^1(G_v,U(\Lambda))\to\HH^1(G_v,U_n(\B_\cris\otimes_{\Q_\ellp}\Lambda))\right) \,.
	\]
	This is a subfunctor of $\HH^1(G_v,U)$. We will show the following two facts, which together imply Proposition~\ref{prop:representability_of_H^1_f_filtered}:
	\begin{enumerate}
		\item $\HH^1_{f,n}(G_v,U)$ is representable by a closed subscheme of $\HH^1(G_v,U)$ with the induced filtration for all~$n\geq0$; and
		\item $\HH^1_f(G_v,U)=\bigcap_n\HH^1_{f,n}(G_v,U)$ as subfunctors of~$\HH^1(G_v,U)$.
	\end{enumerate}

	For the second point, the left-to-right inclusion is immediate. Conversely, a $\Lambda$-valued point on the right-hand side is represented by a continuous cocycle~$\xi$ whose class maps into $\HH^1_f(G_v,U_n(\Lambda))$ for all~$n$. Lemma~\ref{lem:filtered_crystalline_path} ensures that there are unique elements $u_{n,\cris}\in U_n(\B_\cris^{\varphi=1}\otimes_{\Q_\ellp}\Lambda)$ whose coboundaries are congruent to~$\xi$ modulo~$\W_{-n-1}U$. Unicity implies that the~$u_{n,\cris}$ define an element $u\in U(\B_\cris^{\varphi=1}\otimes_{\Q_\ellp}\Lambda)=\varprojlim U_n(\B_\cris^{\varphi=1}\otimes_{\Q_\ellp}\Lambda)$ whose coboundary is~$\xi$, from which it follows that~$\xi\in\HH^1_f(G_v,U(\Lambda))$.
	
	\smallskip
	
	The first point we prove by induction on~$n$, the base case~$n=0$ being trivial. For the inductive step, we consider the central extension
	\begin{equation}\label{eq:crystalline_extension}\tag{$\ast$}
	1 \to V_n(\B_\cris^{\varphi=1}\otimes_{\Q_\ellp}\Lambda) \to U_n(\B_\cris^{\varphi=1}\otimes_{\Q_\ellp}\Lambda) \to U_{n-1}(\B_\cris^{\varphi=1}\otimes_{\Q_\ellp}\Lambda) \to 1 \,.
	\end{equation}
	We endow each term with a topology as in Remark~\ref{rmk:H^1_f_filtered_no_topology}.
	
	With respect to these topologies, \eqref{eq:crystalline_extension} is a $G_v$-equivariant central extension of topological groups with continuous $G_v$-action. It is moreover topologically split, since the exact sequence $1\to V_n\to U_n\to U_{n-1}\to1$ is split in the category of filtered affine schemes. The long exact sequence in non-abelian cohomology provides an action of $\HH^1(G_v,V_n(\B_\cris^{\varphi=1}\otimes_{\Q_\ellp}\Lambda))$ on $\HH^1(G_v,U_n(\B_\cris^{\varphi=1}\otimes_{\Q_\ellp}\Lambda))$, whose orbits are the fibres of the map $\HH^1(G_v,U_n(\B_\cris^{\varphi=1}\otimes_{\Q_\ellp}\Lambda))\to\HH^1(G_v,U_{n-1}(\B_\cris^{\varphi=1}\otimes_{\Q_\ellp}\Lambda))$. The stabiliser of the base point of $\HH^1(G_v,U_n(\B_\cris^{\varphi=1}\otimes_{\Q_\ellp}\Lambda))$ is the image of the coboundary map $\HH^0(G_v,U_{n-1}(\B_\cris^{\varphi=1}\otimes_{\Q_\ellp}\Lambda))\to\HH^1(G_v,V_n(\B_\cris^{\varphi=1}\otimes_{\Q_\ellp}\Lambda))$. However, $\HH^0(G_v,U_{n-1}(\B_\cris^{\varphi=1}\otimes_{\Q_\ellp}\Lambda))$ is the $\varphi,G_v$-invariant subgroup of~$U_{n-1}(\B_\cris\otimes_{\Q_\ellp}\Lambda)$, which is trivial by our assumptions. It follows that the action on the base point gives an identification of~$\HH^1(G_v,V_n(\B_\cris^{\varphi=1}\otimes_{\Q_\ellp}\Lambda))$ with the kernel of the morphism $\HH^1(G_v,U_n(\B_\cris^{\varphi=1}\otimes_{\Q_\ellp}\Lambda))\to\HH^1(G_v,U_{n-1}(\B_\cris^{\varphi=1}\otimes_{\Q_\ellp}\Lambda))$.
	
	Now if $[\xi]\in\HH^1_{f,n-1}(G_v,U(\Lambda))$, then the image of $[\xi]$ in $\HH^1(G_v,U_n(\B_\cris^{\varphi=1}\otimes_{\Q_\ellp}\Lambda))$ lies in the kernel of $\HH^1(G_v,U_n(\B_\cris^{\varphi=1}\otimes_{\Q_\ellp}\Lambda))\to\HH^1(G_v,U_{n-1}(\B_\cris^{\varphi=1}\otimes_{\Q_\ellp}\Lambda))$ by definition. We have thus described a map
	\[
	\HH^1_{f,n-1}(G_v,U(\Lambda))\to\HH^1(G_v,V_n(\B_\cris^{\varphi=1}\otimes_{\Q_\ellp}\Lambda)) \,,
	\]
	whose kernel is easily checked to be $\HH^1_{f,n}(G_v,U(\Lambda))$. This map is natural in~$\Lambda$, so induces a morphism
	\begin{equation}\label{eq:crystalline_cutout}\tag{$\dag$}
	\HH^1_{f,n-1}(G_v,U)\to\HH^1(G_v,V_{n,\B_\cris^{\varphi=1}})
	\end{equation}
	of functors, whose kernel is $\HH^1_{f,n}(G_v,U)$, where $V_{n,\B_\cris^{\varphi=1}}$ is shorthand for the functor $\Lambda\mapsto V_n(\B_\cris^{\varphi=1}\otimes_{\Q_\ellp}\Lambda) = \B_\cris^{\varphi=1}\otimes_{\Q_\ellp}\W_n\Lambda\otimes_{\Q_\ellp}V_n$.
	
	Now an argument similar to that of Proposition~\ref{prop:abelian_cohomology} shows that
	\[
	\HH^1(G_v,\B_\cris^{\varphi=1}\otimes_{\Q_\ellp}\W_n\Lambda\otimes_{\Q_\ellp}V_n)=\W_n\Lambda\otimes_{\Q_\ellp}\HH^1(G_v,\B_\cris^{\varphi=1}\otimes_{\Q_\ellp}V_n) \,.
	\]
	Thus, by an argument similar to the proof of Corollary~\ref{cor:abelian_representability}, $\HH^1(G_v,V_{n,\B_\cris^{\varphi=1}})$ is a subfunctor of a representable functor.
	
	Thus, assuming inductively that $\HH^1_{f,n-1}(G_v,U)$ is representable by a closed subscheme of $\HH^1(G_v,U)$ with the induced filtration, then~\eqref{eq:crystalline_cutout} is a morphism from a representable functor to a subfunctor of a representable functor. It follows that its kernel $\HH^1_{f,n}(G_v,U)$ is representable by a closed subscheme of~$\HH^1_{f,n-1}(G_v,U)$ with the induced filtration. This completes the inductive step, and thus the proof of Proposition~\ref{prop:representability_of_H^1_f_filtered}.
\end{proof}

\subsubsection{Compatibility with the Bloch--Kato logarithm}\label{sss:bk_log}

Now let~$\B_\dR$ denote Fontaine's ring of de Rham periods, and let~$\D_\dR(U)$ denote the filtered pro-unipotent group over~$K_v$ representing the functor
\[
\Lambda\mapsto U(\B_\dR\otimes_{K_v}\Lambda)^{G_v}
\]
from filtered $K_v$-algebras to pointed sets, where~$\B_\dR$ is endowed with the trivial filtration ($\W_n\B_\dR=\B_\dR$ for all~$n$). This functor is indeed representable, by the filtered pro-unipotent group associated to the filtered pro-nilpotent Lie algebra $\D_\dR(\Lie(U))$. There is a subgroup scheme $\Fil^0\D_\dR(U)$ representing the subfunctor $\Lambda\mapsto U(\B_\dR^+\otimes_{K_v}\Lambda)^{G_v}$, and we write~$\Res^{K_v}_{\Q_\ellp}\left(\Fil^0\backslash\D_\dR(U)\right)$ for the filtered affine $\Q_\ellp$-scheme representing the functor
\[
\Lambda\mapsto U(\B_\dR^+\otimes_{\Q_\ellp}\Lambda)\backslash U(\B_\dR\otimes_{\Q_\ellp}\Lambda) \,.
\]
This functor is representable by Lemma~\ref{lem:quotient_by_unipotent_subgroup}, since it is the quotient of the filtered $\Q_\ellp$-pro-unipotent group corresponding to the pro-nilpotent Lie algebra $\Res^{K_v}_{\Q_\ellp}\D_\dR(\Lie(U))$ by the subgroup corresponding to the subalgebra $\Res^{K_v}_{\Q_\ellp}\Fil^0\D_\dR(\Lie(U))$.
\smallskip

In \cite[Proposition~1.4]{minhyong:tangential_localization}, Kim shows that there is an isomorphism
\begin{equation}\label{eq:bk_log}
\log_\BK\colon\HH^1_f(G_v,U)\xrightarrow\sim\Res^{K_v}_{\Q_\ellp}\!\left(\Fil^0\backslash\D_\dR(U)\right)
\end{equation}
on the level of underlying (non-filtered) $\Q_\ellp$-schemes. Our aim here is to show that this isomorphism is filtered, for the filtrations defined above.

\begin{proposition}\label{prop:bk_log_is_filtered_iso}
	The Bloch--Kato logarithm~\eqref{eq:bk_log} is an isomorphism of filtered affine $\Q_\ellp$-schemes.
\end{proposition}

As a consequence, we can compute the Hilbert series of the local Bloch--Kato Selmer scheme~$\HH^1_f(G_v,U)$.

\begin{corollary}\label{cor:hilbert_series_of_local_selmer}
	There is a (non-canonical) isomorphism
	\[
	\HH^1_f(G_v,U)\cong\prod_{n>0}\HH^1_f(G_v,V_n)
	\]
	of filtered affine $\Q_\ellp$-schemes, where the vector space $\HH^1_f(G_v,V_n)$ is given the filtration supported in filtration-degree~$-n$. In particular, we have
	\[
	\HS_{\HH^1_f(G_v,U)}=\prod_{n>0}(1-t^n)^{-\dim_{\Q_\ellp}\HH^1_f(G_v,V_n)} \,.
	\]
	\begin{proof}
		Proposition~\ref{prop:bk_log_is_filtered_iso} and Lemma~\ref{lem:quotient_by_unipotent_subgroup} together imply that~$\HH^1_f(G_v,U)$ is isomorphic to~$V$, where~$V$ is a filtered $\Q_\ellp$-linear complement to $\Fil^0\D_\dR(\Lie(U))$ inside $\D_\dR(\Lie(U))$. By choosing a splitting of the filtration on~$V$ we have
		\[
		V\cong\prod_{n>0}\gr^\W_{-n}V \,.
		\]
		But we have
		\begin{align*}
			\dim_{\Q_\ellp}\gr^\W_{-n}V &= \dim_{\Q_\ellp}\gr^\W_{-n}\D_\dR(\Lie(U))-\dim_{\Q_\ellp}\Fil^0\gr^\W_{-n}\D_\dR(\Lie(U)) \\
			&= \dim_{\Q_\ellp}\HH^1_f(G_v,V_n)
		\end{align*}
		by \cite[Corollary~3.8.4]{bloch-kato:tamagawa_numbers}. Hence~$\gr^\W_{-n}V\simeq\HH^1_f(G_v,V_n)$ as $\Q_\ellp$-vector spaces with filtrations supported in filtration-degree~$-n$. The result follows.
	\end{proof}
\end{corollary}

For the proof of Proposition~\ref{prop:bk_log_is_filtered_iso}, we will construct both the Bloch--Kato logarithm and its inverse as morphisms of filtered $\Q_\ellp$-schemes and verify that these are mutually inverse. To construct~$\log_\BK$, we use the following proposition, a filtered version of the discussion on \cite[p.~118]{minhyong:selmer}.

\begin{lemma}\label{lem:filtered_de_rham_path}
	Let $\Lambda$ be a filtered $\Q_\ellp$-algebra and let $\xi\colon G_v\to U(\Lambda)$ be a continuous cocycle whose class lies in~$\HH^1_f(G_v,U)$. Then there exists an element $u_\dR\in U(\B_\dR^+\otimes_{\Q_\ellp}\Lambda)$ whose coboundary is~$\xi$.
	\begin{proof}
		We follow a similar strategy to the proof of Lemma~\ref{lem:filtered_crystalline_path}, except that we are not proving -- and cannot use -- unicity. We show that there exists a compatible system of elements $u_{n,\dR}\in U_n(\B_\dR^+\otimes_{\Q_\ellp}\Lambda)$ such that the coboundary of $u_{n,\dR}$ is~$\xi$ modulo $\W_{-n-1}U$. These elements then define an element $u_\dR\in U(\B_\dR^+\otimes_{\Q_\ellp}\Lambda)=\varprojlim U_n(\B_\dR^+\otimes_{\Q_\ellp}\Lambda)$ whose coboundary is~$\xi$, as desired.
		
		We begin by taking $u_{0,\dR}$ the unique element of the trivial group $U_0(\B_\dR^+\otimes_{\Q_\ellp}\Lambda)$, which clearly works. Suppose that we have constructed $u_{n-1,\dR}$. We construct the element $u_{n,\dR}$ as follows. To begin with, choose any element $u_n\in U_n(\B_\dR^+\otimes_{\Q_\ellp}\Lambda)$ lifting~$u_{n-1,\dR}$. The twisted cocycle
		\[
		\xi_n\colon g\mapsto u_n\xi(g)g(u_n)^{-1}
		\]
		is then valued in $V_n(\B_\dR^+\otimes_{\Q_\ellp}\Lambda)$, so defines a class in $\HH^1(G_v,V_n(\B_\dR^+\otimes_{\Q_\ellp}\Lambda))$. Here, we adopt conventions on topologies as in Remark~\ref{rmk:H^1_f_filtered_no_topology}.
		
		Now the image of~$[\xi_n]$ in $\HH^1(G_v,U_n(\B_\dR\otimes_{\Q_\ellp}\Lambda))$ is trivial (since $[\xi_n]=[\xi]\in\HH^1_f(G_v,U_n(\Lambda))$). We will show that the map
		\begin{equation}\label{eq:filtered_de_rham_path_lifting}\tag{$\ast$}
		\HH^1(G_v,V_n(\B_\dR^+\otimes_{\Q_\ellp}\Lambda))\to\HH^1(G_v,U_n(\B_\dR\otimes_{\Q_\ellp}\Lambda))
		\end{equation}
		has trivial kernel. This implies the proposition, since then $\xi_n$ is the coboundary of some element $v_n\in V_n(\B_\dR^+\otimes_{\Q_\ellp}\Lambda)$, whence $\xi$ is the coboundary of $u_{n,\dR}:=v_nu_n\in U_n(\B_\dR^+\otimes_{\Q_\ellp}\Lambda)$ modulo~$\W_{-n-1}U$ as desired.
		
		\smallskip
		
		To prove that~\eqref{eq:filtered_de_rham_path_lifting} has trivial kernel, we factor it as the composite
		\[
		\HH^1(G_v,V_n(\B_\dR^+\otimes_{\Q_\ellp}\Lambda))\to\HH^1(G_v,V_n(\B_\dR\otimes_{\Q_\ellp}\Lambda))\to\HH^1(G_v,U_n(\B_\dR\otimes_{\Q_\ellp}\Lambda))
		\]
		and show that both maps have trivial kernel. The left-hand map is identified as the map
		\[
		\W_n\Lambda\otimes_{\Q_\ellp}\HH^1(G_v,\B^+_\dR\otimes_{\Q_\ellp}V_n) \to \W_n\Lambda\otimes_{\Q_\ellp}\HH^1(G_v,\B_\dR\otimes_{\Q_\ellp}V_n) \,,
		\]
		which is injective by \cite[Lemma~3.8.1]{bloch-kato:tamagawa_numbers}.
		
		For the right-hand map, we consider the central extension
		\[
		1 \to V_n(\B_\dR\otimes_{\Q_\ellp}\Lambda) \to U_n(\B_\dR\otimes_{\Q_\ellp}\Lambda) \to U_{n-1}(\B_\dR\otimes_{\Q_\ellp}\Lambda) \to 1 \,,
		\]
		which is topologically split. The fact that $\Lie(U_n)$ is de Rham implies that the map
		\[
		\HH^0(G_v,U_n(\B_\dR\otimes_{\Q_\ellp}\Lambda)) \to \HH^0(G_v,U_{n-1}(\B_\dR\otimes_{\Q_\ellp}\Lambda))
		\]
		is surjective, and hence $\HH^1(G_v,V_n(\B_\dR\otimes_{\Q_\ellp}\Lambda))\to\HH^1(G_v,U_n(\B_\dR\otimes_{\Q_\ellp}\Lambda))$ has trivial kernel, as claimed.
	\end{proof}
\end{lemma}

Using this lemma, the Bloch--Kato logarithm is defined as follows. Given a class $[\xi]\in\HH^1_f(G_v,U(\Lambda))$, we choose a representing cocycle~$\xi$. According to Lemmas~\ref{lem:filtered_crystalline_path} and~\ref{lem:filtered_de_rham_path}, we have elements $u_\cris\in U(\B_\cris^{\varphi=1}\otimes_{\Q_\ellp}\Lambda)$ and $u_\dR\in U(\B_\dR^+\otimes_{\Q_\ellp}\Lambda)$ whose coboundary is~$\xi$. It follows that the element $u_\dR u_\cris^{-1}\in U(\B_\dR\otimes_{\Q_\ellp}\Lambda)$ is $G_v$-fixed. It is then easy to check that the class of $u_\dR u_\cris^{-1}\in U(\B_\dR^+\otimes_{\Q_\ellp}\Lambda)^{G_v}\backslash U(\B_\dR\otimes_{\Q_\ellp}\Lambda)^{G_v}$ is independent of the choice of~$u_\dR$ and the choice of representing cocycle~$\xi$. The construction $\xi\mapsto u_\dR u_\cris^{-1}$ thus provides a map
\[
\HH^1_f(G_v,U(\Lambda))\to U(\B_\dR^+\otimes_{\Q_\ellp}\Lambda)^{G_v}\backslash U(\B_\dR\otimes_{\Q_\ellp}\Lambda)^{G_v} \,,
\]
natural in $\Lambda\in\WALG_{\Q_\ellp}$, and hence a morphism
\[
\log_\BK\colon\HH^1_f(G_v,U) \to \Res^{K_v}_{\Q_\ellp}\!\left(\Fil^0\backslash\D_\dR(U)\right)
\]
of filtered affine~$\Q_\ellp$-schemes. This map is the Bloch--Kato logarithm.

\begin{remark}
	The Bloch--Kato logarithm constructed above is the same as that from \cite[p.~119]{minhyong:selmer}, \cite[Proposition~1.4]{minhyong:tangential_localization}, once one forgets the filtrations. If~$P$ is a $G_v$-equivariant $U$-torsor over a $\Q_\ellp$-algebra~$\Lambda$ whose class lies in $\HH^1_f(G_v,U(\Lambda))$, then $\D_\dR(P):=\Spec(\D_\dR(\O(P)))$ is an admissible $\D_\dR(U)$-torsor over $K_v\otimes_{\Q_\ellp}\Lambda$, in the sense of \cite[p.~104]{minhyong:selmer}. In \cite{minhyong:selmer}, the Bloch--Kato logarithm is defined to be the map sending the class of $P$ to the element $\gamma_\dR^{-1}\gamma_\cris\in\Fil^0\backslash\D_\dR(U)(\Lambda)$, where~$\gamma_\dR\in\Fil^0\D_\dR(P)(\Lambda)$ and~$\gamma_\cris\in\D_\dR(P)(\Lambda)$ is the unique Frobenius-invariant element.
	
	If~$\xi$ is a cocycle representing the class of~$P$, then $P$ is isomorphic to~$U_\Lambda$ with the twisted $G_v$-action $g\colon u\mapsto\xi(g)\cdot g(u)$. Under this identification, the elements~$\gamma_\cris\in P(\B_\cris^{\varphi=1}\otimes_{\Q_\ellp}\Lambda)^{G_v}$ and~$\gamma_\dR\in P(\B_\dR^+\otimes_{\Q_\ellp}\Lambda)^{G_v}$ are equal to~$u_\cris^{-1}\in U(\B_\cris^{\varphi=1}\otimes_{\Q_\ellp}\Lambda)$ and~$u_\dR^{-1}\in U(\B_\dR^+\otimes_{\Q_\ellp}\Lambda)$ for elements~$u_\cris$ and~$u_\dR$ whose coboundaries are both~$\xi$. Thus $\gamma_\dR^{-1}\gamma_\cris=u_\dR u_\cris^{-1}$, so our definition of the Bloch--Kato logarithm above agrees with that in \cite{minhyong:selmer}.
\end{remark}

Next, we construct the inverse of the Bloch--Kato logarithm, as a morphism of filtered affine $\Q_\ellp$-schemes. Here, we follow a different approach to \cite[Proposition~1.4]{minhyong:tangential_localization}, taking a strategy closer to that of \cite[\S6]{me:thesis}. The construction revolves around the following result.

\begin{lemma}\label{lem:non-abelian_fundamental_sequence}
	For every filtered $\Q_\ellp$-algebra~$\Lambda$, the multiplication map
	\[
	U(\B_\dR^+\otimes_{\Q_\ellp}\Lambda) \times U(\B_\cris^{\varphi=1}\otimes_{\Q_\ellp}\Lambda) \to U(\B_\dR\otimes_{\Q_\ellp}\Lambda)
	\]
	is surjective, and
	\[
	U(\B_\dR^+\otimes_{\Q_\ellp}\Lambda) \cap U(\B_\cris^{\varphi=1}\otimes_{\Q_\ellp}\Lambda) = U(\Lambda) \,,
	\]
	where the intersection is taken inside $U(\B_\dR\otimes_{\Q_\ellp}\Lambda)$. Moreover, the topology on~$U(\Lambda)$ is the subspace topology induced from~$U(\B_\dR\otimes_{\Q_\ellp}\Lambda)$, where the topology on the latter is as in Remark~\ref{rmk:H^1_f_filtered_no_topology}.
	\begin{proof}
		We work in the Lie algebra~$\Lie(U)$. The fundamental exact sequence \cite[(1.17.1)]{bloch-kato:tamagawa_numbers} implies that $\Lie(U)(\B_\dR^+\otimes_{\Q_\ellp}\Lambda)$ and $\Lie(U)(\B_\cris^{\varphi=1}\otimes_{\Q_\ellp}\Lambda)$ together span $\Lie(U)(\B_\dR\otimes_{\Q_\ellp}\Lambda)$ and their intersection is~$\Lie(U)(\Lambda)$. An argument similar to the proof of Lemma~\ref{lem:quotient_by_unipotent_subgroup} then shows surjectivity of the multiplication map. The statement regarding the intersection is also immediate.
		
		For the assertion regarding topologies, we observe that the topology on~$\B_\dR$ restricts to the $\ellp$-adic topology on~$\Q_\ellp$, so $\W_0(\Lambda\otimes_{\Q_\ellp}\Lie(U_n))\subseteq\B_\dR\otimes_{\Q_\ellp}\W_0(\Lambda\otimes_{\Q_\ellp}\Lie(U_n))$ has the subspace topology with respect to the free $\Q_\ellp$- and $\B_\dR$-module topologies on either side. Taking an inverse limit, we see that~$U(\Lambda)\subseteq U(\B_\dR\otimes_{\Q_\ellp}\Lambda)$ has the subspace topology.
	\end{proof}
\end{lemma}

We now construct the inverse of the Bloch--Kato logarithm~$\log_\BK$. Given an element $u\in U(\B_\dR\otimes_{\Q_\ellp}\Lambda)^{G_v}$, we have, courtesy of Lemma~\ref{lem:non-abelian_fundamental_sequence}, elements $u_\cris\in U(\B_\cris^{\varphi=1}\otimes_{\Q_\ellp}\Lambda)$ and $u_\dR\in U(\B_\dR^+\otimes_{\Q_\ellp}\Lambda)$ such that $u=u_\dR u_\cris^{-1}$. Since~$u$ is $G_v$-fixed by assumption, it follows that the coboundaries of~$u_\cris$ and~$u_\dR$ are equal. By construction, this coboundary is a continuous cocycle~$\xi$ taking values in $U(\B_\dR^+\otimes_{\Q_\ellp}\Lambda) \cap U(\B_\cris^{\varphi=1}\otimes_{\Q_\ellp}\Lambda) = U(\Lambda)$. Moreover, since~$\xi$ is the coboundary of the element~$u_\cris\in U(\B_\cris\otimes_{\Q_\ellp}\Lambda)$, it follows that the class of~$\xi$ lies in~$\HH^1_f(G_v,U(\Lambda))$.

Now the class of~$\xi$ is easily checked to be independent of the choices of~$u_\cris$ and~$u_\dR$, hence depends only on~$u$. The construction~$u\mapsto\xi$ thus gives a map
\[
U(\B_\dR\otimes_{\Q_\ellp}\Lambda)^{G_v} \to \HH^1_f(G_v,U(\Lambda)) \,,
\]
natural in~$\Lambda\in\WALG_{\Q_\ellp}$. It is easy to check that this map is invariant under the right-multiplication action of $U(\B_\dR^+\otimes_{\Q_\ellp}\Lambda)^{G_v}$ on the domain, and hence we obtain a morphism
\[
\Res^{K_v}_{\Q_\ellp}\!\left(\Fil^0\backslash\D_\dR(U)\right) \to \HH^1_f(G_v,U)
\]
of filtered affine $\Q_\ellp$-schemes, which is easily seen to be inverse to the Bloch--Kato logarithm. This concludes the proof of Proposition~\ref{prop:bk_log_is_filtered_iso}.\qed

\subsection{Global Selmer schemes}\label{ss:global_selmer}

Now we come to the corresponding global theory. We fix a number field~$K$, and write~$G_K$ for its absolute Galois group. Suppose that we are given a $\Q_\ellp$-pro-unipotent group~$U$, endowed with a continuous action of~$G_K$ and a $G_K$-stable separated filtration
\[
1\subseteq\dots\subseteq\W_{-3}U\subseteq\W_{-2}U\subseteq\W_{-1}U=U
\]
in the sense of Example~\ref{ex:filtered_unipotent}, where each subgroup~$\W_{-n}U$ is of finite codimension in~$U$. We make moreover the following assumptions on~$U$:
\begin{itemize}
	\item the $G_K$-action on~$U$ is ramified at only finitely many places of~$K$;
	\item $U$ is pro-crystalline at every place~$v\mid\ellp$, in the sense that~$\Lie(U)$ is an inverse limit of representations which are crystalline at all~$v\mid\ellp$; and
	\item for every finite place~$v$ and all~$n\in\N$, the $G_v$-invariant subgroup $\HH^0(G_v,V_n)=0$ is zero, where~$G_v$ denotes a decomposition group at~$v$ and $V_n=\gr^\W_{-n}U$ as usual.
\end{itemize}
In other words, $U$ is finitely ramified and the restriction of~$U$ to the decomposition group~$G_v$ at a finite place~$v$ satisfies the assumptions of \S\ref{ss:local_selmer_p} whenever $v\mid\ellp$, and satisfies the assumptions of \S\ref{ss:cohomology} whenever $v\nmid\ellp$. This ensures that the local cohomology functors $\HH^1(G_v,U)$ are all representable, as are the subfunctors $\HH^1_f(G_v,U)$ whenever~$v\mid\ellp$. It also ensures that the subfunctors $\HH^1_\nr(G_v,U)\subseteq\HH^1(G_v,U)$ consisting of the unramified cohomology classes are representable for every $v\nmid\ellp\infty$, albeit in a trivial way.

\begin{lemma}\label{lem:unramified_trivial}
	Let~$v\nmid\ellp\infty$. Then the kernel of the restriction map $\HH^1(G_v,U)\to\HH^1(I_v,U)$ is trivial, i.e.\ consists exactly of the base point of~$\HH^1(G_v,U)$.
	\begin{proof}
		It follows from non-abelian inflation--restriction \cite[Lemma in~\S2.8]{jen-etal:non-abelian_tate-shafarevich} that the kernel of the restriction map is the cohomology functor $\HH^1(G_v/I_v,U^{I_v})$. We endow the subgroup-scheme $U^{I_v}\subseteq U$ with the restriction of the filtration~$\W$ on~$U$.
		
		Our assumptions ensure that~$1$ is not an eigenvalue of the Frobenius~$\varphi$ acting on $\Lie(U)^{I_v}$, so that~$\varphi$ acts without fixed points on $\gr^\W_{-n}(U^{I_v})$ for all~$n$. Hence we have
		\begin{align*}
			\HH^0(G_v/I_v,\gr^\W_{-n}(U^{I_v})) &= \ker((\varphi-1)|_{\gr^\W_{-n}(U^{I_v})}) = 0 \,, \\
			\HH^1(G_v/I_v,\gr^\W_{-n}(U^{I_v})) &= \coker((\varphi-1)|_{\gr^\W_{-n}(U^{I_v})}) = 0 \,.
		\end{align*}
		Thus the cohomology functor $\HH^1(G_v/I_v,U^{I_v})$ is representable by \cite[Proposition~2]{minhyong:siegel} and is trivial by Corollary~\ref{cor:cohomology_hilbert_series_bound}.
	\end{proof}
\end{lemma}

Now we define a global Selmer scheme relative to a choice of closed subscheme $\lSel_v\subseteq\HH^1(G_v,U)$ for each place $v\nmid\ellp\infty$, such that~$\lSel_v=\{*\}$ is just the basepoint of $\HH^1(G_v,U)$ for all but finitely many~$v$. We refer to such a collection $\lSel=(\lSel_v)_{v\nmid\ellp\infty}$ as a \emph{Selmer structure} for~$U$, by way of analogy with \cite[\S1.3c]{loeffler-zerbes:euler_systems}.

\begin{definition}
	Let~$\lSel=(\lSel_v)_{v\nmid\ellp\infty}$ be a Selmer structure for~$U$. Restriction to the decomposition groups at each finite place provides a morphism
	\[
	\HH^1(G_K,U) \to \prod_{v\nmid\infty}\HH^1(G_v,U) \,,
	\]
	of functors and we define the \emph{global Selmer scheme} $\Sel_{\lSel,U}\subseteq\HH^1(G_K,U)$ to be the preimage of the closed subscheme $\prod_{v\mid\ellp}\HH^1_f(G_v,U)\times\prod_{v\nmid\ellp\infty}\lSel_v\subseteq\prod_{v\nmid\infty}\HH^1(G_v,U)$ under this map.
\end{definition}

\begin{example}\label{ex:naive_selmer}
	Suppose that~$T_0$ is a finite set of places, containing~$S$, all places dividing~$\infty$ and all places of bad reduction, and not containing any prime above $\ellp$. One choice of Selmer structure is given by taking $\lSel_v=\HH^1(G_v,U)$ for $v\in T_0$ and $\lSel_v=\{*\}$ otherwise. The corresponding Selmer scheme is denoted $\HH^1_{f,T_0}(G_K,U)$, and consists of those non-abelian cohomology classes which are unramified outside~$T_0\cup\{v\mid\ellp\}$ and crystalline at all~$\ellp$-adic places.
\end{example}

There is a small subtlety in the definition of the global Selmer scheme, in that the global cohomology functor $\HH^1(G_K,U)$ need not be representable. Nonetheless, the global Selmer scheme is representable.

\begin{proposition}\label{prop:global_selmer_representability}
	The global Selmer scheme $\Sel_{\lSel,U}$ is representable by an affine $\Q_\ellp$-scheme, which is of finite type if~$U$ is finite-dimensional.
	\begin{proof}
		We follow the argument of \cite[Proposition in~\S2.8]{jen-etal:non-abelian_tate-shafarevich}. Let~$T$ be a set of places of~$K$, containing all places dividing~$\ellp\infty$, all places where the Galois action on~$U$ ramifies, and all places where~$\lSel_v\neq\{*\}$. We write~$G_{K,T}$ for the largest quotient of~$G_K$ unramified outside~$T$, so that the $G_K$-action on~$U$ factors through~$G_{K,T}$ by assumption. The cohomology functor $\HH^1(G_{K,T},U)$ is a subfunctor of $\HH^1(G_K,U)$ by inflation--restriction \cite[Lemma in~\S2.8]{jen-etal:non-abelian_tate-shafarevich}, and $\HH^1(G_{K,T},U)$ is representable since each $\HH^1(G_{K,T},V_n)$ is finite-dimensional by \cite[Theorem~8.3.20(i)]{neukirch-wingberg-schmidt:cohomology_of_number_fields}.
		
		It thus suffices to prove that $\Sel_{\lSel,U}\subseteq\HH^1(G_{K,T},U)$ as subfunctors of~$\HH^1(G_K,U)$, for then $\Sel_{\lSel,U}$ is the preimage of the closed subscheme $\prod_{v\mid\ellp}\HH^1_f(G_v,U)\times\prod_{v\nmid\ellp\infty}\lSel_v$ under the morphism $\HH^1(G_{K,T},U)\to\prod_{v\nmid\infty}\HH^1(G_v,U)$ of affine $\Q_\ellp$-schemes, so is representable by a closed subscheme of $\HH^1(G_{K,T},U)$. To show this, suppose that~$\Lambda$ is a $\Q_\ellp$-algebra and $\xi\colon G_K\to U(\Lambda)$ is a continuous cocycle whose class lies in $\Sel_{\lSel,U}(\Lambda)$. For every place~$v\notin T$, the restriction $\xi|_{I_v}$ to inertia at~$v$ represents the basepoint in the set $\HH^1(I_v,U(\Lambda))=\Hom(I_v,U(\Lambda))$, hence we have $\xi|_{I_v}=1$.
		
		Now it follows from the cocycle condition that the intersection of the kernel of~$\xi$ with the kernel of the action map $G_K\to\Aut(U(\Lambda))$ is a closed normal subgroup of~$G_K$. We have just shown that the inertia group~$I_v$ is contained in this subgroup whenever~$v\notin T$, so this subgroup contains the kernel of $G_K\twoheadrightarrow G_{K,T}$. Hence~$\xi$ factors through~$G_{K,T}$, so $[\xi]\in\HH^1(G_{K,T},U(\Lambda))$. Thus we have shown that $\Sel_{\lSel,U}\subseteq\HH^1(G_{K,T},U)$, which completes the proof.
	\end{proof}
\end{proposition}

\subsubsection{Weight filtration on~$\Sel_{\lSel,U}$}\label{sss:global_selmer_filtration}

Now we endow the global Selmer scheme $\Sel_{\lSel,U}$ with a filtration, induced from that on~$U$. The easiest way to construct this is to note that the proof of Proposition~\ref{prop:global_selmer_representability} shows that~$\Sel_{\lSel,U}$ is a closed subscheme of~$\HH^1(G_{K,T},U)$, and so we can endow it with the induced filtration. As always, we can also characterise this filtration in terms of the functor it represents.

\begin{lemma}\label{lem:filtration_on_global_selmer}
	Let~$\lSel=(\lSel_v)_{v\nmid\ellp\infty}$ be a Selmer structure for~$U$. Endow each~$\lSel_v$ with the induced filtration as a closed subscheme of~$\HH^1(G_v,U)$ and endow the global Selmer scheme~$\Sel_{\lSel,U}$ with the induced filtration as a closed subscheme of~$\HH^1(G_{K,T},U)$, where the set~$T$ is as in the proof of Proposition~\ref{prop:global_selmer_representability}. Then the square
	\begin{center}
		\begin{tikzcd}
			\Sel_{\lSel,U} \arrow[r,hook]\arrow[d]\arrow[dr, phantom, "\scalebox{1.5}{$\lrcorner$}" , very near start, color=black] & \HH^1(G_{K,T},U) \arrow[d] \\
			\prod_{v\mid\ellp}\HH^1_f(G_v,U)\times\prod_{v\nmid\ellp\infty}\lSel_v \arrow[r,hook] & \prod_{v\nmid\infty}\HH^1(G_v,U)
		\end{tikzcd}
	\end{center}
	is a pullback square in the category of filtered affine~$\Q_\ellp$-schemes. In particular, this filtration on~$\Sel_{\lSel,U}$ is independent of the choice of~$T$.
	\begin{proof}
		This is in fact an abstract statement about filtered affine~$\Q_\ellp$-schemes: if $f\colon X\to Y$ is a morphism of filtered affine~$\Q_\ellp$-schemes and $Z\subseteq Y$ is a closed subscheme with the induced filtration, then the induced map $X\times_YZ\hookrightarrow X$ is the inclusion of a closed subscheme with the induced filtration. This is easily verified on the level of affine rings.
	\end{proof}
\end{lemma}

To conclude this section, we give a bound on the Hilbert series of the global Selmer scheme $\Sel_{\lSel,U}$, analogous to the bound in Corollary~\ref{cor:cohomology_hilbert_series_bound}.

\begin{lemma}\label{lem:hilbert_series_bound_global}
	Let~$\lSel$ be a Selmer structure for~$U$. Then the global Selmer scheme $\Sel_{\lSel,U}$ is (non-canonically) a closed subscheme of
	\[
	\prod_{n>0}\HH^1_f(G_K,V_n)\times\prod_{v\nmid\ellp\infty}\lSel_v \,,
	\]
	with the induced filtration. Here, $\lSel_v$ is given the induced filtration from~$\HH^1(G_v,U)$, and $\HH^1_f(G_K,V_n)$ denotes the $\Q_\ellp$-vector space of Galois cohomology classes which are crystalline at all places above~$\ellp$ and unramified at all other places, given the filtration supported in filtration-degree~$-n$. In particular, we have
	\[
	\HS_{\Sel_{\lSel,U}} \preceq \prod_{v\nmid\ellp\infty}\HS_{\lSel_v}\times\prod_{n>0}(1-t^n)^{-\dim_{\Q_\ellp}\HH^1_f(G_K,V_n)} \,.
	\]
\end{lemma}

For our proof of Lemma~\ref{lem:hilbert_series_bound_global}, we introduce some auxiliary notation. We fix a set~$T$ as in the proof of Proposition~\ref{prop:global_selmer_representability}, and for~$n\geq0$, we define~$\Sel_{\lSel,n}$ to be the pullback
\begin{center}
	\begin{tikzcd}
		\Sel_{\lSel,n} \arrow[r]\arrow[d]\arrow[dr, phantom, "\scalebox{1.5}{$\lrcorner$}" , very near start, color=black] & \HH^1(G_{K,T},U_n) \arrow[d] \\
		\prod_{v\mid\ellp}\HH^1_f(G_v,U_n)\times\prod_{v\nmid\ellp\infty}\lSel_v \arrow[r] & \prod_{v\nmid\infty}\HH^1(G_v,U_n)
	\end{tikzcd}
\end{center}
in the category of filtered affine $\Q_\ellp$-schemes. That is, a~$\Lambda$-point of~$\Sel_{\lSel,n}$ consists of a global cohomology class $[\xi]\in\HH^1(G_{K,T},U_n(\Lambda))$ and local cohomology classes $[\xi_v]\in\lSel_v(\Lambda)\subseteq\HH^1(G_v,U(\Lambda))$ for each $v\nmid\ellp\infty$, such that $[\xi]|_{G_v}\in\HH^1_f(G_v,U_n(\Lambda))$ for all $v\mid\ellp$ and $[\xi]|_{G_v}$ is congruent to $[\xi_v]$ modulo $\W_{-n-1}U$ for all $v\nmid\ellp\infty$.

There are natural maps $\Sel_{\lSel,n}\to\Sel_{\lSel,n-1}$ for every~$n>0$, and it follows from Lemma~\ref{lem:unramified_trivial} that there is an action of~$\HH^1_f(G_K,V_n)$ on~$\Sel_{\lSel,n}$, induced from the natural pointwise multiplication actions on $\HH^1(G_{K,T},U_n)$, $\HH^1(G_v,U_n)$ and $\HH^1_f(G_v,U_n)$ and the trivial action on~$\lSel_v$. The bulk of the proof of Lemma~\ref{lem:hilbert_series_bound_global} is then contained in the following lemma.

\begin{lemma}\label{lem:hilbert_series_bound_global_step}
	The pointwise image of the map $\Sel_{\lSel,n}\to\Sel_{\lSel,n-1}$ is a closed subscheme of~$\Sel_{\lSel,n-1}$ with the induced filtration, and $\HH^1_f(G_K,V_n)$ acts freely and simply transitively on the fibres of $\Sel_{\lSel,n}\to\Sel_{\lSel,n-1}$.
	\begin{proof}
		First, recall from the proof of Theorem~\ref{thm:representability} that there is a coboundary map
		\[
		\delta\colon\HH^1(G_{K,T},U_{n-1})\to\HH^2(G_{K,T},V_n)
		\]
		whose kernel is the pointwise image of $\HH^1(G_{K,T},U_n)\to \HH^1(G_{K,T},U_{n-1})$. It follows that the image of $\Sel_{\lSel,n}\to\Sel_{\lSel,n-1}$ is contained in the kernel $\Sel_{\lSel,n-1}'$ of the composite map
		\[
		\Sel_{\lSel,n-1} \to \HH^1(G_{K,T},U_{n-1}) \xrightarrow\delta \HH^2(G_{K,T},V_n) \,.
		\]
		By construction, $\Sel_{\lSel,n-1}'$ is a closed subscheme of~$\Sel_{\lSel,n-1}$ with the induced filtration.
		
		Next, we let~$V$ denote the cokernel of the restriction map
		\begin{equation}\label{eq:global_selmer_cutout}\tag{$\ast$}
		\HH^1(G_{K,T},V_n)\to\prod_{v\mid\ellp}\left(\HH^1(G_v,V_n)/\HH^1_f(G_v,V_n)\right)\times\prod_{v\nmid\ellp\infty}\HH^1(G_v,V_n)
		\end{equation}
		of $\Q_\ellp$-vector spaces. We endow~$V$ with the filtration supported in filtration-degree $-n$, and regard~$V$ as a filtered affine $\Q_\ellp$-scheme as in Example~\ref{ex:filtered_vs}. We will define a map
		\[
		\delta'\colon\Sel_{\lSel,n-1}'\to V
		\]
		of filtered affine $\Q_\ellp$-schemes whose kernel is the pointwise image of $\Sel_{\lSel,n}\to\Sel_{\lSel,n-1}$ -- this proves the proposition. In this, we regard the vector space~$V$ as a filtered affine $\Q_\ellp$-scheme by giving the filtration supported in filtration-degree~$-n$, as in Example~\ref{ex:filtered_vs}.
		
		The map~$\delta'$ is constructed as follows. Suppose that $([\xi],[\xi_v]_{v\nmid\ellp\infty})$ is a $\Lambda$-valued point of~$\Sel_{\lSel,n-1}'$. The condition that this point lies in~$\Sel_{\lSel,n-1}'$ says exactly that~$[\xi]$ is the image of a class $[\xi']\in\HH^1(G_{K,T},U_n(\Lambda))$. Now for every $v\nmid\ellp\infty$, the classes $[\xi']|_{G_v}$ and $[\xi_v]$ map to the same class in~$\HH^1(G_v,U_{n-1}(\Lambda))$, so differ by the action of a unique element $[\eta_v]\in\HH^1(G_v,V_n(\Lambda))$.
		
		Similarly, for every~$v\mid\ellp$, we know that the image of~$[\xi']|_{G_v}$ in $\HH^1(G_v,U_{n-1}(\Lambda))$ lies in $\HH^1_f(G_v,U_{n-1}(\Lambda))$. It follows from Proposition~\ref{prop:bk_log_is_filtered_iso} that $\HH^1_f(G_v,U_n(\Lambda))$ surjects onto~$\HH^1_f(G_v,U_{n-1}(\Lambda))$, and that~$\HH^1_f(G_v,V_n(\Lambda))$ acts freely and transitively on the fibres of $\HH^1_f(G_v,U_n(\Lambda))\to\HH^1_f(G_v,U_{n-1}(\Lambda))$. It follows that~$[\xi']|_{G_v}$ differs from an element of~$\HH^1_f(G_v,U_n(\Lambda))$ by the action of some element~$[\eta_v]\in\HH^1(G_v,V_n(\Lambda))$, and that the element~$[\eta_v]$ is unique modulo~$\HH^1_f(G_v,V_n(\Lambda))$.
		
		We then define~$\delta'([\xi],[\xi_v]_{v\nmid\ellp\infty})$ to be the image of $([\eta_v]_{v\nmid\infty})$ inside~$V(\Lambda)$. The choice of~$[\xi']\in\HH^1(G_{K,T},U_n(\Lambda))$ in the above construction is unique up to multiplication by an element of~$\HH^1(G_{K,T},V_n(\Lambda))$, so by construction $\delta'([\xi],[\xi_v]_{v\nmid\ellp\infty})$ is independent of this choice. We have thus described a map
		\[
		\delta'\colon\Sel_{\lSel,n-1}'(\Lambda)\to V(\Lambda)
		\]
		natural in~$\Lambda$, and hence the desired morphism $\delta'\colon\Sel_{\lSel,n-1}'\to V$ of filtered affine $\Q_\ellp$-schemes.
		
		Now by construction, a point of~$([\xi],[\xi_v]_{v\in T_0})\in\Sel_{\lSel,n-1}'(\Lambda)$ lies in the kernel of~$\delta'$ if~$[\xi]$ has a lift to~$[\xi']\in\HH^1(G_{K,T},U_n(\Lambda))$ for which all the elements~$[\eta_v]$ vanish. But this is equivalent to saying that $([\xi'],[\xi_v]_{v\in T_0})\in\Sel_{\lSel,n}(\Lambda)$, so the kernel of~$\delta'$ is the pointwise image of~$\Sel_{\lSel,n}\to\Sel_{\lSel,n-1}$ as claimed.
		
		\smallskip
		
		Finally, we observe that the choice of~$[\xi']$ for which all the~$[\eta_v]$ vanish is well-defined up to the action of the kernel of~\eqref{eq:global_selmer_cutout}, which is~$\HH^1_f(G_K,V_n(\Lambda))$ by Lemma~\ref{lem:unramified_trivial}. This shows that the action of~$\HH^1_f(G_v,V_n)$ on $\Sel_{\lSel,n}$ is pointwise transitive on the fibres of~$\Sel_{\lSel,n}\to\Sel_{\lSel,n-1}$. That it is free follows from the fact that~$\HH^1(G_{K,T},V_n)$ acts freely on~$\HH^1(G_{K,T},U_n)$.
	\end{proof}
\end{lemma}

\begin{remark}\label{rmk:global_hilbert_series_bound_for_naive}
	For the particular Selmer scheme $\HH^1_{f,T_0}(G_K,U)$ from Example~\ref{ex:naive_selmer}, a similar argument to the proof of Lemma~\ref{lem:hilbert_series_bound_global} gives the bound
	\[
	\HS_{\HH^1_{f,T_0}(G_K,U)}(t) \preceq \prod_{n>0}(1-t^n)^{-\dim_{\Q_\ellp}\HH^1_{f,T_0}(G_K,V_n)} \,,
	\]
	where~$\HH^1_{f,T_0}(G_K,V_n)$ denotes the subgroup of the cohomology group consisting of cohomology classes unramified outside $T_0\cup\{v\mid\ellp\}$ and crystalline at $\ellp$-adic places.
\end{remark}

%% file: coleman.tex
\section{Weights and Coleman analytic functions}\label{s:coleman}

Let~$K_v$ be a finite extension of $\Q_\ellp$ with ring of integers $\O_v$ and residue field~$k_v$. Let~$Y$ be a smooth hyperbolic curve over~$K_v$, and write $Y=X\setminus D$ where~$D$ is a divisor in a smooth proper curve~$X$. Assume that~$Y$ has good reduction at~$\ellp$, that is, there is a smooth proper curve $\X/\O_v$ and an \'etale divisor $\Dvsr\subseteq\X$ whose generic fibres are~$X$ and~$D$ respectively. We write~$\Y=\X\setminus\Dvsr$, and fix a choice of base point~$b\in\Y(\O_v)$.

We write $\Y_0$ for the special fibre of~$\Y$, and write $]\Y_0[$ for the tube of~$\Y_0$ inside~$X$ \cite[\S1]{berthelot:cohomologie_rigide}. This is an admissible open in the rigid analytification~$X^\an$ of~$X$, whose~$K_v$-points is equal to the $\O_v$-points of~$\Y$ (i.e.\ the complement of the residue discs of each $K_v$-point of~$D$).
\smallskip

In this setup, one has the ring $\ACol(\Y)$ of \emph{Coleman analytic functions} on~$\Y$, as constructed by Besser \cite{besser:tannakian}. This is a subring of the locally $K_v$-analytic functions on $]\Y_0[$ which contains all overconvergent $K_v$-rigid-analytic functions.

The ring $\ACol(\Y)$ plays an important role in the non-abelian Chabauty method. Specifically, if~$U^\dR/K_v$ denotes the pro-unipotent de Rham fundamental group of~$(Y,b)$, then there is a certain map
\[
\jj_\dR\colon\Y(\O_v)\to\Fil^0\backslash U^\dR(K_v)
\]
known as the \emph{de Rham non-abelian Kummer map}, which is defined in terms of the crystalline Frobenius and Hodge filtration on the de Rham fundamental groupoid of~$Y$. The coordinates of the de Rham non-abelian Kummer map are known to be Coleman analytic functions, and this is ultimately how one proves Faltings--Siegel-type results in the Chabauty--Kim method: by exploiting the fact that a non-zero Coleman analytic function has only finitely many zeroes.
\smallskip

Our goal in this section is to explain how to enrich this picture to take account of weight filtrations. We will introduce a certain subring $\ACollog(Y)\subseteq \ACol(\Y)$ of \emph{Coleman algebraic functions}, and endow it with a natural weight filtration. We will show that the coordinates of the de Rham non-abelian Kummer map are in fact Coleman \emph{algebraic}, and that the weight filtration on $\ACollog(Y)$ is compatible with the natural weight filtration on $\Fil^0\backslash U^\dR$. In the next section, we will show that the number of zeroes of a non-zero Coleman algebraic function can be bounded in terms of its weight: this is ultimately how we produce effective results of Faltings--Siegel type in our effective Chabauty--Kim method.

\begin{remark}
	The whole ring $\ACol(\Y)$ of Coleman analytic functions on~$\Y$ also carries a natural weight filtration, but this is much less useful in the theory. For instance, $\W_0\ACol(\Y)$ is the ring of overconvergent $K_v$-rigid analytic functions on~$]\Y_0[$, so one cannot put an upper bound on the number of zeroes of a general Coleman analytic function in terms of its weight. By contrast, we have $\W_0\ACollog(Y)=K_v$ is just the constant functions.
\end{remark}

\subsection{Coleman analytic functions}

\subsubsection{Isocrystals and Coleman analytic functions}

To begin with, we recall Besser's definition of the ring $\ACol(\Y)$ of Coleman analytic functions on~$\Y$ \cite[\S4]{besser:tannakian}\footnote{Officially, we mean that~$\ACol(\Y)$ is the ring of Coleman analytic functions associated to the rigid triple $T=(\Y_0,\X_0,\hat\X)$, where~$\hat\X$ denotes the formal completion of~$\X$ along its special fibre.}. Let~$\Isoc^\un_{K_v}(\Y_0)$ denote the Tannakian category of \emph{unipotent isocrystals} on~$\Y_0$, in the sense of \cite[p.~6]{besser:tannakian}. These are coherent modules over the sheaf $\jota^\dagger\O_{X^\an}$ of overconvergent rigid-analytic functions on~$X^\an$ endowed with a connection~$\nabla$, which are iterated extensions of the unit object $(\jota^\dagger\O_{X^\an},\d)$.

If~$\E=(\E,\nabla)$ is a unipotent isocrystal on~$\Y_0$ and $y_0\in\Y_0(k_v)$ is a point on the special fibre, we write $(\E|_{]y_0[})^{\nabla=0}$ for the $K_v$-vector space of flat sections of~$\E$ over the residue disc~$]y_0[$. If~$x_0\in\Y_0(k_v)$ is another point, then there is a canonical $K_v$-linear isomorphism
\[
T_{x_0,y_0}^\nabla\colon(\E|_{]x_0[})^{\nabla=0}\xrightarrow\sim(\E|_{]y_0[})^{\nabla=0} \,,
\]
known as \emph{analytic continuation along Frobenius} \cite[Definition~3.5]{besser:tannakian}, uniquely characterised by the fact that~$T_{x_0,y_0}^\nabla$ is tensor-natural in~$\E$ and compatible with Frobenius \cite[Corollary~3.3]{besser:tannakian}.
\smallskip

Using the category~$\Isoc^\un_{K_v}(\Y_0)$ of unipotent isocrystals on~$\Y_0$, Besser defines a ring~$\ACol(\Y)$ as follows.

\begin{definition}\label{def:coleman_fns}
	An \emph{abstract Coleman analytic function} \cite[Definition~4.1]{besser:tannakian} is a triple $(\E,\sigma,\tau)$ consisting of:
	\begin{itemize}
		\item a unipotent isocrystal~$\E\in\Isoc^\un_{K_v}(\Y_0)$;
		\item an $\jota^\dagger\O_{X^\an}$-linear map $\tau\colon\E\to\jota^\dagger\O_{X^\an}$; and
		\item for every point $y_0\in\Y_0(k_v)$, a flat section $\sigma_{y_0}\in(\E|_{]y_0[})^{\nabla=0}$ of~$\E$ over the residue disc of~$y_0$;
	\end{itemize}
	such that $\sigma_{y_0}=T^\nabla_{x_0,y_0}(\sigma_{x_0})$ for all~$x_0,y_0\in\Y_0(k_v)$.
	
	The ring $\ACol(\Y)$ of \emph{Coleman analytic functions} \cite[Definition~4.3]{besser:tannakian} is defined to be the set of abstract Coleman functions, modulo the equivalence relation generated by identifying $(\E,\sigma,\tau)\sim(\E',\sigma',\tau')$ whenever there is a morphism $f\colon\E\to\E'$ of isocrystals such that $f(\sigma_{y_0})=\sigma'_{y_0}$ for all~$y_0\in\Y_0(k_v)$ and $\tau'\circ f=\tau$. The addition in~$\ACol(\Y)$ is induced from the direct sum of isocrystals; the multiplication is induced from the tensor product of isocrystals. This makes~$\ACol(\Y)$ into a $K_v$-algebra \cite[Proposition~4.4]{besser:tannakian}.
\end{definition}

The ring~$\ACol(\Y)$ is a subring of the ring~$A_\loc(\Y)$ of locally analytic functions on the tube~$]\Y_0[$ of~$\Y_0$ \cite[Definition~4.9 \& Proposition~4.12]{besser:tannakian}. Explicitly, if~$f\in \ACol(\Y)$ is represented by a triple~$(\E,\sigma,\tau)$, then for every~$y_0\in\Y_0(k_v)$ we obtain an analytic function~$f|_{]y_0[}$ on the residue disc~$]y_0[$ via the formula $f|_{]y_0[}:=\tau(\sigma_{y_0})\in\HH^0(]y_0[,\O_{X^\an})$. In particular, we can view~$f$ as a function $\Y(\O_v)\to K_v$, given on each residue disc by $f|_{]y_0[}$.

\begin{example}\label{ex:overconvergent_function}
	Every overconvergent rigid-analytic function~$f\in\HH^0(X^\an,\jota^\dagger\O_{X^\an})$ is a Coleman analytic function. Specifically, to~$f$ we can associate the abstract Coleman function $(\E,\sigma,\tau)$ where $\E=(\jota^\dagger\O_{X^\an},\d)$, $i=1\in K_v=(\E|_{]b_0[})^{\nabla=0}$, and $\tau\colon\E\to\jota^\dagger\O_{X^\an}$ is given by multiplication by~$f$. It is easy to check that the underlying locally analytic function of this abstract Coleman analytic function is again~$f$. In this way, we see that the ring $\HH^0(X^\an,\jota^\dagger\O_{X^\an})$ of overconvergent rigid-analytic functions is a subring of $\ACol(\Y)$.
\end{example}

\subsubsection{Log-connections and Coleman algebraic functions}\label{sss:log-coleman}

For our purposes, we will be interested not in the whole ring $\ACol(\Y)$, but in a certain subring $\ACollog(Y)$ which we call the ring of \emph{Coleman algebraic functions}. For this, let~$\MIC^\un(X,D)$ denote the Tannakian category of unipotent vector bundles with log-connection on~$(X,D)$, i.e.\ the category of vector bundles~$\E$ on~$X$ endowed with a connection~$\nabla$ with logarithmic poles along~$D$ such that $(\E,\nabla)$ is an iterated extension of the unit object $(\O_X,\d)$.

\begin{definition}\label{def:log-coleman_fns}
	An \emph{abstract Coleman algebraic function} is a triple $(\E,\sigma_b,\tau)$ consisting of:
	\begin{itemize}
		\item a unipotent vector bundle with log-connection~$\E\in\MIC^\un(X,D)$;
		\item an $\O_X$-linear map $\tau\colon\E\to\O_X$; and
		\item a point $\sigma_b\in\E_b$ in the fibre of~$\E$ at the basepoint~$b$.
	\end{itemize}
	
	The ring $\ACollog(Y)$ of \emph{Coleman algebraic functions} is defined to be the set of abstract Coleman algebraic functions, modulo the equivalence relation generated by identifying $(\E,\sigma_b,\tau)\sim(\E',\sigma'_b,\tau')$ whenever there is a morphism $f\colon\E\to\E'$ of vector bundles with log-connection such that $f(\sigma_b)=\sigma'_b$ and $\tau'\circ f=\tau$. The addition in~$\ACollog(Y)$ is induced from the direct sum of vector bundles with log-connection; the multiplication is induced from the tensor product. This makes~$\ACollog(Y)$ into a $K_v$-algebra, just as in \cite[Proposition~4.4]{besser:tannakian}.
\end{definition}

It is not immediately obvious that~$\ACollog(Y)$ is a subring of~$\ACol(\Y)$. To see this, recall that if~$\E$ is a unipotent vector bundle on~$X$, then one can form the sheaf~$\jota^\dagger\E^\an$ of overconvergent sections of~$\E$ \cite[\S2.1.1]{berthelot:cohomologie_rigide}. Moreover, since~$D$ is supported outside a strict neighbourhood of~$]\Y_0[$, we have $\jota^\dagger(\Omega^1_X(D))^\an=\jota^\dagger(\Omega^1_X)^\an\cong\Omega^1_{X^\an}\otimes_{\O_{X^\an}}\jota^\dagger\O_{X^\an}$, and hence any unipotent log-connection on~$\E$ gives rise to a connection on~$\jota^\dagger\E^\an$, making it into a unipotent isocrystal.

\begin{lemma}\label{lem:equivalence_of_isocrystals_and_connections}
	The functor $\E\mapsto\jota^\dagger\E^\an$ is a tensor-equivalence from the Tannakian category $\MIC^\un(X,D)$ of unipotent vector bundles with log-connection on~$(X,D)$ to the Tannakian category $\Isoc^\un(\Y_0)$ of unipotent isocrystals on~$\Y_0$. Moreover, for any~$y\in\Y(\O_v)$ reducing to some $y_0\in\Y_0(k_v)$, taking the fibre at~$y$ provides a $K_v$-linear isomorphism
	\[
	(\jota^\dagger\E^\an|_{]y_0[})^{\nabla=0}\cong\E_y \,,
	\]
	tensor-natural in $\E\in\MIC^\un(X,D)$.
	\begin{proof}
		We introduce as an intermediate step the category $\MIC^\un(Y)$ of unipotent vector bundles with connection on~$Y$. On the one hand, the restriction functor $\MIC^\un(X,D)\to\MIC^\un(Y)$ given by $\E\mapsto\E|_Y$ is a tensor-equivalence of Tannakian categories by Deligne's theory of the canonical extension \cite[Proposition~II.5.2(a)]{deligne:equations_differentielles}. On the other, the functor $\MIC^\un(Y)\to\Isoc^\un(\Y_0)$ given by $\E\mapsto\jota^\dagger\E^\an$ is a tensor-equivalence by \cite[Proposition~2.4.1]{chiarellotto-le_stum:f-isocristaux}. Combining these shows the first part of the lemma.
		
		For the second part, taking the fibre at~$y$ provides a map $(\jota^\dagger\E^\an|_{]y_0[})^{\nabla=0}\to\E_y$ which is tensor-natural in $\E$. Standard properties of Tannakian categories imply that this is automatically an isomorphism \cite[Proposition~1.13]{milne:tannakian}.
	\end{proof}
\end{lemma}

Now if~$\E$ is a unipotent vector bundle with log-connection on~$(X,D)$ and~$\sigma_b\in\E_b$ is a point in the fibre over~$b$, then according to Lemma~\ref{lem:equivalence_of_isocrystals_and_connections} there is a unique flat section~$\sigma_{b_0}$ of~$\jota^\dagger\E^\an$ over the residue disc~$]b_0[$ whose fibre at~$b$ is~$\sigma_{b_0}$. Via analytic continuation along Frobenius, we then obtain a flat section~$\sigma_{y_0}$ of~$\jota^\dagger\E^\an$ over every other residue disc~$]y_0[$ by setting $\sigma_{y_0}:=T_{b_0,y_0}^\nabla(\sigma_{b_0})$. This construction gives our embedding of $\ACollog(Y)$ in~$\ACol(\Y)$.

\begin{proposition}\label{prop:coleman_subring}
	The ring $\ACollog(Y)$ is a $K_v$-subalgebra of $\ACol(\Y)$ via the embedding given by $(\E,\sigma_b,\tau)\mapsto(\jota^\dagger\E^\an,\sigma,\jota^\dagger \tau^\an)$, where the flat sections $\sigma_{y_0}$ are defined by $\sigma_{y_0}:= T^\nabla_{b_0,y_0}(\sigma_{b_0})$ with $\sigma_{b_0}$ the unique flat section of~$\jota^\dagger\E^\an$ over the residue disc $]b_0[$ whose fibre at~$b$ is~$\sigma_b$.
\end{proposition}

For the proof of Proposition~\ref{prop:coleman_subring}, we will need to know when a triple~$(\E,\sigma_b,\tau)$ represents the element $0\in \ACollog(Y)$.

\begin{lemma}\label{lem:represents_zero}
	\leavevmode
	\begin{enumerate}
		\item\label{lempart:log-analytic_represents_zero} Let~$(\E,\sigma_b,\tau)$ be an abstract Coleman algebraic function, and let~$\nabla^{-\infty}\ker(\tau)$ be the largest $\nabla$-stable subbundle of the kernel of $\tau\colon\E\to\O_X$. Then $(\E,\sigma,\tau)$ represents the element $0\in \ACollog(Y)$ if and only if $\sigma_b\in\nabla^{-\infty}\ker(\tau)_b$.
		\item Let~$(\E,\sigma,\tau)$ be an abstract Coleman function, and let~$\nabla^{-\infty}\ker(\tau)$ be the largest $\nabla$-stable $\jota^\dagger\O_{X^\an}$-submodule of the kernel of $\tau\colon\E\to\jota^\dagger\O_{X^\an}$, as in \cite[Lemma~2.1]{besser:tannakian}. Then $(\E,\sigma,\tau)$ represents the element $0\in \ACol(\Y)$ if and only if $\sigma_{b_0}\in\left(\nabla^{-\infty}\ker(\tau)|_{]b_0[}\right)^{\nabla=0}$.
	\end{enumerate}
	\begin{proof}
		We prove only the first part, the second following by a similar argument. In one direction, if~$\sigma_b\in\nabla^{-\infty}\ker(\tau)_b$, then let~$\E'=\E/\nabla^{-\infty}\ker(\tau)$ and let~$\tau'\colon\E'\to\O_X$ be the map through which~$\tau$ factors. We then have $(\E,\sigma,\tau)\sim(\E',0,\tau')\sim(0,0,0)$, and hence~$(\E,\sigma,\tau)$ represents~$0\in \ACollog(Y)$.
		
		In the other, it suffices to prove that the property that $\sigma_b\in\nabla^{-\infty}\ker(\tau)_b$ is invariant under the equivalence relation in Definition~\ref{def:log-coleman_fns}. So suppose that~$(\E,\sigma,\tau)$ and~$(\E',\sigma',\tau')$ are abstract Coleman algebraic functions, and that $f\colon\E\to\E'$ is a morphism of vector bundles with log-connection such that $f(\sigma_b)=\sigma'_b$ and $\tau'\circ f=\tau$. This latter property implies that $\nabla^{-\infty}\ker(\tau)=f^{-1}\nabla^{-\infty}\ker(\tau')$, and hence that $\sigma_b\in\nabla^{-\infty}\ker(\tau)_b$ if and only if $\sigma'_b\in\nabla^{-\infty}\ker(\tau')_b$. This completes the proof.
	\end{proof}
\end{lemma}

\begin{proof}[Proof of Proposition~\ref{prop:coleman_subring}]
	It is easily checked that the construction $(\E,\sigma_b,\tau)\mapsto(\jota^\dagger\E^\an,\sigma,\jota^\dagger\tau^\an)$ gives rise to a $K_v$-algebra homomorphism $\ACollog(Y)\to \ACol(\Y)$. Suppose that $f\in \ACollog(Y)$ is represented by a triple~$(\E,\sigma_b,\tau)$. Let $\nabla^{-\infty}\ker(\tau)$ denote the largest $\nabla$-stable $\O_X$-submodule of the kernel of $\tau\colon\E\to\O_X$, and let $\nabla^{-\infty}\ker(\jota^\dagger \tau^\an)$ denote the largest $\nabla$-stable $\jota^\dagger\O_{X^\an}$-submodule of the kernel of $\jota^\dagger \tau^\an\colon\jota^\dagger\E^\an\to\jota^\dagger\O_{X^\an}$. It follows from Lemma~\ref{lem:equivalence_of_isocrystals_and_connections} that in fact $\nabla^{-\infty}\ker(\jota^\dagger \tau^\an)=\jota^\dagger\nabla^{-\infty}\ker(\tau)^\an$.
	
	Now suppose that~$f$ lies in the kernel of the map $\ACollog(Y)\to \ACol(\Y)$. According to Lemma~\ref{lem:represents_zero}, we have $\sigma_{b_0}\in\left(\nabla^{-\infty}\ker(\jota^\dagger \tau^\an)|_{]b_0[}\right)^{\nabla=0}$. It follows that~$\sigma_b\in\nabla^{-\infty}\ker(\tau)_b$, and hence $(\E,\sigma_b,\tau)$ represents $0\in \ACollog(Y)$. In other words, $\ACollog(Y)\to \ACol(\Y)$ is injective.
\end{proof}

\begin{remark}\label{rmk:log-analytic_indept_of_b}
	Although the definition of~$\ACollog(Y)$ appears to depend on the basepoint~$b$, Proposition~\ref{prop:coleman_subring} shows $\ACollog(Y)$ is the subring of $\ACol(\Y)$ consisting of those Coleman analytic functions~$f$ which can be represented by a triple~$(\E,\sigma,\tau)$ in which $\tau=\jota^\dagger\tau^\alg$ for a morphism~$\tau^\alg\colon\E^\alg\to\O_X$ of algebraic vector bundles, where~$\E^\alg\in\MIC^\un(X,D)$ is the unipotent vector bundle with log-connection corresponding to~$\E$ under the equivalence of Lemma~\ref{lem:equivalence_of_isocrystals_and_connections}. Thus, the ring $\ACollog(Y)$ is actually independent of~$b$, when viewed as a filtered subring of $\ACol(\Y)$.
\end{remark}

\begin{remark}\label{rmk:coleman_fns_for_projective_curves}
	In the particular case that~$Y=X$ is projective, then the inclusion $\ACollog(Y)\subseteq \ACol(\Y)$ is an equality, since by rigid-analytic GAGA \cite[Theorem~4.10.5]{fresnel_van-der-put:rigid_analytic_geometry} any morphism $\tau\colon\E\to\O_{X^\an}$ of analytic vector bundles on~$X^\an$ is the analytification of a morphism $\tau^\alg\colon\E^\alg\to\O_X$ of algebraic vector bundles.
	
	On the other hand, if $Y\subsetneq X$ is affine, then the inclusion $\ACollog(Y)\subseteq \ACol(\Y)$ is strict. Indeed, $\ACol(\Y)$ contains all overconvergent rigid-analytic functions on~$]\Y_0[$ by Example~\ref{ex:overconvergent_function}, so in particular has uncountable dimension over~$K_v$. However, one can check that $\ACollog(Y)$ always has countable dimension over~$K_v$, so the inclusion $\ACollog(Y)\subseteq \ACol(\Y)$ must be strict.
\end{remark}

\subsubsection{Weight filtration}\label{sss:coleman_weights}

The ring $\ACollog(Y)$ of Coleman algebraic functions comes with a natural filtration called the \emph{weight filtration}, defined as follows.

\begin{definition}[Weight filtration on $\ACollog(Y)$]\label{def:coleman_weights}
	\leavevmode
	\begin{enumerate}
		\item A unipotent vector bundle~$\E$ with log-connection on~$(X,D)$ is said to have \emph{weight at most~$m$} just when it admits a $\nabla$-stable filtration
		\[
		0=\W_{-1}\E\leq\W_0\E\leq\W_1\E\leq\dots\leq\W_m\E=\E
		\]
		such that:
		\begin{itemize}
			\item $\gr^\W_i\E:=\W_i\E/\W_{i-1}\E$ is a trivial vector bundle with connection (direct sum of copies of $(\O_X,\d)$) for all $0\leq i\leq m$; and
			\item the connection on $\W_i\E/\W_{i-2}\E$ is regular on~$X$ (i.e.\ takes values in $\Omega^1_X\otimes_{\O_X}(\W_i\E/\W_{i-2}\E)$) for all $0<i\leq m$.
		\end{itemize}
		\item A Coleman algebraic function is said to have \emph{weight at most~$m$} just when it is represented by an abstract Coleman algebraic function $(\E,\sigma_b,\tau)$ where~$\E$ has weight at most~$m$. We write~$\W_m\ACollog(Y)$ for the set of elements of $\ACollog(Y)$ of weight at most~$m$. It is easy to check that this defines an exhaustive $K_v$-algebra filtration on $\ACollog(Y)$.
	\end{enumerate}
\end{definition}

\begin{remark}
	In the particular case that~$Y=X$ is projective, the weight filtration is the same as the filtration considered in \cite[Definition~5.4]{besser:tannakian}.
\end{remark}

We conclude this section with some illustrative examples of Coleman algebraic functions of small weight. None of this will be directly used in what follows.

\begin{example}\label{ex:coleman_weight_0}
	A unipotent vector bundle~$\E$ with log-connection on~$(X,D)$ has weight at most~$0$ if and only if it is a direct sum of copies of $(\O_X,\d)$. In particular, if $f\in \W_0\ACollog(Y)$, then~$f$ can be represented by a triple $(\E,\sigma_b,\tau)$ where~$\tau\colon\E\to\O_X$ is compatible with the connection. It follows that~$f$ can be represented by a triple where~$\E=(\O_X,\d)$ and~$\tau$ is the identity. In other words,~$f\in K_v$ is constant.
\end{example}

\begin{example}\label{ex:coleman_weight_1}
	Suppose that~$f\in\W_1\ACollog(Y)$, so that~$f$ is represented by a triple~$(\E,\sigma_b,\tau)$ where~$\E$ is an extension
	\[
	0 \to \W_0\E \to \E \to \gr^\W_1\E \to 0
	\]
	of vector bundles with connection on~$X$ with $\W_0\E$ and $\gr^\W_1\E$ both trivial. The restriction of~$\tau$ to~$\W_0\E$ automatically preserves the connection, so quotienting~$\E$ by the kernel of~$\tau|_{\W_0\E}$ if necessary, we may assume without loss of generality that~$\W_0\E$ has rank at most~$1$ and that~$\tau|_{\W_0\E}$ is injective. Similarly, the smallest subbundle of~$\gr^\W_1\E$ containing the image of~$\sigma_b$ is stable under the connection, so we may also assume without loss of generality that~$\gr^\W_1\E$ has rank at most~$1$, and that its fibre at~$b$ is spanned by the image of~$\sigma_b\in\E_b$.
	
	Now if either~$\W_0\E=0$ or~$\gr^\W_1\E=0$ then~$\E$ has weight at most~$0$ and~$f$ is constant. Otherwise,~$\E$ has rank~$2$, and the map~$\tau\colon\E\to\O_X$ yields a splitting~$\E=\O_X^{\oplus 2}$ as vector bundles (not necessarily compatible with the connection) for which~$\tau$ is the projection onto the first factor and~$\sigma_b\in\E_b=K_v^{\oplus 2}$ is of the form~$(a,1)$ for some~$a$. Using this splitting $\E=\O_X^{\oplus 2}$, we may write the connection on~$\E$ as
	\[
	\nabla = \d -
	\left(\begin{matrix}
		0 & \omega \\
		0 & 0
	\end{matrix}\right)
	\]
	for some~$\omega\in\HH^0(X,\Omega^1_X)$. Put together, this says that~$f$ is given by the Coleman integral
	\[
	z\mapsto a+\int_b^z\omega \,.
	\]
	Thus, $\W_1\ACollog(Y)$ consists exactly of the constant functions plus Coleman integrals of differential forms~$\omega$ of the first kind. 
\end{example}

Already in weight~$2$, the theory of Coleman algebraic functions is much richer, and involves the kinds of functions seen in quadratic Chabauty.

\begin{lemma}\label{lem:coleman_weight_2}
	Let~$E_0$ be an effective divisor on~$X$, supported outside the residue disc of~$b$, such that $\HH^1(X,\O_X(E_0))=0$, and write $E=E_0+\supp(E_0)$. Let $\omega_1,\dots,\omega_{2g}\in\HH^0(X,\Omega^1_X(E))$ be differentials of the second kind forming a basis of~$\HH^1_\dR(X/K_v)$, such that that~$\omega_1,\dots,\omega_g$ form a basis of~$\HH^0(X,\Omega^1_X)$.
	
	Suppose $f\in\W_2\ACollog(Y)$ is a Coleman algebraic function of weight at most~$2$. Then there are constants $a_{ij}\in K_v$ for~$1\leq i\leq 2g$ and~$1\leq j\leq g$, constants $a_i\in K_v$ for~$g+1\leq i\leq2g$, a differential $\eta\in\HH^0(X,\Omega^1_X(E+D))$ of the third kind, and a rational function $h\in\HH^0(X,\O_X(E+D))$ such that~$f$ is given by the iterated Coleman integral
	\[
	f(z) = \sum_{\substack{1\leq i\leq 2g \\ 1\leq j\leq g}}a_{ij}\int_b^z\omega_i\omega_j + \sum_{g+1\leq i\leq 2g}a_i\int_b^z\omega_i + \int_b^z\eta + h(z)
	\]
	outside the residue discs meeting the divisor $E+D$. Here, we adopt the same convention for iterated integrals as \cite[p.~109]{minhyong:selmer}: that the right-hand differential is integrated ``first''.
	\begin{proof}
		We begin by explaining why the basis~$\omega_1,\dots,\omega_{2g}$ exists. A de Rham cohomology class~$[\E]\in\HH^1_\dR(X/K_v)$ is represented by an extension
		\[
		0 \to \O_X \to \E \to \O_X \to 0
		\]
		of vector bundles with connection on~$X$. We write~$\E'\subseteq\E$ for the preimage of the subbundle~$\O_X(-E_0)\subseteq\O_X$ under the map $\E\to\O_X$, which is an extension of~$\O_X(-E_0)$ by~$\O_X$. Since $\Ext^1_{\O_X}(\O_X(-E_0),\O_X)=\HH^1(X,\O_X(E_0))=0$, the extension~$\E'$ splits. A choice of splitting $\E'=\O_X\oplus\O_X(-E_0)$ gives rise to a splitting of the bundle~$\E$ over the generic point of~$X$, with respect to which the connection can be written as
		\[
		\nabla = \d - \left(\begin{matrix}
			0 & \omega \\
			0 & 0
		\end{matrix}\right)
		\]
		for some rational differential~$\omega$.
		
		If now~$\E''\subseteq\E$ denotes the preimage of the subbundle $\O_X(-E)\subseteq\O_X$, then $\E''\subseteq\E'$ and $\nabla(\E'')\subseteq\Omega^1_X\otimes_{K_v}\E'$. Using this, we see that~$\omega\in\HH^0(X,\Omega^1_X(E))$. Thus, we have seen that every de Rham cohomology class~$[\E]\in\HH^1_\dR(X/K_v)$ is represented by a differential~$\omega\in\HH^0(X,\Omega^1_X(E))$, which establishes the existence of the claimed basis $\omega_1,\dots,\omega_{2g}$.
		\smallskip
		
		Now we come to the main part of the proof: showing that~$f$ has the claimed form. Choose a triple~$(\E,\sigma_b,\tau)$ representing~$f$, where~$\E$ has a weight filtration
		\[
		0\leq\W_0\E\leq\W_1\E\leq\W_2\E=\E
		\]
		as in Definition~\ref{def:coleman_weights}. As in Example~\ref{ex:coleman_weight_1}, we may suppose without loss of generality that~$\W_0\E=\O_X$ and the restriction of~$\tau$ to~$\W_0\E$ is the identity map on~$\O_X$, and that~$\gr^\W_2\E=\O_X$ and the image of~$\sigma_b$ in $\gr^\W_2\E_b=K_v$ is equal to~$1$. We have~$\gr^\W_1\E=\O_X^{\oplus m}$ and the extension
		\[
		0 \to \O_X \to \W_1\E \to \O_X^{\oplus m} \to 0
		\]
		is split by~$\tau$, so $\W_1\E=\O_X^{\oplus m+1}$, with connection
		\[
		\nabla = \d - \left(\begin{matrix}
			0 & \eta_1 & \eta_2 & \dots &  \eta_m \\
			0 & 0 & 0 & \cdots & 0 \\
			0 & 0 & 0 & \cdots & 0 \\
			\vdots & \vdots & \vdots & \ddots & \vdots \\
			0 & 0 & 0 & \cdots & 0
		\end{matrix}\right)
		\]
		for some differentials $\eta_1,\dots,\eta_m\in\HH^0(X,\Omega^1_X)$.
		
		Now we let~$\E'\subseteq\E$ denote the preimage of~$\O_X(-E_0)$ under the map $\E\to\gr^\W_2\E=\O_X$. As above, $\E'$ splits as an extension of~$\O_X(-E_0)$ by~$\W_1\E$, so we have $\E'=\O_X^{\oplus m+1}\oplus\O_X(-E_0)$, with~$\tau$ being the projection on the first factor. This splitting gives a trivialisation of~$\E$ over the generic point of~$X$ and, arguing as above, the connection on~$\E$ can be written with respect to this basis as
		\[
		\nabla = \d -
		\left(\begin{matrix}
			0 & \eta_1 & \eta_2 & \dots &  \eta_m & \eta' \\
			0 & 0 & 0 & \cdots & 0 & \eta_1' \\
			0 & 0 & 0 & \cdots & 0 & \eta_2' \\
			\vdots & \vdots & \vdots & \ddots & \vdots & \vdots \\
			0 & 0 & 0 & \cdots & 0 & \eta_m' \\
			0 & 0 & 0 & \cdots & 0 & 0
		\end{matrix}\right)
		\]
		for differentials $\eta_1,\dots,\eta_m\in\HH^0(X,\Omega^1_X)$, $\eta_1',\dots,\eta_m'\in\HH^0(X,\Omega^1_X(E))$ and~$\eta'\in\HH^0(X,\Omega^1_X(E+D))$. In other words, $f$ is given by the iterated Coleman integral
		\[
		f(z) = \sum_{k=1}^m\int_b^z\eta_k\eta_k'+\int_b^z\eta'+\sum_{k=1}^ma_k\int_b^z\eta_k+a_0 \,,
		\]
		where $(a_0,a_1,\dots,a_m,1)$ are the coordinates of~$\sigma_b\in\E_b=K_v^{\oplus m+2}$.
		
		The differentials $\eta_i'$ are of the second kind, since they represent classes inside $\HH^1_\dR(X/K_v)$. Hence, writing each $\eta_k$ and $\eta_k'$ in terms of the basis $\omega_i$, and using the identity $\int_b^z\omega(\d g)=\int_b^zg\omega-g(b)\int_b^z\omega$, we may rearrange~$f$ into the claimed form.
	\end{proof}
\end{lemma}

\begin{remark}\label{rmk:why_better}
	Note that the form of~$f$ in Lemma~\ref{lem:coleman_weight_2} is more restrictive than what is asserted in \cite[Proposition~2.3]{jen-netan:effective} in the case $D=\emptyset$, both in the number of quadratic terms ($2g^2$ instead of~$4g^2$) and the possible pole orders of~$h$. This is ultimately why we obtain better bounds than \cite{jen-netan:effective} in the case of quadratic Chabauty (see Example~\ref{ex:quadratic_chabauty}): we end up bounding the number of zeroes of a more restrictive class of functions than considered in \cite{jen-netan:effective}, so end up with better bounds.
\end{remark}

\subsection{Relation with the unipotent de Rham fundamental group}

As might be expected from the definition, the ring $\ACollog(Y)$ of Coleman algebraic functions is closely related to the unipotent de Rham fundamental group of~$Y$, in a sense we now make precise. Recall that the unipotent de Rham fundamental group~$U^\dR$ is defined to be the Tannaka group of the category $\MIC^\un(Y)$ of unipotent vector bundles with connection on~$Y$ at (the fibre functor associated to) the basepoint~$b$. Equivalently, via Deligne's canonical extension, $U^\dR$ is the Tannaka group of the category~$\MIC^\un(X,D)$ of unipotent vector bundles with log-connection on~$(X,D)$, again at the basepoint~$b$.

The de Rham fundamental group comes endowed with two extra structures: a weight filtration and a Hodge filtration. Of these, the weight filtration is easier to describe \cite[Definition~1.5]{asada-matsumoto-oda:local_monodromy}: it is the filtration given by setting:
\begin{itemize}
	\item $\W_{-1}U^\dR=U^\dR$;
	\item $\W_{-2}U^\dR$ is the kernel of the map from $U^\dR$ to the abelianised de Rham fundamental group of~$X$; and
	\item for~$k\geq3$, $\W_{-k}U^\dR$ is the subgroup-scheme generated by commutators of elements in~$\W_{-i}U^\dR$ and~$\W_{-j}U^\dR$ for $i+j=k$.
\end{itemize}

The description of the Hodge filtration is rather more complicated. The original definition is due to Wojtkowiak \cite[\S5.6]{wojtkowiak:cosimplicial_objects}, but for our applications we will only need the following description of~$\Fil^0$, due to Hadian \cite{hadian:motivic_pi_1}.

\begin{lemma}\label{lem:F^0_is_coherent}
	The subgroup $\Fil^0U^\dR$ is canonically isomorphic to the Tannaka group of the category of unipotent vector bundles on~$X$ (\emph{without} connection) at the basepoint~$b$. The inclusion~$\Fil^0U^\dR\hookrightarrow U^\dR$ corresponds under the Tannakian formalism to the functor $\MIC^\un(X,D)\to\Mod^\un(\O_X)$ given by forgetting the connection.
	\begin{proof}
		When~$Y$ is affine (i.e.\ $D$ is non-empty), this is \cite[Remark~3.8]{hadian:motivic_pi_1}; the projective case can be reduced to the affine case e.g.\ by the argument outlined in \cite{faltings:around}.
	\end{proof}
\end{lemma}

	

\begin{remark}
	There is a subtle technical point at play in the proof of Lemma~\ref{lem:F^0_is_coherent} above, which has to do with the definition of the Hodge filtration on the de Rham fundamental group. Namely, in most of the literature on the Chabauty--Kim method, e.g.\ in \cite{minhyong:selmer}, the definition of the Hodge filtration used is the one given by Wojtkowiak \cite[\S5.6]{wojtkowiak:cosimplicial_objects}, rather than that defined by Hadian \cite[Lemma~3.6]{hadian:motivic_pi_1}\footnote{There is a small error in the statement of \cite[Lemma~3.6]{hadian:motivic_pi_1}. Namely, in order that the filtration satisfying the list of properties is unique, one should also require that the element $p_n$ lies in $\Fil^0$. This is not, as claimed, a consequence of the other properties.}. These two definitions of the Hodge filtration presumably agree, but there does not appear to be a published proof of this fact in the literature\footnote{Richard Hain has suggested to me that one should be able to prove the equivalence of these two definitions by appealing to the Hodge theory of the unipotent fundamental group. Specifically, one can show using the theorem of the fixed part that the universal pro-unipotent $\C$-local system on a curve over~$\C$ can be made into a pro-unipotent variation of mixed Hodge structure in an essentially unique way (subject to a mild condition on the fibre over~$b$). So checking that Hadian's definition of the Hodge filtration agrees with Wojtkowiak's amounts to showing that the filtration defined in \cite[Lemma~3.6]{hadian:motivic_pi_1} underlies a variation of mixed Hodge structure. This should be an easy consequence of Hadian's constructions, but I have not checked this sufficiently carefully.}.
	
	So, for clarity, whenever we refer to the Hodge filtration on the de Rham fundamental groupoid in this paper, we always mean the one constructed by Hadian. This means that when we later use the comparison isomorphism between \'etale and de Rham fundamental groups, we cannot use the comparison isomorphism of Olsson \cite[Theorem~1.8]{olsson:towards}, since it is important for us that this comparison isomorphism is compatible with Hodge filtrations, and Olsson proves this only for Wojtkowiak's definition of the Hodge filtration \cite[Corollary~8.13]{olsson:affine_stacks}. Instead, we have to follow the approach described in \cite[\S7]{hadian:motivic_pi_1}\cite[\S5]{faltings:around}, and use a comparison isomorphism for unipotent fundamental groups deduced from Faltings' comparison isomorphism for cohomology with coefficients in crystalline local systems. Again, the comparison isomorphism constructed in this way presumably coincides with that defined by Olsson, but we do not need, or prove, this fact.
\end{remark}

Using this description of the weight and Hodge filtrations on~$U^\dR$, we can now make precise the relationship between~$U^\dR$ and the ring~$\ACollog(Y)$ of Coleman algebraic functions. Given an abstract Coleman algebraic function~$(\E,\sigma_b,\tau)$, we obtain a corresponding coordinate map $\psi_{(\E,\sigma_b,\tau)}\colon U^\dR\to\A^1$, namely the composite
\begin{equation}\label{eq:definition_of_psi}
U^\dR\to\AUT(\E_b)\xrightarrow{\tau_b\circ-}\HOM(\E_b,\A^1)\xrightarrow{\ev_{\sigma_b}}\A^1 \,,
\end{equation}
where the first map is the action of~$U^\dR$ on the fibre of~$\E$, the second is composition with the fibre of~$\tau$, and the third is evaluation at~$\sigma_b\in\E_b$. It is easy to check, using tensor-naturality of the action of~$U^\dR$ on the fibres~$\E_b$ for various~$\E$, that the construction $(\E,\sigma_b,\tau)\mapsto\psi_{(\E,\sigma_b,\tau)}$ defines a homomorphism
\[
\psi\colon \ACollog(Y)\to\O(U^\dR)
\]
of~$K_v$-algebras.

\begin{theorem}\label{thm:log-coleman_fns_are_coordinates_on_de_rham}
	The map $\psi\colon \ACollog(Y)\to\O(U^\dR)$ is an isomorphism onto the subring $\O(\Fil^0\backslash U^\dR)$, and is strictly compatible with the weight filtration.
\begin{proof}
We proceed in several steps.

\smallskip
\noindent\textit{Step 1: the image of~$\psi$ lies in~$\O(\Fil^0\backslash U^\dR)$.}

Suppose that~$f\in \ACollog(Y)$ is represented by a triple~$(\E,\sigma_b,\tau)$. It follows from Lemma~\ref{lem:F^0_is_coherent} and the fact that~$\tau$ is a morphism of vector bundles that the map $\tau_b\colon\E_b\to K_v$ is invariant under the action of~$\Fil^0U^\dR$. It follows that the map $U^\dR\to\HOM(\E_b,\A^1)$ given by the composite of the first two maps in~\eqref{eq:definition_of_psi} factors through~$\Fil^0\backslash U^\dR$, and hence that~$\psi(f)=\psi_{(\E,\sigma_b,\tau)}\in\O(\Fil^0\backslash U^\dR)$.

\smallskip
\noindent\textit{Step 2: the map $\psi$ preserves the weight filtration.}

Suppose that $f\in\W_m\ACollog(Y)$, so that~$f$ is represented by a triple~$(\E,\sigma_b,\tau)$ with~$\E$ of weight at most~$m$, as in Definition~\ref{def:coleman_weights}. We fix a weight filtration
\[
0=\W_{-1}\E\leq\W_0\E\leq\W_1\E\leq\dots\leq\W_m\E=\E
\]
on~$\E$. Since this weight filtration is a filtration by vector subbundles with log-connection, whose graded pieces are trivial, it follows that the action of~$U^\dR$ on~$\E_b$ is unipotent with respect to the induced filtration on~$\E_b$. Thus, the derivative of the action of~$U^\dR$ is given by a map
\begin{equation}\label{eq:derived_action}\tag{$\ast$}
\Lie(U^\dR) \to \W_{-1}\End(\E_b) \,,
\end{equation}
where $\W_{-1}\End(\E_b)$ denotes the vector space of endomorphisms of~$\E_b$ taking $\W_i\E_b$ into $\W_{i-1}\E_b$ for all~$0\leq i\leq m$.

Now by assumption, each quotient~$\W_i\E/\W_{i-2}\E$ is a vector bundle with connection on~$X$, so the action of~$U^\dR$ on $\W_i\E_b/\W_{i-2}\E_b$ factors through $U^\dR/\W_{-2}U^\dR$. It follows that~\eqref{eq:derived_action} takes~$\W_{-2}\Lie(U^\dR)$ into $\W_{-2}\End(\E_b)$, and hence that~\eqref{eq:derived_action} is a morphism of filtered (pro-)nilpotent Lie algebras.

Then, since the map $\tau_b\colon\E_b\to K_v$ is filtered, it follows that the map $U^\dR\to\HOM(\E_b,\A^1)$ given as the composite of the first two arrows in~\eqref{eq:definition_of_psi} is a morphism of filtered affine $K_v$-schemes, when $\HOM(\E_b,\A^1)=\Spec(\Sym^\bullet(\E_b))$ is given its natural filtration. Since~$\sigma_b\in\W_m\E_b$, it follows that $\psi(f)=\psi_{(\E,\sigma_b,\tau)}\in\W_m\O(U^\dR)$, as desired.

\smallskip
\noindent\textit{Step 3: the map $\psi$ is injective.}

We use the criterion of Lemma~\ref{lem:represents_zero}\eqref{lempart:log-analytic_represents_zero}. Suppose that $f\in \ACollog(Y)$ is represented by a triple~$(\E,\sigma_b,\tau)$, and let~$\nabla^{-\infty}\ker(\tau)$ denote the largest $\nabla$-stable subbundle of the kernel of~$\tau\colon\E\to\O_X$. Then the fibre of~$\nabla^{-\infty}\ker(\tau)$ at~$b$ is the largest $U^\dR$-subrepresentation of~$\E_b$ contained in the kernel of~$\tau_b\colon\E_b\to K_v$.

If~$f$ lies in the kernel of~$\psi$, then this says exactly that~$\tau_b(u(\sigma_b))=0$ for points~$u$ of~$U^\dR$ (valued in any $K_v$-algebra). In other words, $\sigma_b$ is contained in the largest $U^\dR$-subrepresentation of~$\E_b$ contained in the kernel of~$\tau_b$. Thus, $\sigma_b\in(\nabla^{-\infty}\ker(\tau))_b$, so~$f=0$ by Lemma~\ref{lem:represents_zero}.

\smallskip
\noindent\emph{Step~4: the map $\psi$ is surjective and strict for the weight filtration.}

Suppose that~$\alpha\in\W_m\O(\Fil^0\backslash U^\dR)$. The left-multiplication action of~$U^\dR$ on itself makes~$\O(U^\dR)$ into an ind-$U^\dR$-representation. The corresponding coaction of~$\O(U^\dR)$ \cite[\S4a]{milne:algebraic_groups} is just the left-comultiplication of the Hopf algebra~$\O(U^\dR)$. Since the Hopf algebra structure is compatible with the weight filtration by Example~\ref{ex:filtered_unipotent}, it follows that $\W_m\O(U^\dR)$ is a $U^\dR$-subrepresentation of~$\O(U^\dR)$, of finite dimension.

Now consider the map $\alpha^*\colon(\W_m\O(U^\dR))^\dual\to K_v$ given by evaluation on~$\alpha$. This map is $\Fil^0U^\dR$-invariant by choice of~$\alpha$. Hence by Lemma~\ref{lem:F^0_is_coherent}, the $U^\dR$-representation $(\W_m\O(U^\dR))^\dual$ corresponds under the Tannakian formalism to a vector bundle~$\E_m$ with log-connection on~$(X,D)$, and the map~$\alpha^*$ corresponds to a morphism~$\tau_\alpha\colon\E_m\to\O_X$ of vector bundles on~$X$. If we let~$\sigma_b\in(\W_m\O(U^\dR))^\dual$ denote the restriction of~$\W_m\O(U^\dR)$ of the counit of the Hopf algebra~$\O(U^\dR)$, then the triple~$(\E_m,\sigma_b,\tau_\alpha)$ is an abstract Coleman algebraic function. We write~$f_\alpha\in \ACollog(Y)$ for the element it represents.

Now if~$u\in U^\dR(\Lambda)$ for some $K_v$-algebra~$\Lambda$, then we have
\[
\psi(f_\alpha)(u)=\alpha^*(u(\sigma_b))=\alpha^*(\ev_u)=\alpha(u) \,,
\]
where~$\ev_u\colon\O(U^\dR)\to\Lambda$ is the ``evaluate at~$u$'' map. It follows that~$\psi(f_\alpha)=\alpha$, and hence that~$\psi$ is surjective.

To prove strictness, it suffices to prove that in fact~$f_\alpha\in\W_m\ACollog(Y)$, for which it suffices to prove that~$\E_m$ has weight~$\leq m$ in the sense of Definition~\ref{def:coleman_weights}. For this, we consider the filtration on~$\W_m\O(U^\dR)$ given by the subspaces $\W_i\O(U^\dR)$ for $0\leq i\leq m$. This is a filtration by $U^\dR$-subrepresentations. Since the $\O(U^\dR)$-coaction on each graded piece factors through the subcoalgebra $K_v=\W_0\O(U^\dR)\leq\O(U^\dR)$, it follows that the action of~$U^\dR$ on each graded piece is trivial. Similarly, since the $\O(U^\dR)$-coaction on each $\W_i\O(U^\dR)/\W_{i-2}\O(U^\dR)$ factors through the subcoalgebra $\W_1\O(U^\dR)\leq\O(U^\dR)$, it follows that the action of~$U^\dR$ on each of these partial quotients factors through $U^\dR/\W_{-2}$.

Translated through the Tannakian formalism, this filtration on~$\W_m\O(U^\dR)$ induces a corresponding filtration~$\W_\bullet$ on the associated vector bundle with log-connection~$\E_m^\vee$ on~$(X,D)$, whose graded pieces are all trivial, and such that each quotient $\W_i\E_m^\vee/\W_{i-2}\E_m^\vee$ is a vector bundle with connection on~$X$. Taking a suitable shift of the dual filtration on~$\E_m$, we see that~$f_\alpha\in\W_m\ACollog(Y)$, as desired.

\smallskip

This concludes the proof of Theorem~\ref{thm:log-coleman_fns_are_coordinates_on_de_rham}.
\end{proof}
\end{theorem}

As a consequence, we can determine the Hilbert series of~$\ACollog(Y)$.

\begin{corollary}\label{cor:coleman_hilbert_series}
	The Hilbert series of~$\ACollog(Y)$ is $\frac{1-gt}{1-2gt-(r-1)t^2}$.
	\begin{proof}
		The Hilbert series of~$\ACollog(Y)$ is the same as that of~$\Fil^0\backslash U^\dR$; we compute the latter Hilbert series instead. By Lemma~\ref{lem:quotient_by_unipotent_subgroup}, this Hilbert series is given by $\HS_{\Fil^0\backslash U^\dR}(t)=\HS_{U^\dR}(t)/\HS_{\Fil^0U^\dR}(t)$.
		
		Now~$U^\dR$ is the $K_v$-pro-unipotent completion of the surface group~$\Sigma_{g,r}$, so $\HS_{U^\dR}(t)=\frac1{1-2gt-(r-1)t^2}$ by Lemma~\ref{lem:surface_group}. Also,~$\Fil^0U^\dR$ is the free $K_v$-pro-unipotent group on~$g$ generators and the restriction of the weight filtration is the descending central series filtration. Thus $\O(\Fil^0U^\dR)$ is the tensor algebra on~$g$ generators in weight~$1$, and hence $\HS_{\Fil^0U^\dR}(t)=\frac1{1-gt}$. Combined, this yields the result.
	\end{proof}
\end{corollary}

\subsubsection{The de Rham Kummer map}\label{sss:de_rham_kummer}

We want to reinterpret Theorem~\ref{thm:log-coleman_fns_are_coordinates_on_de_rham} in terms of the \emph{de Rham Kummer map}\footnote{This map is also known as the \emph{de Rham unipotent Albanese map} in the works of Kim.} $\jj_\dR$ appearing in the Chabauty--Kim method \cite[\S1]{minhyong:selmer}. We recall the definition.

For any $K_v$-rational point~$y\in Y(K_v)$, one has the de Rham torsor of paths $P^\dR_y$, which represents the scheme $\ISO(\omega^\dR_b,\omega^\dR_y)$ of isomorphisms between the de Rham fibre functors $\MIC^\un(X,D)\to\Mod^\fin_{K_v}$ associated to~$b$ and~$y$, respectively. If~$y\in\Y(\O_v)$ is $\O_v$-integral, reducing to a point~$y_0\in\Y_0(k_v)$, then one also has the crystalline torsor of paths\footnote{The subscript~$K_v$ is there to remind the reader that, according to this definition, the crystalline path-torsor is a scheme over~$K_v$, not over its maximal unramified subfield~$K_{v,0}$, as might be expected. In fact, one can define a crystalline path-torsor $P^\cris_{y_0}$ over~$K_{v,0}$ \cite{shiho:crystalline_fundamental_groups_i}, and the crystalline path-torsor $P^\cris_{y_0,K_v}$ defined above is just the base change of $P^\cris_{y_0}$ to~$K_v$. Of course, in our applications we will have $K_v=\Q_\ellp$, so this distinction will be unnecessary.} $P^\cris_{y_0,K_v}$, which represents the scheme $\ISO(\omega^\cris_{b_0},\omega^\cris_{y_0})$ of isomorphisms between the crystalline fibre functors $\Isoc^\un_{K_v}(\Y_0)\to\Mod^\fin_{K_v}$ associated to~$b_0$ and~$y_0$, respectively. There is then an isomorphism
\begin{equation}\label{eq:berthelot-ogus}\tag{$\ast$}
P^\cris_{y_0,K_v}\cong P^\dR_y
\end{equation}
arising from Lemma~\ref{lem:equivalence_of_isocrystals_and_connections}.

The path-torsors $P^\dR_y$ and $P^\cris_{y_0,K_v}$ carry additional structures: the crystalline path-torsor $P^\cris_{y_0,K_v}$ carries a crystalline Frobenius automorphism $\varphi^f$ \cite[p.~8]{besser:tannakian}\footnote{The Frobenius automorphism is denoted by~$\varphi$ in \cite{besser:tannakian}. We prefer the notation~$\varphi^f$ to avoid confusion with the semilinear crystalline Frobenius on $P^\cris_{y_0}$ \cite[p.~594]{shiho:crystalline_fundamental_groups_i}. It is presumably the case that Besser's Frobenius automorphism \cite[p.~8]{besser:tannakian} is the $f$th power of Shiho's \cite[p.~594]{shiho:crystalline_fundamental_groups_i}, where~$f=[K_{v,0}:\Q_\ellp]$, though we have not carefully checked this. We persist in the notation~$\varphi^f$ nonetheless.}, and the de Rham path-torsor~$P^\dR_y$ carries a Hodge filtration~$\Fil^0$ on its affine ring. The subscheme~$\Fil^0\!P^\dR_y:=\Spec(\O(P^\dR_y)/\Fil^1)$ is non-empty -- indeed, it follows from a similar discussion to Lemma~\ref{lem:F^0_is_coherent} that $\Fil^0\!P^\dR_y$ is the Tannakian path-torsor from~$b$ to~$y$ in the category~$\Mod^\un(\O_X)$ of unipotent vector bundles on~$X$.

Besser has proved \cite[Corollary~3.2]{besser:tannakian} that $P^\cris_{y_0,K_v}$ has a unique $\varphi^f$-fixed $K_v$-point $\gamma^\cris_y$, while~$\Fil^0\!P^\dR_y$ has a $K_v$-point~$\gamma^\dR_y$ by virtue of being a torsor under~$\Fil^0U^\dR$. The element~$\gamma^\dR_y$ is unique up to left-multiplication by elements of $\Fil^0U^\dR(K_v)$. The de Rham non-abelian Kummer map
\[
\jj_\dR\colon\Y(\O_v)\to\Fil^0\backslash U^\dR(K_v)
\]
is then the map
\[
y\mapsto(\gamma^\dR_y)^{-1}\gamma^\cris_y \,,
\]
where we view both~$\gamma^\cris_y$ and~$\gamma^\dR_y$ as $K_v$-points of~$P^\dR_y$ via~\eqref{eq:berthelot-ogus}.
\smallskip

The coordinates of the de Rham non-abelian Kummer map are known to be Coleman analytic functions; we wish here to point out the slightly stronger fact that they are in fact Coleman \emph{algebraic}, and to clarify the role of the weight filtration. This is essentially a rephrasing of Theorem~\ref{thm:log-coleman_fns_are_coordinates_on_de_rham}.

\begin{theorem}\label{thm:log-coleman_and_de_rham_kummer}
	Let~$\alpha\in\O(\Fil^0\backslash U^\dR)$, and let~$f_\alpha\in \ACollog(Y)$ be the Coleman algebraic function corresponding to~$\alpha$ under the isomorphism of Theorem~\ref{thm:log-coleman_fns_are_coordinates_on_de_rham}. Then the composite
	\[
	\Y(\O_v)\xrightarrow{\jj_\dR}\Fil^0\backslash U^\dR(K_v) \xrightarrow{\alpha} K_v
	\]
	is equal to~$f_\alpha$. In particular, if $\alpha\in\W_m\O(\Fil^0\backslash U^\dR)$, then $\alpha\circ\jj_\dR$ is a Coleman algebraic function of weight at most~$m$.
	\begin{proof}
	Let~$(\E,\sigma_b,\tau)$ be a triple representing $f_\alpha\in \ACollog(Y)$. Being $K_v$-points of the de Rham path-torsor $P^\dR_y$, both~$\gamma^\cris_y$ and~$\gamma^\dR_y$ induce isomorphisms $\E_b\xrightarrow\sim\E_y$. Following through the definition of the isomorphism $\ACollog(Y)\cong\O(\Fil^0\backslash U^\dR)$ in~\eqref{eq:definition_of_psi}, we see that
	\[
	\alpha(\jj_\dR(y)) = \tau_b((\gamma^\dR_y)^{-1}\gamma_y^\cris(\sigma_b))
	\]
	for all~$y\in\Y(\O_v)$.
	
	Now on the one hand, $\Fil^0\!P^\dR_y$ is isomorphic to the Tannakian torsor of paths from~$b$ to~$y$ in the category~$\Mod^\un(\O_X)$ of unipotent vector bundles on~$X$, and the inclusion $\Fil^0\! P^\dR_y\hookrightarrow P^\dR_y$ corresponds to the tensor-functor $\MIC^\un(X,D)\to\Mod^\un(\O_X)$ given by forgetting the connection. Since the map $\tau\colon\E\to\O_X$ is a morphism of vector bundles, it follows that $\tau_b\circ(\gamma^\dR_y)^{-1}=\tau_y$ by naturality.
	
	On the other hand, it follows by definition of analytic continuation along Frobenius \cite[\S3]{besser:tannakian} that the square
	\begin{center}
	\begin{tikzcd}
		(\jota^\dagger\E^\an|_{]b_0[})^{\nabla=0} \arrow[r,"T^\nabla_{b_0,y_0}"]\arrow[d,"\wr"] & (\jota^\dagger\E^\an|_{]y_0[})^{\nabla=0} \arrow[d,"\wr"] \\
		\E_b \arrow[r,"\gamma^\cris_y"] & \E_y
	\end{tikzcd}
	\end{center}
	commutes, where the two vertical maps are given by evaluating a flat section at~$b$ and~$y$, respectively, as in Lemma~\ref{lem:equivalence_of_isocrystals_and_connections}.
	
	Put all together, this says that $\jj_\dR(y) = \tau_y(\sigma_{y_0}(y))$, where~$\sigma_{y_0}=T^\nabla_{b_0,y_0}(\sigma_{b_0})$ with~$\sigma_{b_0}$ the unique flat section of $\jota^\dagger\E^\an$ over the residue disc~$]b_0[$ which is equal to~$\sigma_b$ at~$b$. In other words, $\alpha(\jj_\dR(y))$ is given by evaluating the abstract Coleman function~$(\jota^\dagger\E^\an,\sigma,\jota^\dagger \tau^\an)$ at~$y$. Since this represents~$f_\alpha$, we thus have~$\alpha(\jj_\dR(y))=f_\alpha(y)$, as desired.
	\end{proof}
\end{theorem}

\subsubsection{Coleman algebraic functions associated to a quotient}

In our effective Chabauty--Kim method, we will use a slight refinement of the theory of Coleman algebraic functions which focuses just on a subcategory of $\MIC^\un(X,D)$. Let~$U$ be a quotient of~$U^\dR$. We say that an object~$\E$ of $\MIC^\un(X,D)$ is \emph{associated to~$U$} just when the~$U^\dR$-action on~$\E_b$ factors through~$U$.

\begin{definition}\label{def:coleman_associated}
	A Coleman algebraic function~$f$ is said to be \emph{associated} to~$U$ just when it can be represented by a triple $(\E,\sigma_b,\tau)$ where~$\E$ is associated to~$U$. We write~$\ACollog(Y)_U$ for the set of Coleman algebraic functions associated to~$U$, which is a subalgebra of~$\ACollog(Y)$, and endow it with the subspace filtration.
\end{definition}

\begin{remark}\label{rmk:associated_indpt_of_b}
	Just like the whole algebra $\ACollog(Y)$, the subalgebra $\ACollog(Y)_U$ is independent of the basepoint~$b$, in the following sense. Given a second basepoint~$b'$, the de Rham fundamental group ${U^\dR}'$ of~$Y$ based at~$b'$ is isomorphic to~$U^\dR$ via the isomorphism given by conjugating by a de Rham path. This isomorphism is unique up to conjugation by elements of~$U^\dR$, so the quotient~$U'$ of~${U^\dR}'$ corresponding to~$U$ under $U^\dR\simeq {U^\dR}'$ is independent of this choice. Since~$U$ and~$U'$ correspond to the same Tannakian subcategory of~$\MIC^\un(X,D)$, it follows that $\ACollog(Y)_U$ and $\ACollog(Y)_{U'}$ are equal as subalgebras $\ACol(\Y)$, as in Remark~\ref{rmk:log-analytic_indept_of_b}.
\end{remark}

Just as the algebra $\ACollog(Y)$ of all Coleman algebraic functions is closely linked to the whole de Rham fundamental group~$U^\dR$, the subalgebra $\ACollog(Y)_U$ is closely linked to the quotient~$U$. For instance, we have the following version of Theorem~\ref{thm:log-coleman_fns_are_coordinates_on_de_rham}.

\begin{proposition}\label{prop:log-coleman_for_quotients}
	The isomorphism $\ACollog(Y)\cong\O(\Fil^0\backslash U^\dR)$ from Theorem~\ref{thm:log-coleman_fns_are_coordinates_on_de_rham} restricts to a weight-filtered isomorphism $\ACollog(Y)_U\cong\O(\Fil^0\backslash U)$.
	\begin{proof}
		It is easy to see that the isomorphism $\psi\colon\ACollog(Y)\xrightarrow\sim\O(\Fil^0\backslash U^\dR)$ from Theorem~\ref{thm:log-coleman_fns_are_coordinates_on_de_rham} carries the subalgebra $\ACollog(Y)_U$ into $\O(\Fil^0\backslash U)=\O(U)\cap\O(\Fil^0\backslash U^\dR)$, since for any triple~$(\E,\sigma_b,\tau)$ where~$\E$ is associated to~$U$, the functional $\psi_{(\E,\sigma_b,\tau)}\colon U^\dR\to \A^1$ factors through~$U$. Conversely, if $\alpha\in\W_m\O(\Fil^0\backslash U)$, then we know from the proof of Theorem~\ref{thm:log-coleman_fns_are_coordinates_on_de_rham} that the corresponding Coleman algebraic function $f_\alpha$ is represented by the triple $(\E_m,\sigma_b,\tau_\alpha)$, where~$\E_m$ is the vector bundle with integrable logarithmic connection corresponding to $(\W_m\O(U^\dR))^*$ under the Tannakian formalism, and $\tau_\alpha\colon\E_m\to\O_X$ is the map corresponding the evaluation map $\ev_\alpha\colon(\W_m\O(U^\dR))^*\to K_v$. But since $\alpha\in\O(U)$, we know that $\ev_\alpha$ factors through $(\W_m\O(U))^*$, so $f_\alpha$ is also represented by a triple of the form $(\E_m^U,\sigma_b^U,\tau_\alpha^U)$, where $\E_m^U$ corresponds to the $U^\dR$-representation $(\W_m\O(U))^*$. Since $\E_m^U$ is associated to~$U$, it follows that so too is~$f_\alpha$.
		
		This proves that the isomorphism $\ACollog(Y)\cong\O(\Fil^0\backslash U^\dR)$ of Theorem~\ref{thm:log-coleman_fns_are_coordinates_on_de_rham} restricts to an isomorphism $\ACollog(Y)_U\cong\O(\Fil^0\backslash U)$. That this restricted isomorphism is a filtered isomorphism is automatic, since both sides have the subspace filtrations from $\ACollog(Y)$ and $\O(\Fil^0\backslash U^\dR)$.
	\end{proof}
\end{proposition}

In the proof of Proposition~\ref{prop:log-coleman_for_quotients}, the triple $(\E_m^U,\sigma_b^U,\tau_\alpha^U)$ representing the Coleman algebraic function $f_\alpha$ has the property that $\E_m^U$ is simultaneously associated to~$U$ and also of weight~$\leq m$. Hence we obtain the following consequence, which can also be proved by more elementary means.

\begin{corollary}\label{cor:associated_and_weight}
	Any~$f\in\W_m\ACollog(Y)_U$ can be represented by a triple~$(\E,\sigma_b,\tau)$ where~$\E$ is associated to~$U$ and has weight~$\leq m$.
\end{corollary}

\section{Bounding zeroes of Coleman algebraic functions}\label{s:weight_bound}

The final ingredient we will need in our proof of effective Chabauty--Kim is a method to bound the number of zeroes of Coleman algebraic functions. We keep notation as in the previous section. That is, $K_v/\Q_\ellp$ is a finite extension, $Y/K_v$ is a smooth hyperbolic curve, $X$ is the smooth compactification of~$Y$, and~$D$ is the complementary divisor. We assume that~$X$ is the generic fibre of a smooth proper curve~$\X/\O_v$ over the ring of integers~$\O_v$ of~$K_v$, and that the closure~$\Dvsr$ of~$D$ in~$\X$ is \'etale over~$\O_v$. We write~$\Y=\X\setminus\Dvsr$, and assume for simplicity that~$\Y(\O_v)\neq\emptyset$.
\smallskip

The result we prove in this section is an explicit upper bound on the number of zeroes of a Coleman algebraic function in terms of its weight. To state this result, we write~$g$ for the genus of~$X$ and~$r$ for the degree of the divisor~$D$, so that $2g+r>2$ by hyperbolicity. We write~$e_v$ and~$f_v$ for the ramification and residue class degrees of~$K_v/\Q_\ellp$, respectively. We set $\theta_v:=\left\lceil\frac{e_v+1}{\ellp-1}\right\rceil$, and define a positive constant~$\kappa_v$ by
\[
\kappa_v:=\ellp^{(\theta_v-1)f_v}\cdot\left(1+\frac{e_v}{(\theta_v-\frac{e_v}{\ellp-1})\log(\ellp)}\right) \,.
\]
In the particular case that~$K_v=\Q_\ellp$, the constants~$\kappa_v=\kappa_\ellp$ are as in Theorem~\ref{thm:main_bound}.

\begin{theorem}\label{thm:weight_bound}
	Let~$U$ be a quotient of~$U^\dR$, let~$m\geq1$, and suppose that~$f\in\W_m\ACollog(Y)_U$ is non-zero. Then the number of points of~$\Y(\O_v)$ at which~$f$ vanishes is at most
	\[
	\kappa_v\cdot\#\Y_0(k_v)\cdot(4g+2r-2)^m\cdot\prod_{i=1}^{m-1}(c_i+1) \,,
	\]
	where~$c_i:=\dim_{K_v}\gr^\W_i\!\ACollog(Y)_U$ are the coefficients of the Hilbert series of $\ACollog(Y)_U$, and~$\#\Y_0(k_v)$ is the number of points of the special fibre of~$\Y$ which are rational over the residue field of~$K_v$. (In the case~$m=1$, the empty product~$\prod_{i=1}^0(c_i+1)$ above has value~$1$.)
\end{theorem}

\begin{remark}
	The bound in Theorem~\ref{thm:weight_bound} is almost certainly non-optimal. In Section~\ref{s:siegel} we will give considerably better bounds in the case of~$\P^1\setminus\{0,1,\infty\}$. I was unable to find a similar construction in the general case.
\end{remark}

\begin{remark}
	For our applications, we will only need Theorem~\ref{thm:weight_bound} in the case $K_v=\Q_\ellp$. However, when it comes to the better bounds we will discuss in \S\ref{ss:better_omega}, it will be necessary to pass up finite extensions, so it is easiest to work over a general~$K_v$ to begin with.
\end{remark}


\subsection{Action of algebraic differential operators on Coleman algebraic functions}

For the proof of Theorem~\ref{thm:weight_bound}, we follow a strategy similar to that of~\cite{jen-netan:effective}. In outline, we will cook up a certain differential operator~$\DO$ such that~$\DO(f)=0$, and then use $\DO$ to analyse the Newton polygon of the power series of~$f$ on a residue disc and thereby the number of zeroes of~$f$.

To begin with, we describe the action of differential operators on Coleman algebraic functions. To do so, it is useful to introduce a slight generalisation of the ring of Coleman algebraic functions, by allowing them to take values in a line bundle. This definition follows \cite[Definition~4.1]{besser:tannakian}.

\begin{definition}\label{def:coleman_with_coeffs}
	Fix a basepoint~$b\in\Y(\O_v)$, and let~$\mathcal L$ be an $\O_X$-module. An \emph{abstract Coleman algebraic section of~$\mathcal L$} (or abstract Coleman algebraic section) is a triple $(\E,\sigma_b,\tau)$ consisting of:
	\begin{itemize}
		\item a vector bundle~$\E$ on~$X$ with log-connection along~$D$;
		\item an $\O_X$-linear map $\tau\colon\E\to\mathcal L$; and
		\item a point $\sigma_b\in\E_b$ in the fibre of~$\E$ at the basepoint~$b$.
	\end{itemize}

	We let~$\ACollog(Y,\mathcal L)$ denote the set of abstract Coleman algebraic sections, modulo the equivalence relation generated by identifying $(\E,\sigma_b,\tau)\sim(\E',\sigma'_b,\tau')$ whenever there is a morphism $f\colon\E\to\E'$ of vector bundles with log-connection such that $f(\sigma_b)=\sigma'_b$ and $\tau'\circ f=\tau$. The direct sum of vector bundles with log-connection induces on~$\ACollog(Y,\mathcal L)$ the structure of a $K_v$-vector space. As in Definitions~\ref{def:coleman_weights} and~\ref{def:coleman_associated} (cf.\ Corollary~\ref{cor:associated_and_weight}), we let~$\W_m\ACollog(Y,\mathcal L)_U$ denote the subspace of elements represented by triples~$(\E,\sigma_b,\tau)$ where~$\E$ has weight at most~$m$ and is associated to~$U$.
	
	Just as for Coleman algebraic functions, a Coleman algebraic section~$s$ of~$\mathcal L$ gives rise to a section~$s|_{]y_0[}$ of~$\mathcal L^\an$ over each residue disc~$]y_0[$ for~$y_0\in\Y_0(k_v)$. Specifically, if~$(\E,\sigma_b,\tau)$ is a triple representing~$s$, then we define~$s|_{]y_0[}:=\tau(T^\nabla_{b_0,y_0}(\sigma_{b_0}))$ where~$\sigma_{b_0}$ is the unique flat section of~$\E^\an$ over~$]b_0[$ which is equal to~$\sigma_b\in\E_b$ at~$b$, and~$T^\nabla_{b_0,y_0}$ denotes analytic continuation along Frobenius \cite[Definition~3.5]{besser:tannakian}.
\end{definition}

We will be interested in two particular cases. Firstly, we will be interested in the case that~$\mathcal L=\Mero_X$ is the constant sheaf with value the function field~$K_v(X)$. In this case, the ring~$\ACollog(Y,\Mero_X)$ might be called the ring of \emph{Coleman rational functions} on~$X$. Secondly, we will be interested in the case that~$\mathcal L=\Omega^1_X\otimes_{\O_X}\Mero_X$, in which case~$\ACollog(Y,\Omega^1_X\otimes_{\O_X}\Mero_X)$ might be called the space of \emph{Coleman rational differential forms} on~$X$. As this terminology suggests, the differential of an element of $\ACollog(Y,\Mero_X)$ is an element of $\ACollog(Y,\Omega^1_X\otimes_{\O_X}\Mero_X)$.


\begin{definition}\label{def:coleman_derivative}
	If~$\E$ is a unipotent vector bundle with log-connection on~$(X,D)$, we write~$\d$ for the natural connection on the Hom-sheaf~$\Hom_{\O_X}(\E,\Mero_X)$. Explicitly, this is the map
	\[
	\d\colon\Hom_{\O_X}(\E,\Mero_X) \to \Hom_{\O_X}(\E,\Omega^1_X\otimes_{\O_X}\Mero_X)
	\]
	sending an $\O_X$-linear map $\tau\colon\E\to\Mero_X$ to the map $\d f:=\d\circ \tau-(1\otimes \tau)\circ\nabla$, i.e.\ the difference between the two diagonal composites in the (not necessarily commuting) square
	\begin{center}
	\begin{tikzcd}
		\E \arrow[r,"\tau"]\arrow[d,"\nabla"] & \Mero_X \arrow[d,"\d"] \\
		\Omega^1_X(D)\otimes_{\O_X}\E \arrow[r,"1\otimes \tau"] & \Omega^1_X\otimes_{\O_X}\Mero_X \,.
	\end{tikzcd}
	\end{center}
	The Leibniz rules for~$\nabla$ and~$\d$ imply that~$\d f$ is $\O_X$-linear.
	
	Following~\cite[Definition~4.6]{besser:tannakian}, we define the \emph{de Rham differential}
	\[
	\d\colon \ACollog(Y,\Mero_X)\to \ACollog(Y,\Omega^1_X\otimes_{\O_X}\Mero_X)
	\]
	to be the map sending a triple~$(\E,\sigma_b,\tau)$ to~$(\E,\sigma_b,\d \tau)$.
\end{definition}

\begin{proposition}[cf.\ {\cite[Proposition~4.11]{besser:tannakian}}]
	Let $f\in \ACollog(Y,\Mero_X)$ have de Rham differential $\d f\in \ACollog(Y,\Omega^1_X\otimes_{\O_X}\Mero_X)$. Then for all $y_0\in\Y_0(k_v)$, $\d f|_{]y_0[}$ is the derivative of~$f|_{]y_0[}$ as a meromorphic function on the disc~$]y_0[$.
	\begin{proof}
		Let~$(\E,\sigma_b,\tau)$ be a triple representing~$f$. Then we have
		\[
		\d f|_{]y_0[} = (\d \circ \tau - (1\otimes \tau)\circ\nabla)(\sigma_{y_0}) = \d(\tau(\sigma_{y_0})) = \d(f|_{]y_0[}) \,,
		\]
		since the section~$i_{y_0}$ is flat. This is what we wanted to show.
	\end{proof}
\end{proposition}

Now fix a non-zero rational differential form $\omega\in\HH^0(X,\Omega^1_X\otimes_{\O_X}\Mero_X)$ on~$X$. An \emph{algebraic differential operator} is a sum
\[
\DO = \sum_{i=0}^Ng_i\cdot\frac{\d^i}{\omega^i}
\]
of powers of the differentiation operator~$\frac\d\omega$, where the coefficients~$g_i$ are rational functions on~$X$. We say that the operator~$\DO$ has \emph{order}~$N$ provided~$g_N\neq0$. The set of algebraic differential operators forms a non-commutative ring $K_v(X)[\d/\omega]$.

The de Rham differential induces an action of the non-commutative ring $K_v(X)[\d/\omega]$ on the ring $\ACollog(Y,\Mero_X)$ of Coleman rational functions.

\begin{lemma}
	Let~$\E$ be a unipotent vector bundle with log-connection on~$(X,D)$. Then there is an action of~$K_v(X)[\d/\omega]$ on~$\Hom_{\O_X}(\E,\Mero_X)$, where~$K_v(X)$ acts via multiplication on the constant sheaf~$\Mero_X$ and~$\frac\d\omega$ acts via the composite
	\[
	\Hom_{\O_X}(\E,\Mero_X) \xrightarrow\d \Hom_{\O_X}(\E,\Omega^1_X\otimes_{\O_X}\Mero_X) \xrightarrow{\omega^{-1}} \Hom_{\O_X}(\E,\Mero_X)
	\]
	where the first map is the derivative from Definition~\ref{def:coleman_derivative}, and the second map is induced from the isomorphism $\Omega^1_X\otimes_{\O_X}\Mero_X\cong\Mero_X$ given by dividing by~$\omega$.
	
	This induces an action of $K_v(X)[\d/\omega]$ on $\ACollog(Y,\Mero_X)$, where a differential operator~$\DO$ acts via $(\E,\sigma_b,\tau)\mapsto(\E,\sigma_b,\DO(\tau))$.
	\begin{proof}
		For the first part, we need to verify that the action of~$\frac\d\omega$ on $\Hom_{\O_X}(\E,\Mero_X)$ satisfies $\frac\d\omega(g\tau)=g\frac\d\omega(\tau)+\frac{\d g}\omega$ for all $\tau\in\Hom_{\O_X}(\E,\Mero_X)$ and all $g\in K_v(X)$. This is easy to check from the definitions, since
		\begin{align*}
		\frac\d\omega(g\tau) &= \frac{\d\circ(g\tau)-(1\otimes(g\tau)\circ\nabla)}\omega \\
		 &= \frac{g\cdot\d\circ \tau-(1\otimes g)\cdot(1\otimes \tau)\circ\nabla+(\d g)\cdot \tau}\omega = g\frac\d\omega(\tau)+\frac{\d g}\omega \,.
		\end{align*}
		The second part (action on $\ACollog(Y,\Mero_X)$) is then clear, since the action of~$K_v(X)[\d/\omega]$ on~$\Hom_{\O_X}(\E,\Mero_X)$ is clearly natural in~$\E$.
	\end{proof}
\end{lemma}

\subsubsection{Pole orders}

In what follows, we will also need to control the ``pole orders'' of Coleman log-meromorphic functions. For this, we write~$\div(\omega)$ for the divisor of~$\omega$ as a section of~$\Omega^1_X(D)$ (so~$\div(\omega)$ is a log-canonical divisor). We write $\supp(\div(\omega))\cap Y$ for the support of~$\div(\omega)$ outside~$D$, viewed as a reduced divisor, and set $\div(\omega)^+:=\div(\omega)+(\supp(\div(\omega))\cap Y)$.

We say that an algebraic differential operator $\DO=\sum_{i=0}^Ng_i\frac{\d^i}{\omega^i}$ is \emph{regular outside~$\div(\omega)$} just when each~$g_i$ has no poles outside~$\div(\omega)$. If~$\DO\neq0$ is regular outside~$\div(\omega)$, then we define (in an ad hoc manner) its \emph{divisor} to be
\[
\div(\DO) := \min\left(\left(\div(g_i)-i\div(\omega)^+\right)_{0\leq i\leq N},0\right) \,,
\]
where~$\min$ denotes the pointwise minimum of divisors. This is an anti-effective divisor, and it is easily verified that the divisor of an algebraic differential operator is superadditive under composition, i.e.\
\[
\div(\DO'\circ\DO)\geq\div(\DO')+\div(\DO)
\]
for all $\DO,\DO'\in K_v(X)[\d/\omega]\setminus\{0\}$.

Now for any divisor~$E$ on~$X$, the space $\ACollog(Y,\O_X(E))$ is a subspace of $\ACollog(Y,\Mero_X)$, and we can control how algebraic differential operators act on these spaces.

\begin{lemma}\label{lem:differential_operators_increase_poles}
	Let~$E$ be a divisor on~$X$ and let~$\DO\in K_v(X)[\d/\omega]$ be a non-zero algebraic differential operator which is regular outside~$\div(\omega)$. Assume that the support of~$E$ is contained in the support of~$\omega$. Then
	\[
	\DO\left(\ACollog(Y,\O_X(E))\right) \subseteq \ACollog(Y,\O_X(E-\div(\DO)))
	\]
	as subspaces of~$\ACollog(Y,\Mero_X)$.
	\begin{proof}
		It suffices to prove the result in the special case $\DO=\frac\d\omega$, the general case then following by linearity. For this, we simply note that for any $\tau\in\Hom_{\O_X}(\E,\O_X(E))$, the Coleman differential $\d \tau$ is given by the difference between the two diagonal composites in the square
		\begin{center}
		\begin{tikzcd}
			\E \arrow[r,"\tau"]\arrow[d,"\nabla"] & \O_X(E) \arrow[d,"\d"] \\
			\Omega^1_X(D)\otimes_{\O_X}\E \arrow[r,"1\otimes \tau"] & \Omega^1_X(E+(D\cup\supp(\div(\omega)))) \,.
		\end{tikzcd}
		\end{center}
		Composing with the isomorphism $\Omega^1_X(E+(D\cup\supp(\div(\omega))))\cong\O_X(E+\div(\omega)^+)$ given by division by~$\omega$, we see that~$\frac\d\omega(\tau)\in\Hom_{\O_X}(\E,\O_X(E+\div(\omega)^+))$. This in turn implies that $\frac\d\omega(f)\in \ACollog(Y,\O_X(E+\div(\omega)^+))$ for all $f\in \ACollog(Y,\O_X(E))$, which is what we wanted to show.
	\end{proof}
\end{lemma}

\subsection{PD-nice differential operators}

The theory which enables us to bound the number of zeroes of a Coleman algebraic function~$f$ via a differential operator is the ``nice differential operators'' machinery of \cite[\S3]{jen-netan:effective}. We will in fact use a slight variant on the theory, systematically replacing power series rings with their divided power envelopes. One advantage of this variant is that it deals more naturally with the case~$\ellp=2$, which is not addressed by the theory in~\cite{jen-netan:effective}.
\smallskip

Let $\O_v\llbrack t\rrbrack^\PD$ denote the algebra of divided power series with coefficients in~$\O_v$; that is, the ring of power series $f=\sum_{i=0}^\infty a_i\frac{t^i}{i!}$ with $a_i\in\O_v$. This is a subring of the ring of power series convergent on the open disc of radius $\ellp^{-1/(\ellp-1)}$ (with respect to the norm on $K_v$ satisfying $|\ellp|=\ellp^{-1}$).

\begin{definition}\label{def:nice}
	A differential operator $\DO = \sum_{i=0}^Ng_i(t)\frac{\d^i}{\d t^i}\in K_v\llpara t\rrpara[\frac\d{\d t}]$ in the formal variable~$t$ is called \emph{PD-nice}\footnote{Standing for ``pretty darn nice''.} just when each $g_i\in\O_v\llbrack t\rrbrack^\PD$ is a divided power series with $\O_v$-coefficients for $0\leq i\leq N$ and $g_N\in\O_v\llbrack t\rrbrack^{\PD,\times}$ is an invertible divided power series. We refer to~$N$ as the \emph{order} of~$\DO$.
\end{definition}

The following is easily verified.

\begin{lemma}
	The composite of two PD-nice differential operators is PD-nice.
\end{lemma}

As in \cite[Proposition~3.2]{jen-netan:effective}, PD-nice differential operators can be used to bound the number of zeroes of power series.

\begin{proposition}\label{prop:nice}
	Let~$\DO$ be a PD-nice differential operator of order~$N$. Suppose that~$f\in K_v\llbrack t\rrbrack$ is a non-zero power series such that $\DO(f)=0$. Then~$f$ converges on the open disc of radius~$\ellp^{-1/(\ellp-1)}$, and for every~$\lambda>\frac1{\ellp-1}$, the number of zeroes\footnote{Here, ``zeroes'' means zeroes defined over the completed algebraic closure of~$K_v$, counted with multiplicity.} of~$f$ in the closed disc of radius $\ellp^{-\lambda}$ is at most
	\[
	\left(1+\frac1{(\lambda-\frac1{\ellp-1})\log(\ellp)}\right)\cdot(N-1) \,.
	\]
\end{proposition}

\begin{remark}
	Proposition~\ref{prop:nice} in the case $K_v=\Q_\ellp$ is roughly analogous to \cite[Proposition~3.2]{jen-netan:effective}, which constrains the number of $\Q_\ellp$-rational\footnote{The phrasing of \cite[Proposition~3.2]{jen-netan:effective} is somewhat unclear as to what field the zeroes of~$f$ are defined over, but what is proven is a bound on the $\Q_\ellp$-rational zeroes.} zeroes of a power series~$f$ in terms of the number of~$\C_\ellp$-rational zeroes of~$\DO(f)$ when~$\ellp$ is odd. In fact, it is relatively straightforward to prove \cite[Proposition~3.2]{jen-netan:effective} from Proposition~\ref{prop:nice}, though we omit the proof here.
	
	Proposition~\ref{prop:nice} also fixes a small gap in the proof of \cite[Theorem~1.1]{jen-netan:effective}. Specifically, in order to apply \cite[Proposition~3.2]{jen-netan:effective} to the power series~$G$ in the proof of Theorem~1.1(i) on p.~1069 of \cite{jen-netan:effective}, it is necessary to know that~$\DO(G)$ is non-zero, and it is not obvious why this should be the case. Proposition~\ref{prop:nice} clarifies what happens in the case that~$\DO(G)=0$: since~$\DO$ has order~$3g+1$, the number of $\Q_\ellp$-rational zeroes of~$G$ is at most~$\kappa_\ellp\cdot3g$, and this is even better than the bound obtained in the case $\DO(G)\neq0$.
\end{remark}

For the proof of Proposition~\ref{prop:nice}, we use the following result, which asserts that if the differential operator~$\DO$ is PD-nice, then the differential equation~$\DO(f)=0$ has a full basis of solutions over the PD-power series.

\begin{proposition}\label{prop:PD_primitives}
	Let~$\DO$ be a PD-nice differential operator of order~$N$ and let~$a_0,\dots,a_{N-1}\in\O_v$. Then there is a unique sequence $a_N,a_{N+1},\dots$ of elements of~$\O_v$ such that the divided power series $f:=\sum_{i=0}^\infty a_i\frac{t^i}{i!}\in\O_v\llbrack t\rrbrack^\PD$ satisfies~$\DO(f)=0$.
	\begin{proof}
		Write the differential operator~$\DO$ as $\sum_{l=0}^Ng_l(t)\frac{\d^l}{\d t^l}$, and write each~$g_l$ as $\sum_{i=0}^\infty b_{li}\frac{t^i}{i!}$. We may assume without loss of generality that~$g_N=1$. Thus, $f$ satisfies $\DO(f)=0$ if and only if
		\[
		a_{N+k}+\sum_{l=0}^{N-1}\sum_{i+j=k}{k\choose i}b_{li}a_{l+i} = 0
		\]
		for all~$k\geq0$. This sets up a recurrence relation giving the values of $a_{N+k}$ in terms of~$a_0,\dots,a_{N+k-1}$, and hence~$f$ is uniquely determined by the values of~$a_0,\dots,a_{N-1}$.
	\end{proof}
\end{proposition}

\begin{proof}[Proof of Proposition~\ref{prop:nice}]
	By Proposition~\ref{prop:PD_primitives}, the kernel of~$\DO$ on~$\O_v\llbrack t\rrbrack^\PD$ is a free~$\O_v$-module of rank~$N$, while the kernel of~$\DO$ on~$K_v\llbrack t\rrbrack$ is a $K_v$-vector space of dimension~$N$. Thus, if~$f\in K_v\llbrack t\rrbrack$ satisfies $\DO(f)=0$, then $f\in K_v\otimes_{\O_v}\O_v\llbrack t\rrbrack^\PD$, so~$f$ converges on the open disc of radius~$\ellp^{-1/(\ellp-1)}$.
	
	Now rescaling~$f$ if necessary, we may assume that~$f\in\O_v\llbrack t\rrbrack^\PD$ but $f\notin\m_v\O_v\llbrack t\rrbrack^\PD$ where~$\m_v\unlhd\O_v$ is the maximal ideal. If we write~$f=\sum_{i=0}^\infty a_i\frac{t^i}{i!}$, then by Proposition~\ref{prop:PD_primitives} we must have $a_I\in\O_v^\times$ for some~$I<N$.
	
	If~$f$ has at most~$N-1$ zeroes on the closed disc of radius~$\ellp^{-\lambda}$, then we are certainly done. If instead~$f$ has $L\geq N$ zeroes on this closed disc, then by the usual Newton polygon considerations \cite[Corollary~6.4.11]{gouvea:intro_to_p-adics} we have
	\[
	v_\ellp(a_i/i!) \geq v_\ellp(a_L/L!) + \lambda(L-i)
	\]
	for all~$i\leq L$. Taking $i=I$, this then implies that
	\begin{equation}\label{eq:nice_factorial_inequality}\tag{$\ast$}
		\lambda(L-I) \leq v_\ellp(L!/I!) \,.
	\end{equation}
	But by \cite[Lemma~3.4]{jen-netan:effective}, we have\footnote{Strictly speaking, the first inequality here only holds for~$I\neq0$. Nonetheless, the inequality between the outer terms holds even in the case~$I=0$.} $v_\ellp(L!/I!)\leq\log_\ellp(I)+\frac{L-I}{\ellp-1}\leq \frac I{\log(\ellp)}+\frac{L-I}{\ellp-1}$. Combined with \eqref{eq:nice_factorial_inequality}, this yields
	\[
	L \leq \left(1+\frac1{(\lambda-\frac1{\ellp-1})\log(\ellp)}\right)\cdot I \leq \left(1+\frac1{(\lambda-\frac1{\ellp-1})\log(\ellp)}\right)\cdot(N-1)
	\]
	as desired.
\end{proof}

\subsection{Constructing a PD-nice algebraic differential operator}\label{ss:pd-nice_construction}


Now we come to the proof of Theorem~\ref{thm:weight_bound}. Fix an $\O_v$-integral point~$b\in\Y(\O_v)$. We will bound the number of $\O_v$-integral zeroes of~$f$ in a small disc about~$b$, obtaining Theorem~\ref{thm:weight_bound} by summing these bounds over discs covering $\Y(\O_v)$. We choose an integral logarithmic differential $\omega\in\HH^0(\X,\Omega^1_\X(\Dvsr))$ on~$\X$ which does not vanish at~$b_0$ on the special fibre of~$\X$. This can be arranged, for example, by choosing a logarithmic differential~$\omega_0$ on the special fibre which doesn't vanish at~$b_0$ \cite[Lemma~IV.5.1]{hartshorne:algebraic_geometry}, and lifting~$\omega_0$ to an integral logarithmic differential by Grauert's Theorem \cite[Corollary~III.12.9]{hartshorne:algebraic_geometry}.

We also fix an integral local parameter~$t$ at~$b$, i.e.\ a rational function on~$X$ which vanishes at~$b$ such that the maximal ideal of~$\X$ at~$b_0$ is generated by~$t$ and a uniformiser of~$K_v$. Taking Laurent series expansions provides an injective homomorphism
\[
K_v(X)[\d/\omega] \hookrightarrow K_v\llpara t\rrpara[\d/\d t]
\]
from the algebra of algebraic differential operators to the algebra of formal differential operators in the variable~$t$. We say that an algebraic differential operator~$\DO$ is \emph{PD-nice} (at~$b$) just when its Laurent series expansion is a PD-nice differential operator in the sense of Definition~\ref{def:nice}.

Our aim is to construct a PD-nice algebraic differential operator~$\DO_m$ annihilating all Coleman algebraic functions of weight at most~$m$, and to control the order and divisor of such a differential operator. Here is the precise result we will prove.

\begin{proposition}\label{prop:operator_kills_coleman_fns}
	Let~$U$ be a quotient of~$U^\dR$. For every~$m\geq1$, there is a PD-nice algebraic differential operator~$\DO_m$ of order at most
	\[
	(\deg(\div(\omega)^+)+2)^m\cdot\prod_{i=1}^{m-1}(c_i+1) \,,
	\]
	which is regular outside $\div(\omega)$, with divisor
	\[
	\div(\DO_m) \geq -(\deg(\div(\omega)^+)+2)^m\cdot\prod_{i=1}^m(c_i+1)\cdot\div(\omega)^+ \,,
	\]
	such that~$\DO_m(f)=0$ for all~$f\in\W_m\ACollog(Y)_U$. Here, the constants $c_i=\dim_{K_v}\gr^\W_i\!\ACollog(Y)_U$ are the coefficients of the Hilbert series of $\ACollog(Y)_U$. (In the case~$m=1$, the empty product~$\prod_{i=1}^0$ is to be interpreted as~$1$.)
\end{proposition}

Before we come to the proof of this proposition, let us describe how it completes the proof of Theorem~\ref{thm:weight_bound}.

\begin{proof}[Proof of Theorem~\ref{thm:weight_bound}]
	Let $\Disc_b\subset]b_0[$ denote the closed subdisc of the residue disc of~$b$ defined by the inequality $|t|\leq \ellp^{-\theta_v/e_v}$, where~$\theta_v=\left\lceil\frac{e_v+1}{\ellp-1}\right\rceil$ as at the beginning of this section. By the identity principle for Coleman functions \cite[Corollary~4.13]{besser:tannakian}, $f$ is not identically zero on the residue disc~$]b_0[$, so applying Proposition~\ref{prop:nice} with $\lambda=\theta_v/e_v$ to the differential operator~$\DO_m$ from Proposition~\ref{prop:operator_kills_coleman_fns} shows that~$f$ has at most
	\begin{equation}\label{eq:discwise_bound}\tag{$\ast$}
	\left(1+\frac{e_v}{(\theta_v-\frac{e_v}{\ellp-1})\log(\ellp)}\right)\cdot(\deg(\div(\omega)^+)+2)^n\cdot\prod_{i=1}^{n-1}(c_i+1)
	\end{equation}
	zeroes on~$\Disc_b$.
	
	Now the the disc~$\Disc_b$ contains all $\O_v$-integral points of~$\Y$ reducing to the same point as~$b$ modulo $\m_v^{\theta_v}$, where~$\m_v\unlhd\O_v$ is the maximal ideal. In other words, the number of zeroes of~$f$ on the fibre of the reduction map~$\Y(\O_v)\to\Y(\O_v/\m_v^{\theta_v})$ containing~$b$ is at most~\eqref{eq:discwise_bound}. Since this holds for all~$b\in\Y(\O_v)$, we obtain the bound claimed in Theorem~\ref{thm:weight_bound} by summing~\eqref{eq:discwise_bound} over the fibres of the reduction map, using that $\#\Y(\O_v/\m_v^{\theta_v})=\ellp^{(\theta_v-1)f_v}\cdot\#\Y_0(k_v)$ by Hensel's Lemma along with the bound $\deg(\div(\omega)^+)\leq2\deg(\div(\omega))=4g+2r-4$.
\end{proof}

In preparation for the proof of Proposition~\ref{prop:operator_kills_coleman_fns}, we prove two preparatory results, which will enable us to construct the operators~$\DO_n$ recursively.

\begin{lemma}\label{lem:operator_gives_rational_fns}
	Suppose that~$\DO_m$ is an algebraic differential operator which vanishes on $\W_m\ACollog(Y)_U$. Then for every~$f\in\W_{m+1}\ACollog(Y)_U$, $\DO_m(f)$ is a rational function, lying in~$\HH^0(X,\O_X(-\div(\DO_m)))$.
	\begin{proof}
		Suppose that~$f\in\W_{m+1}\ACollog(Y)_U$ is represented by a triple~$(\E,\sigma_b,\tau)$, where~$\E$ is associated to~$U$ and admits a weight filtration
		\[
		0\leq\W_0\E\leq\W_1\E\leq\dots\leq\W_m\E\leq\W_{m+1}\E=\E
		\]
		as in Definition~\ref{def:coleman_weights}. According to Lemma~\ref{lem:differential_operators_increase_poles}, $\DO_m(f)$ is a Coleman algebraic function valued in~$\O_X(-\div(\DO_m))$, represented by the triple~$(\E,\sigma_b,\DO_m(\tau))$.
		
		We claim that~$\DO_m(\tau)$ vanishes on~$\W_m\E$. Indeed, if~$\sigma'_b\in\W_m\E_b$, then the triple $(\E,\sigma'_b,\tau)$ represents a Coleman algebraic function of weight at most~$m$ associated to~$U$. It follows by assumption that~$(\E,\sigma'_b,\DO_m(\tau))$ represents the zero Coleman function, so $\DO_m(\tau)(\sigma'_b)=0$. This implies that~$\DO_m(\tau)$ vanishes on~$\W_n\E$.
		
		Thus $\DO_m(\tau)$ factors  through a map $\tau'\colon\E/\W_m\E\to\O_X(-\div(\DO_m))$, and $\DO_m(f)$ is also represented by the triple~$(\E/\W_m\E,\sigma_b,\tau')$. Since~$\E/\W_m\E$ is a trivial vector bundle with connection, it follows as in Example~\ref{ex:coleman_weight_0} that~$\DO_m(f)$ is given by an algebraic section of $\O_X(-\div(\DO_m))$, namely the image of the unit section of~$\O_X$ under the composite
		\[
		\O_X\to\E/\W_m\E\xrightarrow{\tau'}\O_X(-\div(\DO_m)) \,,
		\]
		where the first map is the map sending~$1\in K_v=\O_{X,b}$ to~$\sigma_b$.
	\end{proof}
\end{lemma}

\begin{lemma}\label{lem:annihilate_rational_functions}
	Let~$E$ be a divisor on~$X$ whose support is contained in the support of~$\omega$, and let~$V$ be a subspace of~$\HH^0(X,\O_X(E))$ of dimension~$c$. Then there is a PD-nice algebraic differential operator
	\[
	\DO = \sum_{i=0}^Ng_i\frac{\d^i}{\omega^i}
	\]
	of order
	\[
	N\leq\deg(E)+1
	\]
	which is regular outside $\div(\omega)$, whose divisor satisfies
	\[
	\div(\DO) \geq - cE - (c+1)N\cdot\div(\omega)^+ \,,
	\]
	and which vanishes on~$V$.
\begin{proof}
	Let~$\mathcal V^\PD$ denote the set of elements of $V\subseteq\HH^0(X,\O_X(E))$ whose power series expansion at~$b$ lies in~$\O_v\llbrack t\rrbrack^\PD$. This is an $\O_v$-lattice in~$V$.
	
	Let~$f_1,\dots,f_c$ be an $\O_v$-basis of~$\mathcal V^\PD$, and write each~$f_i$ as~$\sum_{j=0}^\infty a_{ij}\frac{t^j}{j!}$ with $a_{ij}\in\O_v$. For each~$i$, we let~$n_i$ denote the least integer such that~$a_{in_i}\in\O_v^\times$ (such an index exists since~$f_i\notin\m_v\mathcal V^\PD$). Changing the basis~$f_1,\dots,f_c$ if necessary, we may assume that $0\leq n_1<n_2<\dots<n_c$.
	
	Now following \cite[\S4.1]{jen-netan:effective} we set~$N=n_{c+1}:=n_c+1$, and consider the $c\times(c+1)$ matrix~$F$ given by
	\[
	F_{ij}:=\frac{\d^{n_j}}{\omega^{n_j}}(f_i) \,,
	\]
	whose coefficients are rational functions on~$X$, without poles outside $\div(\omega)$ by assumption that~$E$ is contained in the support of~$\omega$. The desired differential operator~$\DO$ is defined by
	\[
	\DO := \sum_{j=1}^{c+1}(-1)^{j+1}\det(F^{(j)})\frac{\d^{n_j}}{\omega^{n_j}} \,,
	\]
	where~$F^{(j)}$ denotes the~$c\times c$ matrix obtained by deleting the~$j$th column of~$F$. The differential operator~$\DO$ is clearly algebraic of degree~$N$ and regular outside~$\div(\omega)$, so it remains to verify the following:
	\begin{enumerate}
		\item\label{lempart:annihilate_annihilates} $\DO$ annihilates $V$;
		\item\label{lempart:annihilate_nice} $\DO$ is PD-nice;	
		\item\label{lempart:annihilate_order} $N\leq\deg(E)+1$; and
		\item\label{lempart:annihilate_pole_order} $\det(F^{(j)})\in\HH^0(X,\O_X(cE+((c+1)N-n_j)\cdot\div(\omega)^+))$ for all~$j$.
	\end{enumerate}
	\smallskip
	
	\noindent\eqref{lempart:annihilate_annihilates} For any rational function~$f$, $\DO(f)$ is the determinant of the $(c+1)\times(c+1)$ matrix $F(f)$ with entries
	\[
	F(f)_{ij}=
	\begin{cases}
		\frac{d^{n_j}}{\omega^{n_j}}(f_i) & \text{if $i\leq c$,} \\
		\frac{\d^{n_j}}{\omega^{n_j}}(f) & \text{if $i=c+1$.}
	\end{cases}
	\]
	But if~$f=f_i$ for some~$i$, then the $i$th and $(c+1)$th rows of~$F(f)$ are equal, and hence $\det(F(f))=0$. Thus we have~$\DO(f_i)=0$ for all~$i$, and so~$\DO$ annihilates~$V$.
	\smallskip
	
	\noindent\eqref{lempart:annihilate_nice} Since the differential form~$\omega$ was chosen to be regular and non-vanishing at~$b_0$, it follows that the ratio $u:=\frac{\d t}{\omega}$ is a rational function on~$\X$ which is regular and non-vanishing at~$b_0$. In particular, the Taylor expansion of~$u$ is a unit in~$\O_v\llbrack t\rrbrack$. It follows that the differential operator $\frac\d\omega=u\frac\d{\d t}$ is PD-nice; in particular, it preserves $\O_v\llbrack t\rrbrack^\PD$. It follows that all the coefficients of the matrix~$F$ lie in~$\O_v\llbrack t\rrbrack^\PD$, and hence so too do the coefficients~$g_i$ of the differential operator~$\DO$.
	
	It remains to show that the leading coefficient~$g_N$ of~$\DO$ is a unit in~$\O_v\llbrack t\rrbrack^\PD$. For this, we let $\m_n$ denote the ideal of~$\O_v\llbrack t\rrbrack^\PD$ consisting of those divided power series $\sum_{j=0}^\infty a_j\frac{t^j}{j!}$ such that $a_j\in\m_v$ for $j<n$. It follows from the definition of~$n_j$ that
	\[
	f_j \equiv A\cdot\frac{t^{n_j}}{n_j!} \text{ mod }\m_{n_j+1}
	\]
	for some constant~$A\in\O_v^\times$ depending on~$j$. It follows by an inductive argument using that $\frac\d\omega=u\frac\d{\d t}$ for $u\in\O_v\llbrack t\rrbrack^\times$, that for all $n\leq n_j$ we have
	\[
	\frac{\d^n}{\omega^n}(f_j) \equiv A\cdot\frac{t^{n_j-n}}{(n_j-n)!} \text{ mod }\m_{n_j-n+1}
	\]
	for some constant~$A\in\O_v^\times$ depending on~$j$ and~$n$.
	
	In particular, we see that $F_{ij}=\frac{\d^{n_j}}{\omega^{n_j}}(f_i)$ is zero modulo~$\m_1$ if $j<i$ and is non-zero modulo~$\m_1$ if $j=i$. In other words, the matrix~$F$ modulo~$\m_1$ is upper triangular with non-zero diagonal entries. In particular, $\det(F^{(c+1)})$ is non-zero modulo~$\m_1$, and hence a unit in~$\O_v\llbrack t\rrbrack^\PD$. This is what we wanted to prove.
	\smallskip
	
	\noindent\eqref{lempart:annihilate_order} The rational function $f_c\in\HH^0(X,\O_X(E))$ has no poles on the disc~$]b_0[$, so its power series expansion is convergent on the open disc~$|t|<1$, and has at most~$\deg(E)$ zeroes on this disc. On the other hand, it follows from the definition of~$n_c$ that the divided power coefficients $a_{cj}$ of~$f_c$ satisfy $v_\ellp(a_{cj}/j!)>v_\ellp(a_{cn_c}/n_c!)$ for all~$j<n_c$. By the usual Newton polygon considerations, it follows that~$f_c$ has at least~$n_c$ zeroes on the open unit disc. This implies that $N=n_c+1\leq\deg(E)+1$, as desired.
	\smallskip
	
	\noindent\eqref{lempart:annihilate_pole_order} Since each~$f_i\in\HH^0(X,\O_X(E))$ and the support of~$E$ is contained in the support of~$\omega$, we have
	\[
	F(f)_{ij}\in\HH^0(X,\O_X(E+n_j\cdot\div(\omega)^+)) \,.
	\]
	Since~$n_i\leq N$ for all~$i$, this in turn implies that
	\[
	\det(F^{(j)})\in\HH^0(X,\O_X(cE+((c+1)N-n_j)\cdot\div(\omega)^+)) \,,
	\]
	as desired.
\end{proof}
\end{lemma}

\begin{proof}[Proof of Proposition~\ref{prop:operator_kills_coleman_fns}]
	Throughout the proof, we write~$\delta=\deg(\div(\omega)^+)$ for short. We proceed by induction, beginning with the base case $m=1$. Since the weight zero Coleman algebraic functions are all constant (see Example~\ref{ex:coleman_weight_0}), Lemma~\ref{lem:operator_gives_rational_fns} shows that $\frac\d\omega(f)$ is a rational function in $\HH^0(X,\O_X(\div(\omega)^+))$ for all $f\in\W_1\ACollog(Y)$. Thus the values of $\frac\d\omega(f)$ for~$f\in\W_1\ACollog(Y)$ span a subspace of~$\HH^0(X,\O_X(\div(\omega)^+))$ of dimension~$\leq c_1:=\dim_{K_v}\gr^\W_1\!\ACollog(Y)_U$. Applying Lemma~\ref{lem:annihilate_rational_functions}, we obtain a PD-nice algebraic differential operator~$\DO_{(1)}$ such that $\DO_{(1)}(\frac\d\omega(f))=0$ for all~$f\in\W_1\ACollog(Y)$. The order of~$\DO_{(1)}$ is $\leq \delta+1$, and its divisor is $\geq-(c_1+(c_1+1)(\delta+1))\cdot\div(\omega)^+$. Thus the differential operator $\DO_1:=\DO_{(1)}\circ\frac\d\omega$ satisfies the conditions of Proposition~\ref{prop:operator_kills_coleman_fns}: its order is at most $\delta+2$ and its divisor satisfies
	\[
	\div(\DO_1) \geq -(c_1+1)(\delta+2)\cdot\div(\omega)^+ \,.
	\]
	
	Now we proceed inductively, supposing that the differential operator~$\DO_m$ has already been constructed. Lemma~\ref{lem:operator_gives_rational_fns} implies that $\DO_m(f)$ is a rational function in $\HH^0(X,\O_X(-\div(\DO_m)))$ for all $f\in\W_{m+1}\ACollog(Y)_U$. Since~$c_{m+1}$ is the dimension of $\gr^\W_{m+1}\ACollog(Y)$ by Corollary~\ref{cor:coleman_hilbert_series}, it follows that the values of~$\DO_m(f)$ for $f\in\W_{m+1}\ACollog(Y)$ span a subspace of $\HH^0(X,\O_X(-\div(\DO_m)))$ of dimension $\leq c_{m+1}:=\dim_{K_v}\gr^\W_{m+1}\!\ACollog(Y)_U$. Applying Lemma~\ref{lem:annihilate_rational_functions} we obtain a PD-nice algebraic differential operator $\DO_{(m+1)}$ such that $\DO_{(m+1)}(\DO_m(f))=0$ for all $f\in\W_{m+1}\ACollog(Y)_U$. The order of $\DO_{(m+1)}$ is at most $-\deg(\div(\DO_m))+1$ and its divisor satisfies
	\[
	\div(\DO_{(m+1)}) \geq c_{m+1}\div(\DO_m) - (c_{m+1}+1)\cdot(-\deg(\div(\DO_m))+1)\cdot\div(\omega)^+ \,.
	\]
	Thus the differential operator $\DO_{m+1}:=\DO_{(m+1)}\circ\DO_m$ satisfies the conditions of Proposition~\ref{prop:operator_kills_coleman_fns}. The bounds on the order and divisor of~$\DO_{m+1}$ follow from the corresponding bounds for~$\DO_m$ via the calculations
	{\small
	\begin{align*}
		\order(\DO_{m+1}) & \leq \order(\DO_m)-\deg(\div(\DO_m))+1 \\
		 &\leq (\delta+2)^m\prod_{i=1}^{m-1}(c_i+1)+\delta(\delta+2)^m\prod_{i=1}^m(c_i+1)+1 \\
		 &\leq (\delta+2)^{m+1}\prod_{i=1}^m(c_i+1) \,,\\
		\div(\DO_{m+1}) &\geq \div(\DO_m) + c_{m+1}\div(\DO_m) - (c_{m+1}+1)(-\deg(\div(\DO_m))+1)\div(\omega)^+ \\
		 &\geq -(c_{m+1}+1)\left((\delta+2)^m\prod_{i=1}^m(c_i+1)+\delta(\delta+2)^m\prod_{i=1}^m(c_i+1)+1\right)\div(\omega)^+ \\
		 &\geq (\delta+2)^{m+1}\prod_{i=1}^{m+1}(c_i+1)\cdot\div(\omega)^+ \,.
	\end{align*}
	}%
\end{proof}

\subsection{Choosing the differential form~$\omega$}\label{ss:better_omega}

In the proof of Theorem~\ref{thm:weight_bound}, we showed that~$\Y(\O_v)$ could be covered by $\ellp^{(\theta_v-1)f_v}$ closed discs, such that~$f$ had at most
\begin{equation}\label{eq:discwise_bound_again}\tag{$\ast$}
\left(1+\frac{e_v}{(\theta_v-\frac{e_v}{\ellp-1})\log(\ellp)}\right)\cdot(\deg(\div(\omega)^+)+2)^n\cdot\prod_{i=1}^{n-1}(c_i+1)
\end{equation}
zeroes on each disc. We deduced Theorem~\ref{thm:weight_bound} from this using the trivial bound $\deg(\div(\omega)^+)\leq2\deg(\div(\omega))=4g+2r-4$. However, in many cases it is possible to choose~$\omega$ so that $\deg(\div(\omega)^+)$ is even smaller than this, and accordingly obtain a better bound on the number of zeroes of~$f$ than claimed in Theorem~\ref{thm:weight_bound}.

For example, if~$X$ has genus~$1$, then the invariant differential~$\omega$ on~$\X/\O_v$ has divisor~$\Dvsr$ as a section of $\Omega^1_\X(\Dvsr)$, and is non-vanishing at~$b_0$ in the special fibre. Since this divisor is entirely supported within~$D$ on the generic fibre, we see that for this particular choice of~$\omega$, we have $\deg(\div(\omega)^+)=\deg(\div(\omega))=r$. So using this particular~$\omega$ in the proof of Theorem~\ref{thm:weight_bound} yields the better bound
\[
\kappa_\ellp\cdot\#\Y_0(k_v)\cdot (r+2)^n\cdot\prod_{i=1}^{n-1}(c_i+1)
\]
on the number of $\O_v$-integral zeroes of~$f$ in the genus~$1$ case.
\smallskip

In the other examples we will examine in this section, the differentials~$\omega$ we will choose will not necessarily be defined over the ground field~$K_v$, instead over a finite extension. This does not affect the conclusion of the method, as per the following more precise version of Theorem~\ref{thm:weight_bound}.

\begin{proposition}\label{prop:how_to_improve_bounds}
	Let~$\Obar_v$ denote the ring of integers of an algebraic closure~$\Kbar_v$ of~$K_v$. For~$N\in\N_0$, we say that an integral log-differential~$\omega\in\HH^0(\X_{\Obar_v},\Omega^1_\X(\Dvsr))$ is \emph{$N$-small} just when it is supported at $\leq N$ points of~$Y(K_v)$.
	
	Suppose that for every~$b_0\in\Y(k_v)$ there is an $N$-small integral log-differential~$\omega$ which does not vanish at~$b_0$. Then for every quotient~$U$ of~$U^\dR$, every~$m\geq1$ and every non-zero $f\in\W_m\ACollog(Y)_U$, the number of zeroes of~$f$ on~$\Y(\O_v)$ is at most
	\[
	\kappa_v\cdot\#\Y_0(k_v)\cdot(2g+r+N)^m\cdot\prod_{i=1}^{m-1}(c_i+1) \,,
	\]
	where~$c_i:=\dim_{K_v}\gr^\W_i\!\ACollog(Y)_U$ as usual.
	\begin{proof}
		This is proved in exactly the same way as Theorem~\ref{thm:weight_bound}, with the only subtlety being that the differential~$\omega$ may be defined over a finite extension of~$K_v$. For any choice of base point~$b$, let~$\Disc_b\subseteq]b_0[$ be the closed subdisc of radius~$\ellp^{-\theta_v/e_v}$ containing~$b$, as in the proof of Theorem~\ref{thm:weight_bound}. Let~$\omega\in\HH^0(\X_{\Obar_v},\Omega^1_\X(\Dvsr))$ be an $N$-small log-differential which does not vanish at~$b_0$, and let~$\O_v'\supseteq\O_v$ be the ring of integers of a finite extension~$K_v'/K_v$ over which~$\omega$ is defined.
		
		The Coleman algebraic function $f\in\W_m\ACollog(Y)_U$ determines via base-change a Coleman algebraic function $f_{K_v'}\in\W_m\ACollog(Y_{K_v'})_{U_{K_v'}}$, whose restriction to the disc~$\Disc_{b,K_v'}$ is just the base change of~$f|_{\Disc_b}$ (indeed, this is true on any disc). Applying Proposition~\ref{prop:operator_kills_coleman_fns} over the finite extension $K_v'$ shows that there is a PD-nice algebraic differential operator $\DO_m\in K_v'(X)[\d/\omega]$ of order at most $(\deg(\div(\omega)^+)+2)^m\cdot\prod_{i=1}^{m-1}(c_i+1)$ such that~$\DO_m(f)=0$. Applying Proposition~\ref{prop:nice} with $\lambda=\theta_v/e_v$ shows that $f_{K_v'}$ has at most~\eqref{eq:discwise_bound_again} zeroes on~$\Disc_{b,K_v'}$, and hence~$f$ also has at most~\eqref{eq:discwise_bound_again} zeroes on~$\Disc_b$. Summing this bound over $\ellp^{(\theta_v-1)f_v}$ different discs of the form~$\Disc_b$ for various~$b\in\Y(\O_v)$ then completes the proof.
	\end{proof}
\end{proposition}

Theorem~\ref{thm:weight_bound} is Proposition~\ref{prop:how_to_improve_bounds} combined with the observation that any~$\omega$ is automatically $(2g+r-2)$-small. We have just seen that when~$g=1$, there is a $0$-small~$\omega$ which doesn't vanish at any point on the special fibre of~$\Y$. We spend the remainder of this section discussing several other examples.
\smallskip

In the case that~$X$ has genus~$0$, the effective divisor $-z_1-z_2+\Dvsr$ on~$\X_{\Obar_v}$ is relatively log-canonical whenever~$z_1,z_2\in\Dvsr(\Obar_v)$ are two different points. If~$\omega\in\HH^0(\X_{\Obar_v},\Omega^1_\X(\Dvsr))$ is a log-differential associated to this divisor, then~$\omega$ does not vanish at any point of the special fibre of~$\Y$, and is $0$-small by construction. Thus Proposition~\ref{prop:how_to_improve_bounds} applies in the genus~$0$ case with~$N=0$. Note that the differential~$\omega$ here is only defined over~$\Obar_v$ in general.
\smallskip

In the case that~$X$ is hyperelliptic of genus at least~$2$, then the effective divisor~$(2g-2)z+\Dvsr$ is relatively log-canonical whenever~$z\in\X(\Obar_v)$ lies in the ramification locus of~$\X\to\P^1_{\O_v}$. By construction, a log-differential~$\omega$ associated to this divisor is $1$-small, and by a suitable choice of~$z$ can be chosen not to vanish at any particular point~$b_0\in\Y_0(k_v)$. Thus Proposition~\ref{prop:how_to_improve_bounds} applies for hyperelliptic~$X$ with~$N=1$.
\smallskip

Our remaining two examples are of a much more general nature, and accordingly weaker than the preceding specific examples. We focus on the case that~$g\geq2$ and~$D=\emptyset$, since an $N$-small $\omega\in\HH^0(\X_{\Obar_v},\Omega^1_\X)$ is also $N$-small when viewed as a section of~$\Omega^1_\X(\Dvsr)$. The best result we can get in complete generality only gives a small saving on the naive bound.

\begin{lemma}\label{lem:slightly_better_bound}
	Suppose that~$g\geq2$. Then there is an integral differential $\omega\in\HH^0(\X_{\Obar_v},\Omega^1_\X)$ which does not vanish at~$b_0$ in the special fibre and which is $(2g-3)$-small.
	\begin{proof}
		We assume that~$g\geq3$, since we have already dealt with the case that~$X$ is hyperelliptic. We will show that there is a point $x_0\neq b_0\in\X_0(\kbar_v)$ such that $h^0(b_0+2x_0)=1$. This implies the lemma, since a Riemann--Roch computation then shows that there is a regular differential~$\omega_0$ on~$\X_0$ which vanishes to order $\geq2$ at~$x_0$ but doesn't vanish at~$b_0$. If we choose an integral point~$x\in\X(\Obar_v)$ reducing to~$x_0$, then we have
		\[
		g-2 \leq h^0(X_{\Kbar_v},\Omega^1_X(-2x)) \leq h^0(\X_{0,\kbar_v},\Omega^1_{\X_0}(-2x_0))=g-2
		\]
		by Riemann--Roch and semicontinuity \cite[Theorem~III.12.8]{hartshorne:algebraic_geometry}. Thus the inequalities above are equalities, so~$\omega_0$ lifts to an integral $1$-form $\omega\in\HH^0(\X_{\Obar_v},\Omega^1_\X(-2x))$ by Grauert's Theorem \cite[Corollary~III.12.9]{hartshorne:algebraic_geometry}. In other words, $\omega$ vanishes to order~$\geq2$ at~$x$, so is certainly $(2g-3)$-small.
		\smallskip
		
		To find the point~$x_0$, if~$\X_0$ is not hyperelliptic, then we consider the canonical embedding~$\X_0\hookrightarrow\P^{g-1}_{k_v}$. Under this embedding, we have that~$h^0(b_0+2x_0)\geq2$ if and only if the tangent line to~$\X_0$ at~$x_0$ passes through~$b_0$. But for a general point~$x_0\in\X_0(\kbar_v)$, the tangent line through~$x_0$ does not pass through~$b_0$ by Samuel's Theorem \cite[Theorem~IV.3.9]{hartshorne:algebraic_geometry}. So we are done in this case.
		
		If~$\X_0$ is hyperelliptic, then we may choose a point~$x_0\in\X_0(\kbar_v)$ for which the points $b_0$, $x_0$ and $\iota(x_0)$ are distinct, where~$\iota$ is the hyperelliptic involution. The canonical divisor $(g-1)(x_0+\iota(x_0))$ then represents a class in $\HH^0(\X_{0,\kbar_v},\Omega^1_{\X_0}(-2x_0))$ which does not lie in~$\HH^0(\X_{0,\kbar_v},\Omega^1_{\X_0}(-b_0-2x_0))$. This implies by Riemann--Roch that $h^0(b_0+2x_0)=h^0(2x_0)=1$, as desired.
	\end{proof}
\end{lemma}

When the residue characteristic~$\ellp$ is sufficiently large, this bound can be considerably improved.

\begin{lemma}\label{lem:much_better_bound}
	Suppose that~$g\geq2$ and~$\ellp>2g-2$. Then there is an integral differential $\omega\in\HH^0(\X_{\Obar_v},\Omega^1_\X)$ which does not vanish at~$b_0$ in the special fibre and which is $(g+1)$-small.
	\begin{proof}
		As in the proof of Lemma~\ref{lem:slightly_better_bound}, it suffices to show that there is a point~$x_0\neq b_0\in\X_0(\kbar_v)$ such that~$h^0(b_0+(g-1)x_0)=1$. In fact, we will show this for a general point~$x_0$. Note that the assumption that~$\ellp>2g-2$ ensures that~$\X_0$ is classical, i.e.\ $h^0(gx_0)=1$ for a general point~$x_0\in\X_0(\kbar_v)$ \cite[Theorem~11(ii)]{laksov:weierstrass_points}.
		
		
		In the case that~$\X_0$ is hyperelliptic, we are done as in the proof of Lemma~\ref{lem:slightly_better_bound}, by taking~$x_0$ to be any point such that $h^0((g-1)x_0)=1$ and $x_0$ is different from~$b_0$ and its image under the hyperelliptic involution.
		
		We thus assume from now on that~$\X_0$ is not hyperelliptic, and embedded canonically in~$\P^{g-1}_{k_v}$. For a point~$x_0$ of~$\X_0$, we have that~$h^0((g-1)x_0)=1$ just when there is a unique hyperplane in~$\P^{g-1}_{k_v(x_0)}$ meeting~$\X_{0,k_v(x_0)}$ to multiplicity~$\geq g-1$ at~$x_0$, and we have~$h^0(gx_0)=1$ just when this hyperplane meets~$\X_0$ with multiplicity exactly~$g-1$. The set of points~$x_0$ where~$h^0(gx_0)=1$ is open by semicontinuity \cite[Theorem~III.12.8]{hartshorne:algebraic_geometry}, and is non-empty since~$\X_0$ is classical. It follows that this set contains the generic point~$\eta_0$ of~$\X_0$. In other words, if~$F=k_v(\X_0)$ is the function field of~$\X_0$, then there is a unique hyperplane~$H\subseteq\P^{g-1}_F$ which meets~$\X_{0,F}$ with multiplicity exactly~$g-1$ at~$\eta_0\in\X_0(F)$.
		
		The hyperplane~$H$ can be described explicitly. Let the projective coordinates of the point~$\eta_0\in\X_0(F)\subseteq\P^{g-1}(F)$ be $(f_1:\dots:f_g)$ for some~$f_1,\dots,f_g\in F$. If we choose a non-constant element~$t\in F$, then the module~$\Omega^1_{F/\F_\ellp}$ of K\"ahler differentials is one-dimensional over~$F$, spanned by~$\d t$, so we have the differential operator~$\frac\d{\d t}$ acting on~$F$. With respect to the standard coordinates~$T_1,\dots,T_g$ on~$\P^{g-1}_F$, the hyperplane~$H$ is then given by the equation
		\[
		\det\left(\begin{matrix}
			f_1 & \frac\d{\d t} f_1 & \dots & \frac{\d^{g-2}}{\d t^{g-2}}f_1 & T_1 \\
			f_2 & \frac\d{\d t} f_2 & \dots & \frac{\d^{g-2}}{\d t^{g-2}}f_2 & T_2 \\
			\vdots & \vdots & \ddots & \vdots & \vdots \\
			f_g & \frac\d{\d t} f_g & \dots & \frac{\d^{g-2}}{\d t^{g-2}}f_g & T_g
		\end{matrix}\right)=0 \,.
		\]
		The fact that~$H$ meets~$\X_{0,F}$ to order exactly~$g-1$ at~$\eta_0$ implies that the Wronskian of $f_1,\dots,f_g$ is non-vanishing; in other words, the column vectors $\mathbf f^{(i)} := \left(\frac{\d^i}{\d t^i}f_1,\dots,\frac{\d^i}{\d t^i}f_g\right)^T\in F^g$ are~$F$-linearly independent for~$0\leq i\leq g-1$.
		
		Now we claim that the point~$b_0$ does not lie on the hyperplane~$H$. For this, write the projective coordinates of~$b_0$ as~$(b_1:\dots:b_g)$, and let~$\mathbf b:=(b_1,\dots,b_g)^T\in F^g$ be the corresponding column vector. Assuming for contradiction that~$b_0$ lies in~$H$, we find from the explicit equation above that~$\mathbf b$ lies in the $F$-linear span of $\mathbf f^{(0)},\dots,\mathbf f^{(g-2)}$. Let~$n\leq g-2$ denote the least non-negative integer such that~$\mathbf b$ lies in the span of $\mathbf f^{(0)},\dots,\mathbf f^{(n)}$. This implies that all~$(n+1)\times(n+1)$ minors of the matrix $\left(\mathbf f^{(0)}|\mathbf f^{(1)}|\dots|\mathbf f^{(n)}|\mathbf b\right)$ vanish. Applying the operator~$\frac\d{\d t}$ to this matrix shows that all the~$(n+1)\times(n+1)$ minors of the matrix $\left(\mathbf f^{(0)}|\mathbf f^{(1)}|\dots|\mathbf f^{(n-1)}|\mathbf f^{(n+1)}|\mathbf b\right)$ also vanish, so that~$\mathbf b$ is also in the span of $\mathbf f^{(0)},\dots,\mathbf f^{(n-1)},\mathbf f^{(n+1)}$. But this implies that~$\mathbf b$ is in fact in the span of $\mathbf f^{(0)},\dots,\mathbf f^{(n-1)}$. If~$n\geq1$, this contradicts the minimality of~$n$; if $n=0$ then $\mathbf b$ is in the span of the empty set, which contradicts~$\mathbf b\neq0$.
		
		We have thus shown that~$b_0$ does not lie in~$H$, which implies that we have $h^0(b_0+(g-1)\eta_0)=1$. Since the locus of points~$x_0$ such that~$h^0(b_0+(g-1)x_0)=1$ is open by semicontinuity, there must also be an $\kbar_v$-point $x_0\in\X_0(\kbar_v)$ different from~$b_0$ with this property. This is what we wanted to show.
	\end{proof}
\end{lemma}

We have thus shown that Proposition~\ref{prop:how_to_improve_bounds} holds for~$N=2g-3$ if $g\geq2$, and for~$N=g+1$ if moreover~$\ellp>2g-2$.

\begin{remark}
	Avi Kulkarni has suggested to me that one should expect Proposition~\ref{prop:how_to_improve_bounds} to hold for~$N=g-1$ when~$g\geq2$, at least for a general curve. Indeed, on~$X$ there are $2^{g-1}(2^g-1)$ odd theta characteristics, each of which has at least one associated divisor of degree~$g-1$. Taking a differential associated to twice each of these divisors provides a large number of $(g-1)$-small differentials on~$X$. It seems highly likely that one should be able to choose one of these differentials so as not to vanish at any chosen point~$b_0$ in the special fibre of~$\X$, but I have been unable to prove this in any generality.
\end{remark}

%% file: effective.tex
\section{Effective Chabauty--Kim}\label{s:effective}

With all the preliminary machinery set up, we are now ready to describe our effective Chabauty--Kim method and give the proofs of Theorems~\ref{thm:main}, \ref{thm:main_bound} and~\ref{thm:main_det}.

We start from the following setup. Let~$Y/\Q$ be a smooth hyperbolic curve, written as $X\setminus D$ for~$X/\Q$ a smooth projective curve of genus~$g$ and~$D\subseteq X$ a reduced divisor of degree~$r$. Let~$S$ be a finite set of primes, of size~$s=\#S$. Let~$\X/\Z_S$ be a regular model of~$X$ \cite[Definition~10.1.1]{liu:arithmetic_curves} over the ring of $S$-integers, not necessarily minimal. Let~$\Dvsr\subseteq\X$ denote the closure of~$D$ in~$\X$, and set~$\Y:=\X\setminus\Dvsr$. We pick an $S$-integral basepoint\footnote{The existence of such an $S$-integral point is of course an assumption on~$Y$ and its model~$\Y$. However, it is a very mild assumption: there is not much to be said about $S$-integral points when~$\Y(\Z_S)=\emptyset$.}~$b\in\Y(\Z_S)$.

Now choose an auxiliary prime~$\ellp\notin S$ of good reduction for~$(\X,\Dvsr)$, i.e.\ such that~$\X$ is smooth and~$\Dvsr$ is \'etale over~$\Z_\ellp$. We let~$U^\et$ denote the $\Q_\ellp$-pro-unipotent \'etale fundamental group of~$Y_\Qbar$ based at~$b$, which is a $\Q_\ellp$-pro-unipotent group endowed with a continuous action of the absolute Galois group~$G_\Q$. The group~$U^\et$ carries a $G_\Q$-invariant \emph{weight filtration} \cite[Definition~1.5]{asada-matsumoto-oda:local_monodromy}
\[
1\subseteq\dots\subseteq\W_{-3}U^\et\subseteq\W_{-2}U^\et\subseteq\W_{-1}U^\et=U^\et \,,
\]
where $\W_{-1}U^\et=U^\et$, $\W_{-2}U^\et$ is the kernel of the map from~$U^\et$ to the abelianisation of the $\Q_\ellp$-pro-unipotent \'etale fundamental group of~$X_{\Qbar}$, and for $k\geq3$, $\W_{-k}U^\et$ is the subgroup-scheme generated by the commutator subgroups~$[\W_{-i}U^\et,\W_{-j}U^\et]$ for~$i+j=k$. When~$Y=X$ is projective, this is just the descending central series up to reindexing. We note the following regarding the graded pieces of~$U^\et$.

\begin{lemma}\label{lem:global_semisimplicity}
	For any~$n>0$, the $G_\Q$-representation $\gr^\W_{-n}U^\et$ is semisimple.
	\begin{proof}
		For~$n=1$, $\gr^\W_{-1}U^\et$ is the $\Q_\ellp$-linear Tate module of the Jacobian of~$X$, which is semisimple by \cite[Theorem~(a)]{faltings:endlichkeitssaetze}.
		
		For~$n=2$, there are $G_\Q$-equivariant maps
		\[
		\bigwedgesquare\gr^\W_{-1}U^\et\to\gr^\W_{-2}U^\et \hspace{0.4cm}\text{and}\hspace{0.4cm} \bigoplus_{z\in D(\Qbar)}\Q_\ellp(1)\to\gr^\W_{-2}U^\et \,,
		\]
		the first of which is the commutator map, and the second of which arises from the inclusions of the cusps of~$Y$. The images of these two maps span $\gr^\W_{-2}U^\et$, so $\gr^\W_{-2}U^\et$ is a quotient of $\bigwedgesquare\gr^\W_{-1}U^\et\oplus\bigoplus_{z\in D(\Qbar)}\Q_\ellp(1)$. Since the class of semisimple representations of~$G_\Q$ is closed under tensor products (see e.g.\ \cite{serre:semisimplicity}), direct sums and quotients and contains all Artin--Tate representations, it follows that $\gr^\W_{-2}U^\et$ is also semisimple.
		
		For~$n\geq3$, we proceed inductively, noting that the commutator map gives a surjection
		\[
		(\gr^\W_{-1}U^\et\otimes\gr^\W_{1-n}U^\et)\oplus(\gr^\W_{-2}U^\et\otimes\gr^\W_{2-n}U^\et)\twoheadrightarrow\gr^\W_{-n}U^\et \,,
		\]
		exhibiting~$\gr^\W_{-n}U^\et$ as a quotient of a semisimple representation, so we are done.
	\end{proof}
\end{lemma}

Now we fix a $G_\Q$-equivariant quotient~$U$ of~$U^\et$, not necessarily finite-dimensional, and endow it with the induced weight filtration~$\W_\bullet$ from~$U^\et$. The quotient~$U$ satisfies the three conditions from \S\ref{ss:global_selmer}: it is unramified at all primes $\pell\notin S\cup\{\ellp\}$ of good reduction for~$(\X,\Dvsr)$; its Lie algebra $\Lie(U)$ is pro-crystalline by \cite[Theorem~1.8]{olsson:towards}; and $V_n$ has no $G_\pell$-invariants for any~$\pell$ since it is a direct summand of~$\gr^\W_{-n}U^\et$ by Lemma~\ref{lem:global_semisimplicity} and the latter has no $G_\pell$-invariants by \cite[Theorem~1.3(1)]{me-daniel:weight-monodromy}. Hence the local cohomology schemes $\HH^1(G_\pell,U)$ are representable for all primes~$\pell$, as are the local and global Selmer schemes $\HH^1_f(G_\ellp,U)$ and $\Sel_{\lSel,U}$ for any Selmer structure~$\lSel$ on~$U$.

Associated to the quotient~$U$, there is a \emph{non-abelian Kummer map}
\[
\jj_U\colon Y(\Q)\to\HH^1(G_\Q,U(\Q_\ellp))
\]
given by sending a point~$y\in Y(\Q)$ to the class of the $\Q_\ellp$-pro-unipotent \'etale torsor of paths from~$b$ to~$y$, pushed out along~$U^\et\twoheadrightarrow U$. Similarly, for every prime~$\pell$ there is a \emph{local non-abelian Kummer map}
\[
\jj_{\pell,U}\colon Y(\Q_\pell)\to\HH^1(G_\pell,U(\Q_\ellp))
\]
defined in exactly the same way. For $\pell\neq\ellp$, the image $\jj_{\pell,U}(\Y(\Z_\pell))$ of the $\Z_\pell$-integral points is finite \cite[Corollary~0.2]{minhyong-tamagawa:l-component}, and consists of just the basepoint if $\pell$ is of good reduction for~$(\X,\Dvsr)$.

For $\pell=\ellp$, the image $\jj_{\ellp,U}(\Y(\Z_\ellp))$ of the $\Z_\ellp$-integral points is contained inside $\HH^1_f(G_p,U)\subseteq\HH^1(G_p,U)$. There is also a \emph{de Rham Kummer map}
\[
\jj_{\dR,U}\colon\Y(\Z_\ellp)\to\Fil^0\backslash\D_\dR(U)(\Q_\ellp) \,,
\]
given as the composite of the de Rham Kummer map $\jj_\dR\colon\Y(\Z_\ellp)\to\Fil^0\backslash U^\dR(\Q_\ellp)$ from \S\ref{sss:de_rham_kummer} with the map $\Fil^0\backslash U^\dR\cong\Fil^0\backslash\D_\dR(U^\et)\twoheadrightarrow\Fil^0\backslash\D_\dR(U)$. The group~$U^\dR$ here denotes the pro-unipotent de Rham fundamental group of~$Y_{\Q_\ellp}$, the functor~$\D_\dR(-)$ is the de Rham Dieudonn\'e functor at the prime~$\ellp$, and the isomorphism $U^\dR\cong\D_\dR(U^\et)$ is the \'etale--de Rham comparison isomorphism described in \cite[\S7]{hadian:motivic_pi_1}\cite[\S5]{faltings:around}, which preserves the Hodge filtration by its construction.

These Kummer maps fit together in a commuting diagram 
\cite[p.~95]{minhyong:selmer}\footnote{In \cite{minhyong:selmer}, the comparison isomorphism used between the \'etale and de Rham fundamental groups is the one of Olsson \cite[Theorem~1.8]{olsson:towards}. However, the proof of commutativity of the diagram only requires one fact about the comparison isomorphism: that it preserves torsor structures. So the diagram also commutes when the comparison isomorphism used is that from \cite[\S7]{hadian:motivic_pi_1}, as here.}
\begin{equation}\label{diag:c-k}
	\begin{tikzcd}
		\Y(\Z_S) \arrow[r,hook]\arrow[d,"\jj_U"] & \Y(\Z_\ellp) \arrow[rd,"\jj_{\dR,U}"]\arrow[d,"\jj_{\ellp,U}"] & \\
		\HH^1_{f,T_0}(G_\Q,U(\Q_\ellp)) \arrow[r,"\loc_\ellp"] & \HH^1_f(G_\ellp,U)(\Q_\ellp) \arrow[r,"\log_\BK","\sim"'] & \Fil^0\backslash\D_\dR(U)(\Q_\ellp) \,,
	\end{tikzcd}
\end{equation}
where~$T_0\not\ni\ellp$ is a finite set of primes, containing~$S$ and all primes of bad reduction for~$(\X,\Dvsr)$. Collecting what we have proved in the preceding sections, we see that the sets appearing on the bottom row are the $\Q_\ellp$-points of affine $\Q_\ellp$-schemes $\HH^1_{f,T_0}(G_\Q,U)$, $\HH^1_f(G_\ellp,U)$ and $\Fil^0\backslash\D_\dR(U)$, and that these all come with natural weight filtrations on their affine rings, induced from the weight filtration on~$U$ (\S\ref{ss:global_selmer}, \S\ref{sss:H^1_f} and \S\ref{sss:bk_log}). The localisation map~$\loc_\ellp$ is a morphism of filtered affine $\Q_\ellp$-schemes, since it represents a natural transformation of functors. The Bloch--Kato logarithm $\log_\BK$ is an isomorphism of filtered affine $\Q_\ellp$-schemes by Proposition~\ref{prop:bk_log_is_filtered_iso}. The map~$\jj_{\dR,U}$ has Zariski-dense image by \cite[Theorem~1]{minhyong:selmer}, and for every~$\alpha\in\W_m\O(\Fil^0\backslash\D_\dR(U))$, the composite $\alpha\circ\jj_{\dR,U}$ is a Coleman algebraic function of weight at most~$m$ by Theorem~\ref{thm:log-coleman_and_de_rham_kummer}. Moreover, $\alpha\circ\jj_{\dR,U}$ is associated to the quotient~$\D_\dR(U)$ of~$U^\dR$ by Proposition~\ref{prop:log-coleman_for_quotients}; we will say that it is \emph{associated to~$U$} for short.

Putting all of this together, we obtain our first effective Chabauty--Kim result.

\begin{theorem}\label{thm:main_naive}
	Let $(c_{i,T_0}^\glob)_{i\geq0}$ and $(c_i^\loc)_{i\geq0}$ be the coefficients of the power series
	\begin{align*}
		\HS_{\glob,T_0}(t) &:= \prod_{n\geq1}^\infty(1-t^n)^{-\dim\HH^1_{f,T}(G_\Q,V_n)} \\
		\HS_\loc(t) &:= \prod_{n\geq1}^\infty(1-t^n)^{-\dim\HH^1_f(G_p,V_n)} \,.
	\end{align*}
	Suppose that~$m$ is a positive integer such that the inequality
	\begin{equation}\label{eq:c-k_ineq_series_naive}
		\sum_{i=0}^mc_{i,T_0}^\glob < \sum_{i=0}^mc_i^\loc
	\end{equation}
	holds. Then:
	\begin{enumerate}[label=\Alph*), ref=(\Alph*), font=\normalfont]
		\item\label{thmpart:main_naive} The set $\Y(\Z_S)$ is contained in the vanishing locus of a Coleman algebraic function of weight at most~$m$ which is associated to~$U$.
		\item\label{thmpart:main_naive_bound} We have
		\[
		\#\Y(\Z_S)\leq\kappa_p\cdot\#\Y(\F_p)\cdot(4g+2r-2)^m\cdot\prod_{i=1}^{m-1}(c_i^\loc+1) \,,
		\]
		where the constant $\kappa_p$ is as in Theorem~\ref{thm:main_bound}.
		\item\label{thmpart:main_naive_det} If $C_{m,T_0}^\glob:=\sum_{i=0}^mc_{i,T_0}^\glob$ and $C_m^\loc:=\sum_{i=0}^mc_i^\loc$, and if $f_1,\dots,f_{C_m^\loc}$ is a basis of the space of Coleman algebraic functions of weight at most~$m$ associated to~$U$, then for every $C_{m,T_0}^\glob+1$-tuple of points $x_0,x_1,\dots,x_{C_{m,T_0}^\glob}\in\Y(\Z_S)$, all $(C_{m,T_0}^\glob+1)\times(C_{m,T_0}^\glob+1)$ minors of the matrix~$M$ with entries $M_{ij}=f_i(x_j)$ vanish.
	\end{enumerate}
	\begin{proof}
		According to Corollary~\ref{cor:hilbert_series_of_local_selmer} and Remark~\ref{rmk:global_hilbert_series_bound_for_naive}, we see that the Hilbert series of $\HH^1_{f,T_0}(G_\Q,U)$ and $\HH^1_f(G_\ellp,U)$ satisfy
		\[
		\HS_{\HH^1_{f,T_0}(G_\Q,U)}(t)\preceq\HS_{\glob,T_0}(t) \hspace{0.4cm}\text{and}\hspace{0.4cm} \HS_{\HH^1_f(G_\ellp,U)}(t)=\HS_\loc(t) \,.
		\]
		Hence, when~$m$ is such that inequality~\eqref{eq:c-k_ineq_series_naive} holds, we have
		\[
		\dim\W_m\O(\HH^1_{f,T_0}(G_\Q,U))\leq\sum_{i=0}^mc_{i,T_0}^\glob\leq\sum_{i=0}^mc_i^\loc=\dim\W_m\O(\HH^1_f(G_\ellp,U)) \,.
		\]
		In particular, the map $\loc_\ellp^*\colon\W_m\O(\HH^1_f(G_\ellp,U))\to\W_m\O(\HH^1_{f,T_0}(G_\Q,U))$ must fail to be injective for dimension reasons. If~$\alpha$ is a non-zero element of the kernel, then by commutativity of the diagram \eqref{diag:c-k} and Theorem~\ref{thm:log-coleman_and_de_rham_kummer}, the composite $\alpha\circ\jj_{\ellp,U}\colon\Y(\Z_\ellp)\to\Q_\ellp$ vanishes on $\Y(\Z_S)$, and is a Coleman algebraic function of weight~$m$ which is associated to~$U$. This proves~\ref{thmpart:main_naive}.
		
		Part~\ref{thmpart:main_naive_bound} follows from Theorem~\ref{thm:weight_bound} and the observation that the Hilbert series of $\HH^1_f(G_\ellp,U)$, $\Fil^0\backslash\D_\dR(U)$ and $\ACollog(Y)_{\D_\dR(U)}$ are the same by Propositions~\ref{prop:bk_log_is_filtered_iso} and~\ref{prop:log-coleman_for_quotients}. Part~\ref{thmpart:main_naive_det} follows from the simple observation that the kernel of the map $\loc_\ellp^*\colon\W_m\O(\HH^1_f(G_\ellp,U))\to\W_m\O(\HH^1_{f,T_0}(G_\Q,U))$ has dimension $\geq C_m^\loc-C_{m,T_0}^\glob$, and hence the image of $\Y(\Z_S)$ under the map $\Y(\Z_\ellp)\to\Q_\ellp^{C_m^\loc}$ given by $(f_1,\dots,f_{C_m^\loc})$ must be contained in a subspace of dimension~$\leq C_{m,T_0}^\glob$. This is equivalent to the vanishing of the minors of the claimed matrix.
	\end{proof}
\end{theorem}

Although Theorem~\ref{thm:main_naive} does give effective constraints on the set~$\Y(\Z_S)$, it doesn't incorporate any information at places inside the set~$T_0$, and as a consequence the inequality~\eqref{eq:c-k_ineq_series_naive} required as input in the theorem depends strongly on~$S$ and the bad primes of~$(\X,\Dvsr)$. This limitation can be overcome by using a more refined type of Selmer scheme, as defined in \cite[p.~371]{jen-etal:non-abelian_tate-shafarevich} in the case that~$Y=X$ is projective, and in \cite[Definition~1.2.2]{me-netan:ramification} in general.

\subsection{Properties of local Kummer maps}

The theory of refined Selmer schemes of the kind studied in \cite{jen-etal:non-abelian_tate-shafarevich} and \cite{me-netan:ramification} revolves around a careful analysis of the images of the local non-abelian Kummer maps $\jj_{\pell,U}$ for all primes~$\pell$, not just those of good reduction. The image of~$\jj_{\pell,U}$ turns out to be closely related to the mod-$\pell$ reduction type of~$(\X,\Dvsr)$ \cite{me-netan:ramification}, as we now recall.

\subsubsection{Local decomposition: $\pell\notin S$}

To begin with, let $\pell\notin S\cup\{\ellp\}$ be prime, and let $\Cpts_\pell$ be the set of irreducible components of the special fibre of~$\X_{\Z_\pell}$. If~$\Sigma_\pell\in\Cpts_\pell$ is such a component, then we write $\Y(\Z_\pell)_{\Sigma_\pell}\subseteq\Y(\Z_\pell)$ for the set of $\Z_\pell$-points of~$\Y$ reducing onto the component~$\Sigma_\pell$ of the special fibre of~$\X_{\Z_\pell}$: we say that $\Y(\Z_\pell)_{\Sigma_\pell}$ is the set of $\pell$-adic points whose \emph{reduction type} is~$\Sigma_\pell$. These sets give a partition
\begin{equation}\label{eq:local_decomposition_integral}
\Y(\Z_\pell) = \coprod_{\Sigma_\pell\in\Cpts_\pell}\Y(\Z_\pell)_{\Sigma_\pell}
\end{equation}
of the $\pell$-adic points.

For a reduction type~$\Sigma_\pell\in\Cpts_\pell$, we write $\lSel_{\Sigma_\pell}$ for the image of $\Y(\Z_\pell)_{\Sigma_\pell}$ under the local non-abelian Kummer map $\jj_{\pell,U}$. The theory of \cite{me-netan:ramification} then shows the following.

\begin{lemma}\label{lem:local_images_outside_s}
	For all~$\Sigma_\pell\in\Cpts_\pell$, $\lSel_{\Sigma_{\pell}}$ either is empty or consists of a single point.
	\begin{proof}
		It suffices to prove the result in the case $U=U^\et$. Choose a finite extension~$K_v/\Q_\pell$ over which~$Y$ acquires semistable reduction, and write~$\O_v$ for the ring of integers of~$K_v$. Let~$(\X',\Dvsr')$ be the minimal regular normal crossings\footnote{Recall that a \emph{model} $(\X',\Dvsr')$ of~$(X,D)$ over~$\O_v$ is a pair of a proper flat $\O_v$-scheme~$\X'$ whose generic fibre is~$X_{K_v}$, and the divisor~$\Dvsr'\subseteq\X'$ of~$D_{K_v}$ in~$\X'$. A model~$(\X',\Dvsr')$ is said to be \emph{regular normal crossings} just when $\X'$ is a regular scheme and $\X'_0\cup\Dvsr'$ is a normal crossings divisor on~$\X'$, where $\X'_0$ denotes the special fibre of~$\X'$. A model is \emph{semistable} just when the special fibre~$\X'_0$ is reduced. See \cite[\S9.3]{liu:arithmetic_curves} for the theory of regular normal crossings models in the case~$D=\emptyset$, and \cite[Appendix~B]{me:thesis} for the general case.} desingularisation of~$(\X_{\O_v},\Dvsr_{\O_v})$. The exceptional locus of the map $\X'\to\X_{\O_v}$ is contained in the union of the singular locus of the special fibre of~$\X_{\O_v}$ and the special fibre of~$\Dvsr_{\O_v}$.
		
		Now let~$(\X_{\min},\Dvsr_{\min})$ be the minimal regular normal crossings model of the base change $(X_{K_v},D_{K_v})$. This is a semistable model, and there is a unique morphism~$(\X',\Dvsr')\to(\X_{\min},\Dvsr_{\min})$ of models of~$(X_{K_v},D_{K_v})$. If the exceptional locus of the map $\X'\to\X_{\min}$ contains a point in the special fibre of~$\Dvsr_{\min}$, then blowing up this point yields another semistable model of~$(X_{K_v},D_{K_v})$ dominated by~$(\X',\Dvsr')$. Repeating this construction if necessary, we find that there is a semistable model~$(\X'',\Dvsr'')$ of~$(X_{K_v},D_{K_v})$ together with a map $(\X',\Dvsr')\to(\X'',\Dvsr'')$ of models, whose exceptional locus does not contain any point in the special fibre of~$\Dvsr''$.
		
		Suppose now that~$x,y\in\Y(\Z_\pell)_{\Sigma_\pell}$ both reduce onto the component~$\Sigma_\pell$ of the special fibre of~$\X$. Since~$x$ and~$y$ reduce to non-singular points of the special fibre of~$\X$ \cite[Lemma~3.1]{liu-tong:neron_models} and do not reduce onto the special fibre of~$\Dvsr$, it follows that~$x$ and~$y$ also reduce onto the same component of the special fibre of~$\X'$. Moreover, these reductions do not lie in the special fibre of~$\Dvsr'$. Now, via the map $(\X',\Dvsr')\to(\X'',\Dvsr'')$ we see that~$x$ and~$y$ reduce onto the same component of the special fibre of~$\X''$, and moreover do not reduce onto the special fibre of~$\Dvsr''$. Thus, $x$ and~$y$ are $\O_v$-integral points of~$\Y'':=\X''\setminus\Dvsr''$. Since the model~$(\X'',\Dvsr'')$ is semistable, we can apply \cite[Proposition~3.8.1]{me-netan:ramification} to deduce that $\jj_{\pell,U^\et}(x)=\jj_{\pell,U^\et}(y)\in\HH^1(G_\pell,U^\et(\Q_\ellp))$, which completes the proof of the claim.
	\end{proof}
\end{lemma}

\begin{remark}\label{rmk:local_images_outside_s}
	Lemma~\ref{lem:local_images_outside_s} gives an alternative proof of \cite[Corollary~0.2]{minhyong-tamagawa:l-component}, that the image $\jj_{\pell,U}(\Y(\Z_\pell))$ is finite. More importantly, it gives us control over the size of the image: for example, we have $\#\jj_{\pell,U}(\Y(\Z_\pell))\leq \#\Cpts_\pell$.
\end{remark}

\subsubsection{Local decomposition: $\pell\in S$}

If now~$\pell\in S$, then we have a similar decomposition of the set $Y(\Q_\pell)$, but its description is more complicated. Let~$(\X_{\min},\Dvsr_{\min})$ be the minimal regular normal crossings model of~$(X,D)$ over~$\Z_\pell$, and denote its special fibre by $(\X_{\min,0},\Dvsr_{\min,0})$. We define~$\Cpts_\pell^0$ to be the set of connected/irreducible components of~$\X_{\min,0}^\sm\setminus\Dvsr_{\min,0}$, where~$\X_{\min,0}^\sm$ denotes the smooth locus in $\X_{\min,0}$. We also define~$\Cpts_\pell^1:=|\Dvsr_{\min,0}|$ to be the set of closed points of the special fibre of~$\Dvsr_{\min}$, and set $\Cpts_\pell:=\Cpts_\pell^0\cup\Cpts_\pell^1$. If~$\Sigma_\pell\in\Cpts_\pell$, then we write $Y(\Q_\pell)_{\Sigma_\pell}\subseteq Y(\Q_\pell)$ for the set of~$\Q_\pell$-points of~$Y$ reducing onto~$\Sigma_\pell$. These sets give a partition
\begin{equation}\label{eq:local_decomposition_rational_naive}
	Y(\Q_\pell) = \coprod_{\Sigma_\pell\in\Cpts_\pell}Y(\Q_\pell)_{\Sigma_\pell} \,.
\end{equation}

We note for future reference the following.

\begin{lemma}\label{lem:non-empty_discs}
	Let $\Sigma_\pell^1\in\Cpts_\pell^1$ be a closed point of~$\Dvsr_{\min,0}$. Then $Y(\Q_\pell)_{\Sigma_{\pell}}\neq\emptyset$ if and only if $\Sigma_\pell^1$ is an $\F_\pell$-point of $\Dvsr_{\min,0}$. Moreover, when this occurs, there is a unique point $z\in D(\Q_\pell)$ reducing to~$\Sigma_\pell^1$.
	\begin{proof}
		In one direction, if $Y(\Q_\pell)_{\Sigma_\pell^1}\neq\emptyset$, then $\Sigma_\pell^1$ contains the reduction of a $\Q_\pell$-point of~$X$, so must be $\F_\pell$-rational. Conversely, since $\Dvsr_{\min}\cup\X_{\min,0}$ is a normal crossings divisor, $\Dvsr_{\min,0}$ must be contained in the smooth locus of $\X_{\min,0}$. Hence if $\Sigma_\pell^1$ is $\F_\pell$-rational, then by Hensel's Lemma it is the reduction of infinitely many points of $X(\Q_\pell)$, in particular $Y(\Q_\pell)_{\Sigma_\pell^1}\neq\emptyset$.
		
		For the final point, suppose that $\Sigma_\pell^1$ is $\F_\pell$-rational. The component of~$\Dvsr_{\min}$ containing~$\Sigma_\pell^1$ is~$\Spec(\O_v)$ for~$\O_v$ the ring of integers in a finite extension~$K_v/\Q_\pell$. On the one hand, since this component has an $\F_\pell$-point, the extension $K_v/\Q_\pell$ must be totally ramified. On the other, since~$\Dvsr_{\min}\cup\X_{\min,0}$ is a normal crossings divisor, the ring $\O_v/p$ must have length~$1$ as a module over itself, so the extension $K_v/\Q_\pell$ is unramified. Thus in fact~$K_v=\Q_\pell$ and so~$\Sigma_\pell^1$ is the reduction of a point~$z\in D(\Q_\pell)$ as desired.
	\end{proof}
\end{lemma}

Similarly to before, for a reduction type $\Sigma_\pell\in\Cpts_\pell$, we write~$\lSel_{\Sigma_\pell}\subseteq\HH^1(G_\pell,U)$ for the Zariski-closure of the image of $Y(\Q_\pell)_{\Sigma_\pell}$ under the local non-abelian Kummer map~$\jj_{\pell,U}$. The theory of \cite{me-netan:ramification} allows us to understand the geometry of these subschemes $\lSel_{\Sigma_\pell}$.

\begin{lemma}\label{lem:local_images_in_s}
	\leavevmode
	\begin{enumerate}
		\item\label{lempart:local_images_in_s_integral} If~$\Sigma_\pell^0\in\Cpts_\pell^0$, then $\lSel_{\Sigma_\pell^0}$ either is empty or a single $\Q_\pell$-point.
		\item\label{lempart:local_images_in_s_rational} If~$\Sigma_\pell^1\in\Cpts_\pell^1$, then $\lSel_{\Sigma_\pell^1}$ either is empty, a single $\Q_\pell$-point or a curve of genus~$0$ (possibly singular).
		\item\label{lempart:local_images_in_s_compatibility} If the cusp~$\Sigma_\pell^1$ is $\F_\pell$-rational and contained in the Zariski-closure of the component~$\Sigma_\pell^0$, then $\lSel_{\Sigma_\pell^0}\subseteq\lSel_{\Sigma_\pell^1}$.
	\end{enumerate}
\end{lemma}

The proof of Lemma~\ref{lem:local_images_in_s} uses the full strength of the theory developed in \cite{me-netan:ramification}. We may assume without loss of generality that~$U=U^\et$ is the whole $\Q_\ellp$-pro-unipotent \'etale fundamental group. Choose a finite Galois extension~$K_v/\Q_\pell$ over which $(X,D)$ acquires semistable reduction, and write~$\O_v$ for the ring of integers of~$K_v$. Write~$G_v$ for the absolute Galois group of~$K_v$, and $\jj_{v,U^\et}\colon Y(K_v)\to\HH^1(G_v,U^\et(\Q_\ellp))$ for the non-abelian Kummer map, taking a $K_v$-point~$y$ to the class of the torsor of $\Q_\ellp$-pro-unipotent \'etale paths from~$b$ to~$y$. The cohomology scheme $\HH^1(G_\pell,U^\et)$ is a closed subscheme of $\HH^1(G_v,U^\et)$ via the restriction map \cite[Lemma~2.3.1]{me-netan:ramification}\footnote{Strictly speaking, \cite[Lemma~2.3.1]{me-netan:ramification} only makes an assertion about the $\Q_\ellp$-points of these schemes. However, the corresponding result is true for points valued in any $\Q_\ellp$-algebra, and the same proof works. Indeed, many of the results from \cite{me-netan:ramification} we cite here are phrased in terms of $\Q_\ellp$-points but in fact hold functorially; see \cite[Remark~2.8.1]{me-netan:ramification}.}, and the non-abelian Kummer map~$\jj_{\pell,U^\et}\colon Y(\Q_\pell)\to\HH^1(G_\pell,U^\et(\Q_\ellp))$ is just the restriction of the map~$\jj_{v,U^\et}$ to~$Y(\Q_\pell)$.

Let $(\X',\Dvsr')$ be the minimal normal crossings desingularisation of the base change $(\X_{\min,\O_v},\Dvsr_{\min,\O_v})$. As in the proof of Lemma~\ref{lem:local_images_outside_s}, there is a semistable model $(\X^{(0)},\Dvsr^{(0)})$ of~$(X_{K_v},D_{K_v})$ dominated by~$(\X',\Dvsr')$ such that the exceptional locus of the map $\X'\to\X^{(0)}$ is disjoint from the special fibre of~$\Dvsr^{(0)}$. We then recursively define a sequence of models~$(\X^{(n)},\Dvsr^{(n)})$ for~$n\geq0$ by setting~$\X^{(n)}$ to be the blowup of~$\X^{(n-1)}$ at the special fibre of~$\Dvsr^{(n-1)}$ and setting~$\Dvsr^{(n)}$ to be the strict transform of~$\Dvsr^{(n-1)}$ in~$\X^{(n)}$. All of these models thus fit into a sequence
\[
(\X_{\min,\O_v},\Dvsr_{\min,\O_v}) \leftarrow (\X',\Dvsr') \rightarrow (\X^{(0)},\Dvsr^{(0)}) \leftarrow (\X^{(1)},\Dvsr^{(1)}) \leftarrow \cdots \,.
\]

Now let~$\Gamma^{(n)}$ denote the dual graph of the geometric special fibre of~$\X^{(n)}$, i.e.\ the graph whose vertices correspond to the irreducible components of the geometric special fibre, and whose edges correspond to the singular points of the geometric special fibre. For~$n>0$ there is an injection $\Gamma^{(n-1)}\hookrightarrow\Gamma^{(n)}$ given on the level of vertices by taking the strict transform of irreducible components, and we set $\Gamma:=\varinjlim\Gamma^{(n)}$. The graph~$\Gamma$ thus admits a decomposition
\[
\Gamma = \Gamma^{(0)}\cup\bigcup_{z\in D(\Kbar_v)}\Gamma_z \,,
\]
where~$\Gamma_z$ is a half-infinite chain of edges, joined to $\Gamma^{(0)}$ at the vertex corresponding to the component of the geometric special fibre of~$\X^{(0)}$ containing the reduction of~$z$.

For every~$K_v$-point $y\in Y(K_v)$, there is some~$n$ for which the reduction of~$y$ to the special fibre of~$\X^{(n)}$ does not land in the special fibre of~$\Dvsr^{(n)}$, and we define $\red(y)\in V(\Gamma)$ to be the vertex of~$\Gamma$ corresponding to the irreducible component of the geometric special fibre of~$\X^{(n)}$ onto which~$y$ reduces. This is independent of the choice of~$n$, and defines a map $\red\colon Y(K_v)\to V(\Gamma)$ called the \emph{reduction map} \cite[Definition~1.1.1]{me-netan:ramification}. According to \cite[Proposition~3.8.1]{me-netan:ramification}, the non-abelian Kummer map $\jj_v\colon Y(K_v)\to\HH^1(G_v,U^\et(\Q_\ellp))$ factors through the reduction map, as the composite
\[
Y(K_v) \xrightarrow\red V(\Gamma) \xrightarrow{\jj_\Gamma} \HH^1(G_v,U^\et(\Q_\ellp))
\]
for some map $\jj_\Gamma\colon V(\Gamma)\to\HH^1(G_v,U^\et(\Q_\ellp))$. Moreover, for any geometric point~$z\in D(\Kbar_v)$, the restriction of~$\jj_\Gamma$ to~$V(\Gamma_z)$ is a polynomial map, in the sense that for every~$\alpha\in\O(\HH^1(G_v,U^\et))$, the composite $\alpha\circ\jj_\Gamma|_{V(\Gamma_z)}$ is a polynomial function when we identify $V(\Gamma_z)\cong\N_0$ in the natural way \cite[Corollary~8.1.2 \& Theorem~6.1.1]{me-netan:ramification}. In particular, the Zariski-closure of~$\jj_\Gamma(V(\Gamma_z))$ is the scheme-theoretic image of a morphism $\A^1_{\Q_\ellp}\to\HH^1(G_v,U^\et)$, so is either a single point or a curve of genus~$0$.
\smallskip

Using this theory, it is easy to deduce Lemma~\ref{lem:local_images_in_s}. For point~\eqref{lempart:local_images_in_s_integral}, every point of $Y(\Q_\pell)_{\Sigma_\pell^0}$ (assumed non-empty without loss of generality) reduces onto the same component of the geometric special fibre of~$\X_{\min,\O_v}$ by definition. Since all of these points reduce onto smooth points outside the special fibre of~$\Dvsr_{\min,\O_v}$, they avoid the exceptional locus of the minimal normal crossings desingularisation $(\X',\Dvsr')\to(\X_{\min,\O_v},\Dvsr_{\min,\O_v})$, and so all points of~$Y(\Q_\pell)_{\Sigma_\pell^0}$ also reduce onto the same component of the geometric special fibre of~$\X'$, avoiding the special fibre of~$\Dvsr'$. Since the exceptional locus of $\X'\to\X^{(0)}$ is disjoint from the special fibre of~$\Dvsr^{(0)}$, it follows that all points of~$Y(\Q_\pell)_{\Sigma_\pell^0}$ also reduce onto the same component of the geometric special fibre of~$\X^{(0)}$, avoiding the special fibre of~$\Dvsr^{(0)}$. In other words, $\red(Y(\Q_\pell)_{\Sigma_\pell^0})$ consists of a single vertex of~$\Gamma^{(0)}\subseteq\Gamma$. Thus, $\jj_\pell(Y(\Q_\pell)_{\Sigma_\pell^0})=\jj_v(Y(\Q_\pell)_{\Sigma_\pell^0})$ also consists of a single point, as desired.

For point~\eqref{lempart:local_images_in_s_rational}, assume without loss of generality that $Y(\Q_\pell)_{\Sigma_\pell^1}\neq\emptyset$, write $z_0\in\Dvsr_{\min,0}(\F_\pell)$ as a shorthand for the point $\Sigma_\pell^1\in\Cpts_\pell^1$, and write~$z\in D(\Q_\pell)$ for the unique point of~$D$ reducing to~$z_0$, as in Lemma~\ref{lem:non-empty_discs}. The point~$z_0$ lies in the smooth locus of $\X_{\min}\to\Spec(\Z_\pell)$ and the divisor $\Dvsr_{\min}\cup\X_{\min,0}$ is normal crossings in a neighbourhood of~$z_0$, so~$z_0$ does not lie in the exceptional locus of the minimal normal crossings desingularisation $(\X',\Dvsr')\to(\X_{\min,\O_v},\Dvsr_{\min,\O_v})$. It follows that all points of~$Y(\Q_\pell)_{\Sigma_\pell^1}$ reduce to the same point on the special fibre of~$\X'$, namely the reduction of~$z\in D(\Q_\pell)$. In particular, all points of~$Y(\Q_\pell)_{\Sigma_\pell^1}$ reduce to the same point as~$z$ in the special fibre of~$\X^{(0)}$, so that $\red(Y(\Q_\pell)_{\Sigma_\pell^1})\subseteq V(\Gamma_z)$. In fact, $\red(Y(\Q_\pell)_{\Sigma_\pell^1})$ is infinite, since for instance the point $z$ is a limit point of $Y(\Q_\pell)_{\Sigma_\pell^1}$, but not of $(\X^{(n)}\setminus\Dvsr^{(n)})(\O_v)$ for any~$n$. It follows that the Zariski-closure of $\jj_\pell(Y(\Q_\pell)_{\Sigma_\pell^1})$ is the same as the Zariski-closure of $\jj_\Gamma(V(\Gamma_z))$, which is either a single point or a curve of genus~$0$, as previously discussed.

Finally, for point~\eqref{lempart:local_images_in_s_compatibility}, write again $z_0$ for the point $\Sigma_\pell^1\in\Cpts_\pell^1=\Dvsr_{\min,0}(\F_\pell)$, and write~$z\in D(\Q_\pell)$ for the unique point of~$D$ reducing to~$z_0$. Following through the arguments of the previous two points, we see that all points of $Y(\Q_\pell)_{\Sigma_\pell^0}\cup Y(\Q_\pell)_{\Sigma_\pell^1}$ reduce onto the same component of the special fibre of~$\X^{(0)}$, namely the component containing the reduction of~$z$. This implies that $\red(Y(\Q_\pell)_{\Sigma_\pell^0})\subseteq V(\Gamma_z)$, and hence that $\lSel_{\Sigma_\pell^0}\subseteq\jj_\Gamma(V(\Gamma_z))^\Zar=\lSel_{\Sigma_\pell^1}$, as desired. This completes the proof of Lemma~\ref{lem:local_images_in_s}.\qed
\smallskip

We will need one final property of the subschemes $\lSel_{\Sigma_\pell}$, namely a bound on their Hilbert series. For~$\Sigma_\pell\in\Cpts_\pell^0$ the Hilbert series is either~$0$ or~$1$ according as $\lSel_{\Sigma_\pell}$ is empty or a single $\Q_\pell$-point. For~$\Sigma_\pell\in\Cpts_\pell^1$, the computation is rather more complicated.

\begin{lemma}\label{lem:local_hilbert_series_bound_in_s}
	Suppose that $\Sigma_\pell^1\in\Cpts_\pell^1$. Then
	\[
	\HS_{\lSel_{\Sigma_\pell^1}}(t) \preceq (1-t^2)^{-1} \,.
	\]
	\begin{proof}
		Again, we may suppose without loss of generality that~$U=U^\et$ and that $Y(\Q_\pell)_{\Sigma_\pell^1}\neq\emptyset$. We write~$z_0$ as a shorthand for the cusp $\Sigma_\pell^1\in\Dvsr_{\min,0}(\F_\pell)$, and write~$z\in D(\Q_\pell)$ for the point reducing onto~$z_0$, as in Lemma~\ref{lem:non-empty_discs}. Let~$]z_0[\subseteq X^\an$ denote the residue disc of~$z_0$ in the model $\X_{\min}$, and let~$]z_0[^\times:=]z_0[\setminus\{z\}$. Choose a $\Q_\pell$-rational basepoint~$b'\in Y(\Q_\pell)_{\Sigma_\pell^1}=]z_0[^\times\cap Y(\Q_\pell)$ The functor from finite \'etale covers of~$Y_{\Qbar_\pell}$ to finite \'etale coverings of~$]z_0[^\times_{\C_\pell}$, given by $Y'\mapsto (Y'_{\C_\pell})^\an|_{]z_0[^\times_{\C_\pell}}$ then induces a morphism
		\[
		\Z_\ellp(1)\cong\pi_1^{\alg,\ellp}(]z_0[_{\C_\pell}^\times,b') \to \pi_1^{\et,\ellp}(Y_{\Qbar_\pell},b') \simeq \pi_1^{\et,\ellp}(Y_{\Qbar_\pell},b)
		\]
		on pro-$\ellp$ \'etale fundamental groups, where the left-hand isomorphism comes from \cite[Theorem~6.3.5]{berkovich:etale_cohomology}, and the right-hand isomorphism is conjugation by a pro-$\ellp$ \'etale path between~$b$ and~$b'$. The first fundamental group here is in the sense of \cite[p.~94]{de_jong:rigid_fundamental_groups}.
		
		The first two of the above maps are equivariant for the action of $G_\pell$, while the final isomorphism exhibits $\pi_1^{\et,\ellp}(Y_{\Qbar_\pell},b')$ as a Serre twist of $\pi_1^{\et,\ellp}(Y_{\Qbar_\pell},b)$, so the above sequence in particular induces a morphism
		\[
		\HH^1(G_\pell,\Q_\pell(1)) \to \HH^1(G_\pell,U^\et)
		\]
		on cohomology schemes. The left-hand cohomology scheme is isomorphic to~$\A^1_{\Q_\ellp}$ by Kummer theory.
		
		If now~$y\in Y(\Q_\pell)_{\Sigma_\pell^1}=]z_0[^\times\cap Y(\Q_\pell)$, then there is a map $\pi_1^{\alg,\ellp}(]z_0[^\times_{\C_\pell};b',y)\to\pi_1^{\et,\ellp}(Y_{\Qbar_\pell};b',y)$ on pro-$\ellp$ \'etale path-torsors, exhibiting $\pi_1^{\et,\ellp}(Y_{\Qbar_\pell};b',y)$ as the pushout of $\pi_1^{\alg,\ellp}(]z_0[^\times_{\C_\pell};b',y)$ along the map $\pi_1^{\alg,\ellp}(]z_0[^\times_{\C_\pell},b')\to\pi_1^{\et,\ellp}(Y_{\Qbar_\pell},b')$. This map on path-torsors is $G_\pell$-equivariant, so we see that there is a commuting square
		\begin{center}
		\begin{tikzcd}
			Y(\Q_\pell)_{\Sigma_\pell^1} \arrow[r,hook]\arrow[d,"\jj_{\pell,z_0}"] & Y(\Q_\pell) \arrow[d,"\jj_{\pell,U^\et}"] \\
			\HH^1(G_\pell,\Q_\ellp(1)) \arrow[r] & \HH^1(G_\pell,U^\et(\Q_\ellp)) \,,
		\end{tikzcd}
		\end{center}
		where~$\jj_{\pell,z_0}$ is the map sending~$y$ to the class of the $\Q_\ellp$-pro-unipotent path-torsor $\pi_1^{\alg,\Q_\ellp}(]z_0[^\times;b',y)$. As mentioned above, $\HH^1(G_\pell,\Q_\ellp(1))\cong\Q_\ellp$ is one-dimensional, and $\jj_{\pell,z_0}$ is given by $y\mapsto v_\pell(t(y)) - v_\pell(t(b'))$ for~$t$ an integral local parameter at~$z_0$. It follows that~$\lSel_{\Sigma_\pell^1}$ is the scheme-theoretic image of~$\HH^1(G_\pell,\Q_\ellp(1))\to\HH^1(G_\pell,U)$.
		
		Now the image of the map $\Q_\ellp(1)\to U$ is contained in~$\W_{-2}U$, by definition of the weight filtration. Hence this map is filtered if we endow $\Q_\pell(1)$ with the vector space filtration supported in degree~$-2$, as in Example~\ref{ex:filtered_vs}. Since the map $\HH^1(G_\pell,\Q_\ellp(1))\to\lSel_{\Sigma_\pell^1}$ is dominant, we have the inequality
		\[
		\HS_{\lSel_{\Sigma_\pell^1}}(t) \preceq \HS_{\HH^1(G_\pell,\Q_\ellp(1))}(t) = (1-t^2)^{-1}
		\]
		by Lemma~\ref{lem:hilbert_series_algebra} and Corollary~\ref{cor:abelian_representability}.
	\end{proof}
\end{lemma}

\subsection{Effective refined Chabauty--Kim}\label{ss:proof_general}

The partitions~\eqref{eq:local_decomposition_integral} and~\eqref{eq:local_decomposition_rational_naive} of the local points of~$Y$ give rise to a corresponding partition of the $S$-integral points. We write $\Cpts:=\prod_{\pell\neq\ellp}\Cpts_\pell$, which is a finite set since $\#\Cpts_\pell=1$ whenever $\pell\notin S\cup\{\ellp\}$ is a prime of good reduction for~$(\X,\Dvsr)$. For an element $\Sigma=(\Sigma_\pell)_{\pell\neq\ellp}\in\Cpts$, we write $\Y(\Z_S)_\Sigma\subseteq\Y(\Z_S)$ for the set of $S$-integral points whose mod-$\pell$ reduction lies in $\Sigma_\pell$ for all~$\pell\neq\ellp$, and say that $\Y(\Z_S)_\Sigma$ is the set of points of \emph{reduction type}~$\Sigma$. As a consequence of the local decompositions~\eqref{eq:local_decomposition_integral} and~\eqref{eq:local_decomposition_rational_naive}, these sets constitute a finite partition
\begin{equation}\label{eq:global_decomposition}
\Y(\Z_S) = \coprod_{\Sigma\in\Cpts}\Y(\Z_S)_\Sigma
\end{equation}
of the $S$-integral points.

For any~$\Sigma=(\Sigma_\pell)_{\pell\neq\ellp}\in\Cpts$, the closed subschemes $\lSel_{\Sigma_\pell}\subseteq\HH^1(G_\pell,U)$ constructed above form a Selmer structure in the sense of \S\ref{ss:global_selmer}, and we denote by $\Sel_{\Sigma,U}$ the corresponding Selmer scheme. By construction, $\Sel_{\Sigma,U}$ contains the image of $\Y(\Z_S)_\Sigma$ under the global non-abelian Kummer map~$\jj_U$, and so the Chabauty--Kim diagram~\eqref{diag:c-k} restricts to a commuting diagram
\begin{equation}\label{diag:c-k_refined}
	\begin{tikzcd}
		\Y(\Z_S)_\Sigma \arrow[r,hook]\arrow[d,"\jj_U"] & \Y(\Z_\ellp) \arrow[rd,"\jj_{\dR,U}"]\arrow[d,"\jj_{\ellp,U}"] & \\
		\Sel_{\Sigma,U}(\Q_\ellp) \arrow[r,"\loc_\ellp"] & \HH^1_f(G_\ellp,U)(\Q_\ellp) \arrow[r,"\log_\BK","\sim"'] & \Fil^0\backslash\D_\dR(U)(\Q_\ellp) \,,
	\end{tikzcd}
\end{equation}
where again the objects appearing on the bottom row are the $\Q_\ellp$-points of filtered affine $\Q_\ellp$-schemes, and the maps $\loc_\ellp$ and $\log_\BK$ are a morphism and isomorphism thereof. Following through the same argument as in Theorem~\ref{thm:main_naive} gives the following more refined effective Chabauty--Kim theorem. In the case that~$Y=X$ is projective, $\X$ is the minimal regular model of~$X$, and $S=\emptyset$, this specialises to Theorems~\ref{thm:main}, \ref{thm:main_bound} and~\ref{thm:main_det} of the introduction.

\begin{theorem}\label{thm:main_refined}
	Let $s:=\#S$, and let $(c_i^\glob)_{i\geq0}$ and $(c_i^\loc)_{i\geq0}$ be the coefficients of the power series
	\begin{align*}
		\HS_\glob(t) &:= (1-t^2)^{-s}\cdot\prod_{n\geq1}^\infty(1-t^n)^{-\dim\HH^1_f(G_\Q,V_n)} \\
		\HS_\loc(t) &:= \prod_{n\geq1}^\infty(1-t^n)^{-\dim\HH^1_f(G_p,V_n)} \,.
	\end{align*}
	Suppose that~$m$ is a positive integer such that the inequality
	\begin{equation}\label{eq:c-k_ineq_series_refined}
		\sum_{i=0}^mc_i^\glob < \sum_{i=0}^mc_i^\loc
	\end{equation}
	holds. Then:
	\begin{enumerate}[label=\Alph*), ref=(\Alph*), font=\normalfont]
		\item\label{thmpart:main_refined} For all reduction types~$\Sigma\in\Cpts$, the set $\Y(\Z_S)_\Sigma$ is contained in the vanishing locus of a Coleman algebraic function of weight at most~$m$ which is associated to~$U$.
		\item\label{thmpart:main_refined_bound} We have
		\[
		\#\Y(\Z_S)\leq\kappa_p\cdot\prod_{\ell\in S}(n_\ell+r)\cdot\prod_{\ell\notin S}n_\ell\cdot\#\Y(\F_p)\cdot(4g+2r-2)^m\cdot\prod_{i=1}^{m-1}(c_i^\loc+1) \,,
		\]
		where~$n_\pell$ denotes the number of irreducible components of the mod-$\pell$ special fibre of~$\X$, $\kappa_\ellp=1+\frac{\ellp-1}{(\ellp-2)\log(\ellp)}$ if $\ellp$ is odd, and $\kappa_\ellp=2+\frac2{\log(2)}$ if~$\ellp=2$.
		\item\label{thmpart:main_refined_det} If $C_m^\glob:=\sum_{i=0}^mc_i^\glob$ and $C_m^\loc:=\sum_{i=0}^mc_i^\loc$, and if $f_1,\dots,f_{C_m^\loc}$ is a basis of the space of Coleman algebraic functions of weight at most~$m$ associated to~$U$, then for every $C_m^\glob+1$-tuple of points $x_0,x_1,\dots,x_{C_m^\glob}\in\Y(\Z_S)$ of the same reduction type, all $(C_m^\glob+1)\times(C_m^\glob+1)$ minors of the matrix~$M$ with entries $M_{ij}=f_i(x_j)$ vanish.
	\end{enumerate}
	\begin{proof}
		The proofs of~\ref{thmpart:main_refined} and~\ref{thmpart:main_refined_det} follow exactly as the corresponding parts of Theorem~\ref{thm:main_naive} once we have the bound $\HS_{\Sel_{\Sigma,U}}(t)\preceq\HS_\glob(t)$ on the Hilbert series of the refined Selmer scheme, which follows by combining Lemma~\ref{lem:hilbert_series_bound_global} with Lemmas~\ref{lem:local_hilbert_series_bound_in_s} and~\ref{lem:local_images_outside_s}. Using part~\ref{thmpart:main_refined}, Theorem~\ref{thm:weight_bound} gives the bound
		\[
		\#\Y(\Z_S)_\Sigma\leq\kappa_p\cdot\#\Y(\F_p)\cdot(4g+2r-2)^m\cdot\prod_{i=1}^{m-1}(c_i^\loc+1)
		\]
		on the size of each $\Y(\Z_S)_\Sigma$. Part~\ref{thmpart:main_refined_bound} follows by summing over all~$\Sigma$.
	\end{proof}
\end{theorem}

\begin{remark}\label{rmk:slightly_better_bound}
	The constants $n_\pell$ and $n_\pell+r$ appearing in Theorem~\ref{thm:main_refined}\ref{thmpart:main_refined_bound} are not always optimal. It can happen that some of the local Selmer schemes $\lSel_{\Sigma_\pell}$ associated to different reduction types~$\Sigma_\pell$ coincide (or one is contained in the other in the case $\pell\in S$). Then if~$\Sigma,\Sigma'\in\Cpts$ are such that $\lSel_{\Sigma_\pell}\subseteq\lSel_{\Sigma'_\pell}$ for all~$\pell\neq\ellp$, we have that $\Sel_{\Sigma,U}\subseteq\Sel_{\Sigma',U}$. It follows that, in the setup of Theorem~\ref{thm:main_refined}, there is a single non-zero Coleman algebraic function~$f$ of weight~$\leq m$ which vanishes on $\Y(\Z_S)_\Sigma\cup\Y(\Z_S)_{\Sigma'}$. By bounding the size of such unions using Theorem~\ref{thm:weight_bound}, one can improve the bound in Theorem~\ref{thm:main_refined}\ref{thmpart:main_refined_bound} by replacing the constants $n_\pell$ by the size of the finite set $\jj_{\pell,U}(\Y(\Z_\pell))$ for $\pell\notin S\cup\{\ellp\}$, and the constants $n_\pell+r$ by the number of irreducible components of the Zariski-closure of $\jj_{\pell,U}(Y(\Q_\pell))$ for $\pell\in S$.
\end{remark}

\begin{remark}
	The inequality~\eqref{eq:c-k_ineq_series_refined} depends on $s:=\#S$ and not on the set~$S$ itself. Moreover, if the regular model~$\X$ is defined over~$\Z$, then it is easy to see that the quantity $\prod_{\pell\in S}(n_\pell+r)\cdot\prod_{\pell\notin S}n_\pell$ appearing in part~\ref{thmpart:main_refined_bound} of Theorem~\ref{thm:main_refined} is bounded in terms of~$s$ only. Hence when~$\X$ is defined over~$\Z$ and when inequality~\eqref{eq:c-k_ineq_series_refined} holds for some value of~$s$, then Theorem~\ref{thm:main_refined}\ref{thmpart:main_refined_bound} gives a uniform upper bound on $\#\Y(\Z_S)$ for all finite sets~$S$ of primes of size~$s$, not containing~$\ellp$. In other words, the bounds obtained by Theorem~\ref{thm:main_refined}\ref{thmpart:main_refined_bound} are automatically $S$-uniform, at least for sets~$S$ not containing~$\ellp$.
\end{remark}

\begin{remark}
	In general, Theorem~\ref{thm:main_refined} affords better control of $S$-integral points than Theorem~\ref{thm:main_naive}, in that the weight~$m$ for which Theorem~\ref{thm:main_refined} holds is in general less than the corresponding weight in Theorem~\ref{thm:main_naive}, and the dominant term in the bounds is usually the part depending on~$m$. In fact, it can happen that the inequality in Theorem~\ref{thm:main_refined} holds for some~$m$, while the inequality in Theorem~\ref{thm:main_naive} holds for no~$m$ at all, so Theorem~\ref{thm:main_naive} tells us nothing about $S$-integral points.
\end{remark}

\subsubsection{Example: quotients of~$U^\et$ by the weight filtration}\label{ss:depth_quotients}

In the particular case that the quotient~$U$ is chosen finite-dimensional, then the power series $\HS_\glob(t)$ and $\HS_\loc(t)$ appearing in Theorem~\ref{thm:main_refined} are rational functions in~$t$ without poles inside the unit disc. The order of their poles at~$1$ are $\#S+\sum_{n>0}\dim_{\Q_\ellp}\HH^1_f(G_\Q,V_n)$ and $\sum_{n>0}\dim_{\Q_\ellp}\HH^1_f(G_\ellp,V_n)$, respectively, where~$(V_n)_{n>0}$ denote the graded pieces of the weight filtration on~$U$. In particular, when the inequality
\begin{equation}\label{eq:c-k_ineq_refined}
\#S+\sum_{n>0}\dim_{\Q_\ellp}\HH^1_f(G_\Q,V_n) < \sum_{n>0}\dim_{\Q_\ellp}\HH^1_f(G_\ellp,V_n)
\end{equation}
holds, then inequality~\eqref{eq:c-k_ineq_series_refined} must hold for~$m\gg0$, and so we may apply Theorem~\ref{thm:main_refined} to explicitly bound the size of $\Y(\Z_S)$. In other words, Theorem~\ref{thm:main_refined} applies in every case that refined Chabauty--Kim applies.

In fact, the value of~$m$ required can be controlled rather explicitly in terms of the weights of~$U$ and the right-hand side of~\eqref{eq:c-k_ineq_refined}. Suppose that the weight filtration on~$U$ is supported in degrees~$\geq-n$. If we write $d_\glob=\#S+\sum_{i>0}\dim\HH^1_f(G_\Q,V_n)$ and $d_\loc=\sum_{i>0}\dim\HH^1_f(G_\ellp,V_n)$, then the Hilbert series $\HS_\glob(t)$ and $\HS_\loc(t)$ from Theorem~\ref{thm:main_refined} are bounded by
\[
\HS_\glob(t) \preceq (1-t)^{-d_\glob} \hspace{0.4cm}\text{and}\hspace{0.4cm} \HS_\loc(t) \succeq (1-t^n)^{d_\loc} \,,
\]
respectively. This implies via the binomial theorem that
\[
\sum_{i=0}^mc_i^\glob \leq {{m+d_\glob}\choose{d_\glob}} \hspace{0.4cm}\text{and}\hspace{0.4cm} \sum_{i=0}^mc_i^\loc \geq {{\lfloor m/n\rfloor+d_\loc}\choose{d_\loc}}
\]
for all~$m$. When~\eqref{eq:c-k_ineq_refined} holds (i.e.\ $d_\glob<d_\loc$), this gives the upper bound
\[
\sum_{i=0}^mc_i^\glob \leq n^{d_\glob}{{m/n+d_\glob}\choose{d_\glob}} < \frac{n^{d_\glob}\cdot d_\loc!}{(m/n)^{d_\loc-d_\glob}\cdot d_\glob!}{{m/n+d_\loc}\choose{d_\loc}}
\]
and hence inequality~\eqref{eq:c-k_ineq_series_refined} holds for~$m=d_\loc\cdot n^{d_\loc}$. So Theorem~\ref{thm:main_refined}\ref{thmpart:main_refined_bound} gives the bound
\begin{equation}\label{eq:coarse_bound}\tag{$\ast$}
\#\Y(\Z_S)\leq\kappa_\ellp\cdot\prod_\pell n_\pell\cdot\#\Y(\F_\ellp)\cdot(4g+2r-2)^{d_\loc n^{d_\loc}}\cdot\prod_{i=1}^{d_\loc n^{d_\loc}-1}(1+c_i^\loc) \,.
\end{equation}

In the particular case that~$Y=X$ is projective and $S=\emptyset$, this gives the coarse upper bound on $\#X(\Q)$ claimed in the Corollary to Theorem~\ref{thm:main_bound}. To extract the precise form of the bound there, one observes that the coefficients~$c_i^\loc$ are bounded above by the coefficients~$c_i$ of the power series~$\frac{1-gt}{1-2gt+t^2}$ by Corollary~\ref{cor:coleman_hilbert_series}, and these latter coefficients satisfy $c_0=1$, $c_1=g$ and $c_i=2gc_{i-1}-c_{i-2}$ for~$i\geq2$. An easy induction gives the bound $c_i^\loc\leq\frac12(2g)^i$ for~$i\geq1$, hence also the bound $d_\loc\leq\frac g{2g-1}(2g)^n$. Substituting these values into~\eqref{eq:coarse_bound} gives the bound claimed in the Corollary.

\subsection{The Balakrishnan--Dogra trick}\label{ss:jen-netan}

Calculating the dimensions of the global Bloch--Kato Selmer groups in order to compute the power series $\HS_\glob(t)$ from Theorem~\ref{thm:main_refined} can be quite subtle, and in general depends on deep conjectures in number theory. One already sees this subtlety on the abelian level: the quotient $U^\et/\W_{-2}=\HH^1_\et(X_{\Qbar},\Q_\ellp)^\dual$ is the $\Q_\ellp$-linear Tate module $V_\ellp J$ of the Jacobian~$J$ of~$X$, so the dimension of $\HH^1_f(G_\Q,U^\et/\W_{-2})$ is the $\ellp^\infty$-Selmer rank of~$J$. According to the Tate--Shafarevich Conjecture, this should be equal to the rank of~$J(\Q)$, but this is still far from known in general.

It was observed in \cite[Remark~2.3]{jen-netan:quadratic_i} that one can sidestep the need to assume the Tate--Shafarevich Conjecture by a suitable modification of the definition of the Selmer scheme and Chabauty--Kim locus. We outline here -- in sketch form -- how to adapt our above theory to the variant of the Selmer scheme from \cite[Definition~2.2]{jen-netan:quadratic_i}.

Let~$U$ be a $G_\Q$-equivariant quotient of the $\Q_\ellp$-pro-unipotent \'etale fundamental group of the hyperbolic curve~$Y/\Q$. The quotient~$V_1=U/\W_{-2}U$ is a quotient of the abelianised $\Q_\ellp$-pro-unipotent \'etale fundamental group of~$X$, which is the $\Q_\ellp$-linear Tate module $V_\ellp J$ of the Jacobian~$J$ of~$X$. Hence the Kummer map for the abelian variety~$J$ provides a map
\[
J(\Q) \to \HH^1_f(G_\Q,V_\ellp J) \to \HH^1_f(G_\Q,V_1) \,.
\]
We write $\HH^1_{f,\BD}(G_\Q,V_1)$ for the image of this map, which is a $\Q_\ellp$-vector space of dimension $d_1\leq\rk(J(\Q))$, and conflate $\HH^1_{f,\BD}(G_\Q,V_1)$ with its associated vector group as usual.

For a reduction type~$\Sigma\in\Cpts$, the \emph{Balakrishnan--Dogra Selmer scheme} $\Sel_{\Sigma,U}^\BD$ is defined to be the preimage of the closed subscheme $\HH^1_{f,\BD}(G_\Q,V_1)\subseteq\HH^1_f(G_\Q,V_1)$ under the map $\Sel_{\Sigma,U}\to\HH^1_f(G_\Q,V_1)$ induced from the map $U\twoheadrightarrow V_1$. This is a closed subscheme of the Selmer scheme~$\Sel_{\Sigma,U}$ studied in the previous section, and an inductive argument along the lines of the proof of Lemma~\ref{lem:hilbert_series_bound_global} shows that its Hilbert series is bounded by
\[
\HS_{\Sel_{\Sigma,U}^\BD}(t) \preceq \HS_\glob^\BD(t):=(1-t^2)^{-s}\cdot(1-t)^{-d_1}\cdot\prod_{n=2}^\infty(1-t^n)^{-\dim\HH^1_f(G_\Q,V_n)} \,.
\]

The Balakrishnan--Dogra Selmer scheme $\Sel_{\Sigma,U}^\BD$ still contains the image of $\Y(\Z_S)_\Sigma$ under the non-abelian Kummer map, and by using this in place of $\Sel_{\Sigma,U}$, we see that the statement of Theorem~\ref{thm:main_refined} still holds if the power series $\HS_\glob(t)$ is replaced by $\HS_\glob^\BD(t)$. Using this variant of Theorem~\ref{thm:main_refined} makes all the bounds in the examples in \S\ref{ss:examples} unconditional on the Tate--Shafarevich Conjecture.

%% file: siegel.tex
\section{$S$-uniformity in Siegel's Theorem}\label{s:siegel}

As an application and illustration of our theory, we prove our $S$-uniform Siegel's Theorem~\ref{thm:main_siegel}. Throughout this section, let $\Y=\P^1_\Z\setminus\{0,1,\infty\}$, $\ellp$ be an odd prime, and $S$ a variable non-empty finite set of primes not containing $\ellp$, of size $s=\#S$.

For the quotient~$U$ we take the whole $\Q_\ellp$-pro-unipotent \'etale fundamental group of $\Y_{\overline\Q}=\P^1_{\overline\Q}\setminus\{0,1,\infty\}$ (based at some basepoint $b\in\Y(\Z_{(\ellp)})$). The group~$U$ is the free pro-unipotent group on two generators in filtration-degree $-2$, and hence its graded pieces are
\[
\gr^\W_{-2n}U \simeq \Q_\ellp(n)^{\oplus M(2,n)}
\]
where $M(2,n)=\frac1n\sum_{d\mid n}\mu(n/d)2^i$ is the $n$th Moreau necklace polynomial, evaluated at $2$ \cite[Theorem~6]{reutenauer:friegebras}. Since $\dim_{\Q_\ellp}\HH^1_f(G_\ellp,\Q_\ellp(n))=1$ for all $n>0$ by \cite[Example~3.9]{bloch-kato:tamagawa_numbers}, we find via the cyclotomic identity \cite{metropolis-rota:cyclotomic} that the local Hilbert series $\HS_\loc(t)$ is given by
\[
\HS_\loc(t) = \prod_{n>0}(1-t^{2n})^{-M(2,n)} = \frac1{1-2t^2}\,.
\]
Note that this is the same as in Corollary~\ref{cor:coleman_hilbert_series}, and reflects the fact that the affine ring of~$\HH^1_f(G_\ellp,U)=\D_\dR(U)$ is the shuffle algebra in~$2$ variables. From now on, we will write all our power series in the variable~$\tau:=t^2$, which we may do since all the filtrations we will see are supported in even degrees.

As for the global Hilbert series $\HS_\glob(t)$, Soul\'e's vanishing theorem \cite{soule:vanishing} implies that $\dim\HH^1_f(G_\Q,\Q_\ellp(n))=0$ if $n=1$ or if $n\geq2$ is even, and $\dim\HH^1_f(G_\Q,\Q_\ellp(n))\leq1$ if $n\geq3$ is odd. This gives the coefficientwise inequality
\[
\HS_\glob(t) \leq (1-\z)^{-s}\cdot\prod_{n\geq3\text{ odd}}(1-\z^{2n})^{-M(2,n)}
\]
for the global Hilbert series. Although~$\HS_\glob(t)$ does not have a simple closed form expression similar to~$\HS_\loc(t)$, it does satisfy a functional equation.

\begin{lemma}\label{lem:global_functional_equation}
	The power series
	\[
	F(\z) := \prod_{n>0\text{ odd}}(1-\z^n)^{-M(2,n)}
	\]
	satisfies the functional equation
	\[
	F(\z)^2=\frac{1+2\z}{1-2\z}\cdot F(\z^2)\,.
	\]
	\begin{proof}
		By the cyclotomic identity, we have
		\[
		\frac1{1-2\z} = F(\z)\cdot\prod_{n>0}(1-\z^{2n})^{-M(2,2n)}\,.
		\]
		But we have $M(2,2n)=\frac1{2n}\sum_{d\mid n}\mu(n/d)4^d=\frac12M(4,n)$ if $n$ is even, and have $M(2,2n)=\frac1{2n}\sum_{d\mid n}\mu(n/d)4^d - \frac1{2n}\sum_{d\mid n}\mu(n/d)2^d=\frac12M(4,n)-\frac12M(2,n)$ if $i$ is odd. Substituting this into the above equation yields
		\begin{align*}
			\frac1{1-2\z} &= F(\z)\cdot\prod_{n>0}(1-\z^{2n})^{-\frac12M(4,n)}\cdot\prod_{n>0\text{ odd}}(1-\z^{2n})^{\frac12M(2,n)} \\
			&= F(\z)\cdot\left(\frac1{1-4\z^2}\right)^{1/2}\cdot F(\z^2)^{-1/2}\,.
		\end{align*}
		Squaring both sides and rearranging gives the desired functional equation.
	\end{proof}
\end{lemma}


\begin{remark}\label{rmk:computing_global_power_series}
	Lemma~\ref{lem:global_functional_equation} implies that the coefficients~$a_i$ of~$F(\z)$ satisfy the equation
	\[
	\sum_{i=0}^ma_ia_{m-i} = \sum_{i=0}^{\lfloor m/2\rfloor}b_{m-2i}a_i
	\]
	for every~$m\geq0$, where $b_0=1$ and $b_i=2^{i+1}$ for~$i>0$. Since~$a_0=1$, this in particular expresses $a_m$ in terms of the coefficients~$a_i$ for $i<m$. This gives an efficient recursive algorithm to compute the coefficients of~$F(\z)$ without first computing the values of~$M(2,n)$. Using this method, it is feasible to compute the first few thousand coefficients of~$F(\z)$ on a desktop computer.
\end{remark}

Now in order to apply Theorem~\ref{thm:main_refined}, we need to find some~$m$ for which the $t^m$ coefficient of~$\frac1{1-t}\HS_\glob(t)$ is strictly less than the corresponding coefficient of~$\frac1{1-t}\HS_\loc(t)$. This is provided by the following calculation.

\begin{lemma}\label{lem:local_hilbert_series_overtakes_global}
There is some $m\leq 4^s$ such that the $t^{2m}$ coefficient of $\frac1{1-t}\HS_\glob(t)$ is strictly less than the $t^{2m}$ coefficient of $\frac1{1-t}\HS_\loc(t)$.
\begin{proof}
This is easily verified in the case $S=\emptyset$; we deal with the case $S\neq\emptyset$. We will show that the $4^s$th coefficient of $\left(\frac1{1-\z}\HS_\glob(t)\right)^2$ is strictly less than the $4^s$th coefficient of $\left(\frac1{1-\z}\HS_\loc(t)\right)^2$; this yields the desired result by expanding both squares.

Now on the one hand, $\frac1{1-\z}\HS_\glob(t)\leq\frac1{(1-\z)^{s-1}}\cdot F(\z)$, where $F(\z)$ is as in Lemma~\ref{lem:global_functional_equation}. We thus have coefficientwise inequalities
\begin{align*}
\left(\frac1{1-\z}\HS_\glob(t)\right)^2 &\leq \frac1{(1-\z)^{2s-2}}\cdot\frac{1+2\z}{1-2\z}\cdot F(\z^2) \\
 &\leq \frac1{(1-\z)^{2s-2}}\cdot\frac{1+2\z}{1-2\z}\cdot\frac1{1-2\z^2} \\
 &\leq 2^{2s}\cdot\frac1{1-2\z}\,,
\end{align*}
using, in order, Lemma~\ref{lem:global_functional_equation}, the identity $\frac1{1-2\z}=F(\z)\cdot\prod_{n>0}(1-\z^{2n})^{-M(2,2n)}\geq F(\z)$, and the inequalities $\frac{1+2\z}{1-2\z}\leq\frac2{1-2\z}$, $\frac1{1-\z}\cdot\frac1{1-2\z}\leq\frac2{1-2\z}$ and $\frac1{1-2\z^2}\cdot\frac1{1-2\z}\leq\frac2{1-2\z}$ (geometric series). Thus, we see that the $4^s$th coefficient of $\left(\frac1{1-\z}\HS_\glob(t)\right)^2$ is at most $4^s\cdot2^{4^s}$.

On the other hand, $\frac1{1-\z}\HS_\loc(t)=\frac1{1-\z}\cdot\frac1{1-2\z}\geq\frac1{1-2\z}$, and so
\[
\left(\frac1{1-\z}\HS_\loc(\z)\right)^2 \geq \frac1{(1-2\z)^2} = \sum_{i\geq0}(i+1)2^i\z^i\,.
\]
Thus the $4^s$th coefficient of $\left(\frac1{1-\z}\HS_\loc(t)\right)^2$ is at least $(4^s+1)\cdot2^{4^s}$, which is strictly greater than the corresponding coefficient of $\left(\frac1{1-\z}\HS_\glob(t)\right)^2$, as desired.
\end{proof}
\end{lemma}

\begin{remark}
The upper bound on $m$ in Lemma~\ref{lem:local_hilbert_series_overtakes_global} appears to be relatively close to optimal. For $s=0,1,2,\dots,9$, the minimal values of $m$ for which $\z^m$ coefficient of $\frac1{(1-\z)^{s-1}}\cdot F(\z)$ is strictly less than the $\z^m$ coefficient of $\frac1{1-\z}\cdot\frac1{1-2\z}$ are $m=1,1,2,9,24,81,308,1212,4827,19284$. These first few values grow roughly exponentially, with exponent~$4$, suggesting that Lemma~\ref{lem:local_hilbert_series_overtakes_global} captures something close to the correct growth rate. These values of~$m$ were found with the assistance of Steven Charlton, by computing the first $20000$ coefficients of $F(\z)$ using the recursive algorithm described in Remark~\ref{rmk:computing_global_power_series}.
\end{remark}

We can then plug the bound on~$m$ from Lemma~\ref{lem:local_hilbert_series_overtakes_global} into Theorem~\ref{thm:main_refined} to obtain a bound on the size of the set~$\Y(\Z_S)$. The local indices are $n_\pell=1$ for all~$\pell$, since~$\Y$ has good reduction everywhere. Hence Theorem~\ref{thm:main_refined}\ref{thmpart:main_refined_bound} gives the bound
\[
\#\Y(\Z_S) \leq \kappa_\ellp\cdot4^s\cdot(\ellp-2)\cdot4^{2^{2s+1}}\cdot\prod_{i=1}^{4^s-1}(2^i+1) \,.
\]

In this particular case, we can actually do considerably better than this bound. For any prime $\pell\in S$, the Zariski-closure of $\jj_{\pell,U^\et}(Y(\Q_\pell))$ has $\leq3$ irreducible components, corresponding to the three cusps $0,1,\infty$, which all meet at the point corresponding to the unique irreducible component of the special fibre of~$\Y$ by Lemma~\ref{lem:local_images_in_s}\ref{lempart:local_images_in_s_compatibility}. Hence, as in Remark~\ref{rmk:slightly_better_bound}, we see that $\Y(\Z_S)$ is the union of $3^s$ subsets, each of which is contained in the vanishing locus of a non-zero Coleman algebraic function~$f$ of weight at most $2\cdot4^s$. Theorem~\ref{thm:weight_bound} provides an upper bound of $\kappa_\ellp\cdot(\ellp-2)\cdot4^{2^{2s+1}}\cdot\prod_{i=1}^{4^s-1}(2^i+1)$ on the number of zeroes of such an~$f$, but it turns out that one can give a better bound in this particular case.

\begin{lemma}\label{lem:weight_bound_for_the_line}
	Let~$f$ be a non-zero Coleman algebraic function on $\P^1_{\Z_\ellp}\setminus\{0,1,\infty\}$ of weight at most~$2m$. Then the number of $\Z_\ellp$-integral zeroes of~$f$ is at most~$\kappa_\ellp\cdot(\ellp-2)\cdot2^{m+1}$.
	\begin{proof}
		We follow a similar strategy to the proof of Theorem~\ref{thm:weight_bound}, except that we can be more explicit about the differential operators involved. For~$m\geq0$ let~$\DO_m$ denote the differential operator
		\[
		\DO_m:= (z-z^2)^{2^m}\frac{\d^{2^m}}{\d z^{2^m}}\cdot (z-z^2)^{2^{m-1}}\frac{\d^{2^{m-1}}}{\d z^{2^{m-1}}} \cdot\ldots\cdot (z-z^2)\frac{\d}{\d z} \,,
		\]
		where~$z$ is the standard coordinate on $\Y=\P^1_\Z\setminus\{0,1,\infty\}$.
		
		Now we claim that~$\DO_m$ vanishes on all Coleman algebraic functions of weight at most~$2m$. We do this by induction on~$m$. In the case~$m=0$, this is clear since all Coleman algebraic functions of weight at most~$0$ are constant. Thus suppose that~$\DO_m$ vanishes on all Coleman algebraic functions of weight at most~$2m$. Since the filtration on the ring of Coleman algebraic functions on~$\P^1_{\Z_\ellp}\setminus\{0,1,\infty\}$ is supported in even degrees, it also vanishes on all Coleman algebraic functions of weight at most~$2m+1$.
		
		Now the divisor of~$\d z$, \emph{as a section of the log-canonical bundle $\Omega^1_{\P^1_{\Q_\ellp}}([0]+[1]+[\infty])$}, is $[0]+[1]-[\infty]$. It follows that the divisor of~$\DO_m$ satisfies
		\[
		\div(\DO_m) \geq \sum_{i=0}^m\div\left((z-z^2)^{2^i}\frac{\d^{2^i}}{\d z^{2^i}}\right) = (1-2^{m+1})\cdot[\infty] \,.
		\]
		Hence by Lemma~\ref{lem:operator_gives_rational_fns} we see that for every Coleman algebraic function~$f$ of weight at most~$2m+2$, we have that $\DO_m(f)$ is a rational function lying in $\HH^0(\P^1_{\Q_\ellp},\O_{\P^1_{\Q_\ellp}}((2^{m+1}-1)\cdot[\infty]))$. In other words, $\DO_m(f)$ is a polynomial in~$z$ of degree at most~$2^{m+1}-1$. Since the differential operator $(z-z^2)^{2^{m+1}}\frac{\d^{2^{m+1}}}{\d z^{2^{m+1}}}$ vanishes on all such polynomials, it follows that $\DO_{m+1}=(z-z^2)^{2^{m+1}}\frac{\d^{2^{m+1}}}{\d z^{2^{m+1}}}\circ\DO_m$ vanishes on all Coleman algebraic functions of weight at most~$2m+2$. This completes the induction.
		\smallskip
		
		To conclude, we observe that the power series expansion of~$\DO_m$ in the variable~$t=z-a$ is a PD-nice differential operator for every~$a\in\Z_\ellp$ not in the residue disc of~$0$ or~$1$. Since~$\DO_m(f)=0$, the bound on the number of $\Z_\ellp$-integral zeroes of~$f$ then follows from Proposition~\ref{prop:nice}.
	\end{proof}
\end{lemma}

\begin{remark}
	The Coleman algebraic functions on~$\P^1_{\Z_\ellp}\setminus\{0,1,\infty\}$ of weight at most~$2m$ are exactly the linear combinations of $\ellp$-adic multiple polylogarithms of weight at most~$m$, so Lemma~\ref{lem:weight_bound_for_the_line} provides an upper bound on the number of zeroes of any such linear combination. Note that there is a small discrepancy in the usage of the word ``weight'' between our usage and what is standard in the theory of polylogarithms.
\end{remark}

Using this stronger bound, we can now complete the proof of Theorem~\ref{thm:main_siegel}. If~$2\notin S$, then $\Y(\Z_S)=\emptyset$ and so we are certainly done, so we assume henceforth that $2\in S$. We take~$\ellp$ the smallest prime not in~$S$. By Lemma~\ref{lem:local_images_in_s}\eqref{lempart:local_images_in_s_compatibility}, for every~$\pell\in S$, the Zariski-closure of $\jj_{\pell,U}(Y(\Q_\pell))$ has at most three irreducible components, which are genus~$0$ curves corresponding to the cusps~$0,1,\infty$ and which intersect in the point corresponding to the unique irreducible component of the $\pell$-adic special fibre of~$\Y$. Hence, as in Remark~\ref{rmk:slightly_better_bound}, we see that $\Y(\Z_S)$ is the union of~$\leq3^s$ subsets, each of which is contained in the vanishing locus of a Coleman algebraic function~$f$ of weight at most $2\cdot4^s$. Applying the bound from Lemma~\ref{lem:much_better_bound}, we find that
\[
\#\Y(\Z_S)\leq\kappa_\ellp\cdot3^s\cdot(\ellp-2)\cdot2^{4^s+1} \,.
\]
Since~$\ellp$ was chosen to be the smallest prime outside~$S$, we know that $\ellp\leq2^{s+1}$, so that $\kappa_\ellp\cdot(\ellp-2)=\ellp-2+\frac{\ellp-1}{\log(\ellp)}<2^{s+2}$ and we obtain the bound
\[
\#\Y(\Z_S)\leq 8\cdot6^s\cdot2^{4^s}
\]
as desired.\qed